\documentclass[11pt, oneside]{amsart}   	
\usepackage{geometry}                		
\usepackage{graphicx}				
\usepackage{enumitem}								
\usepackage{amssymb}
\usepackage[mathscr]{eucal}
\usepackage{tikz}
\usepackage{alltt}
\usetikzlibrary{shapes}
\usetikzlibrary{plotmarks}

\newtheorem{theorem}{Theorem}[section]

\newtheorem{corollary}[theorem]{Corollary}   

\newtheorem{lemma}[theorem]{Lemma}   

\newtheorem{proposition}[theorem]{Proposition}

\newtheorem{definition}[theorem]{Definition}

\newtheorem{remark}[theorem]{Remark}   

\newcommand{\R}{\mathbb{R}}
\newcommand{\C}{\mathscr{C}}

\renewcommand{\P}{\mathcal{P}}
\renewcommand{\S}{\mathcal{S}}
\newcommand{\U}{\mathcal{U}}
\newcommand{\V}{\mathcal{V}}
\newcommand{\W}{\mathcal{W}}
\newcommand{\e}{\varepsilon}
\newcommand{\Area}{\mathrm{Area}}
\newcommand{\Vol}{\mathrm{Vol}}
\newcommand{\diam}{\mathrm{diam}}

\newcommand{\geucl}{g_{\mathbb{E}}}
\newcommand{\BL}{_{\mathrm{BL}}}
\newcommand{\RN}{_{\mathrm{RN}}}

\newcommand{\grad}{\overrightarrow{\mathrm{grad}}}
\newcommand{\dvol}{\mathrm{dvol}_{\geucl}}

\newcommand{\omegam}{\omega_{\mathrm{mass}}}
\newcommand{\omegae}{\omega_{\mathrm{el}}}

\title{Brill-Lindquist-Riemann sums and their limits}

\author{Tatyana Benko}
\address{University of Oregon}
\email{tbenko@uoregon.edu}

\author{Iva Stavrov Allen}
\address{Lewis \& Clark College}
\email{istavrov@lclark.edu}

\begin{document}

\begin{abstract} 
This article commences a study of convergence of discretized point-object configurations, which we call Brill-Lindquist-Riemann sums, towards a charged dust continuum from the perspective of relativistic initial data. We are motivated by the interpretation of Brill-Lindquist-Riemann sums as collections of relatively isolated astrophysical bodies such as stars and galaxies in the universe, and the interpretation of the dust continuum as the universe itself. Our work begins by establishing the existence and the uniqueness of horizons/minimal surfaces of Brill-Lindquist metrics in the vicinity of the point-sources (``stars"). We then study the geometries of the regions exterior to said minimal surfaces, and discuss their Gromov-Hausdorff and intrinsic flat limit. An interesting and purely geometric byproduct of our work are examples in which the scalar curvature jumps upon taking Gromov-Hausdorff and /or intrinsic flat limits.
\end{abstract}

\maketitle

\section{Introduction}\label{geometry-sec}
\subsection{Motivation}
Real life objects are often understood as being made out of a great number of smaller constituents whose cumulative behavior manifests itself in the behavior of the object as a whole. Keeping track of attributes of each individual constituent is, due to their great number, at least impractical if not impossible. To remedy this situation we replace the collection of great many individual small constituents by the concept of continuum. The idea here is to introduce averaged attributes which, in an infinitesimal form, are attached to idealized point-object constituents. One classic example is the concept of mass density, i.e. mass per unit volume, $\rho(x)$. We pretend that there is some idealized point-object at location $x$ with mass $\rho(x)\,\mathrm{dvol}_{x}$ and that the total mass of the greater object is simply the ``sum" of the masses of all the point-objects: $\int \rho(x)\,\mathrm{dvol}_{x}$. Another example would be momentum density. 

Studying the behavior of a fluid is best done by studying the behavior of the densitized quantities, and by ignoring the exact nature of the individual constituents. For example, one studies a body of water by studying the continuity and the Euler equations satisfied by the mass and momentum densities and not by addressing the details of the behavior of the individual molecules. This particular line of reasoning is in many ways a cornerstone of cosmology, where universe as a whole is commonly treated as a perfect fluid. In some very loose sense of the word what molecules are to a body of water, stars are to galaxies and better yet: galaxies are to the universe. 

When we study a compact object from a great distance, such as a star from the standpoint of a galaxy or a galaxy from the standpoint of the universe, the length scale of the object itself is very miniscule. A very natural simplification here would be to take the limit as the length scale of the small body approaches zero. The implementation of such an  approach in theories where the governing equations are linear is relatively straightforward and somewhat standard because one is able to successfully use Dirac delta functions to describe bodies with zero length scale. However, the work of \cite{NI} clearly shows that playing with the length scale in such a manner within the relativistic context is a very delicate business and can be quite subtle. The culprit here is the interaction energy and the lack of additivity of mass-energy. (See also Section \ref{BLintro} below.)

In place of taking the limit of the length scale of a body to zero we treat any deviations, inhomogeneities or lack of symmetry within a small body as negligible and employ a mathematically simple (if ``unrealistic") toy model with great deal of symmetry, in which the object under consideration is in some other way concentrated at a single point. Most often this is achieved by treating the point-object as being in what is otherwise vacuum. A classic example of this approach is the Schwarzschild body, a toy model we use to approximate stars while explaining relativistic phenomena such as perihelion of Mercury or gravitational lensing. Moreover, there are theorems which -- very roughly speaking -- confirm that many an (uncharged) isolated gravitational system can be treated as a Schwarzschild body from afar. (Precise formulations can be found in \cite{CD, corvino}.) 

Due to the very large void between stars, galaxies and other constituents of the universe it is perhaps appropriate to think of them as being point-objects in their own right. Thus, there is a great deal of wisdom in the perspective that an integral (i.e continuum) approximates a Riemann sum (i.e a ``realistic" object) as opposed to the standard standpoint that Riemann sums approximate the integral. This is in some ways related to how we can view the integral of mass density as approximating the ``Riemann sum" of individual masses.

With the exception of the so-called swiss-cheese models, cosmological models used to study inhomogeneities in the universe \cite{cosmology}, there is a very limited amount of work done in modeling the universe as a collection of point-objects. To the best of our knowledge, there are no results in the literature which treat relativistic continuum as a limit of (a sequence of) discretized point-object configurations! This relativistic situation lies in sharp contrast with much of the classical context where the mathematical underpinning of the passage from point-object configurations to continuum lies in elementary calculus\footnote{Admittedly, there are substantial difficulties in classical context as well. The issue of self-energy of a charged point source is perhaps a perfect example.}. For example, our description of mass (density) from the opening paragraph demonstrates that in the relativistic context we need much more sophisticated methodology than simple integration. The relativistic difficulties are once again rooted in the lack of additivity of relativistic mass-energy. 

In this paper we attempt to -- pun intended!-- fill the void in the literature and provide a study of ways in which a compactly supported conformally flat dust cloud can be seen as a limit of a de-facto Riemann sum of point-sources.

\subsection{Reissner-Nordstr\"om metrics}
A single charged relativistic point-object in otherwise empty space is commonly modeled by Reissner-Nordstr\"om time-symmetric initial data. Here the metric is given by  
$$g_{\RN}=\left(1+\frac{\alpha}{r}\right)^2\left(1+\frac{\beta}{r}\right)^2\geucl,\ \ \alpha,\beta\ge 0$$ 
and the electric field permits an electrostatic potential: 
$$\vec{E}_{\RN}=\grad_{g_{\RN}}\left(\ln\left(1+\frac{\alpha}{r}\right)-\ln\left(1+\frac{\beta}{r}\right)\right).$$
The metric $g_{\RN}$ and the electric field $\vec{E}_{\RN}$ satisfy the constraints
$$\mathrm{Scal}(g_{\RN})=2|\vec{E}_{\RN}|^2_{g_{\RN}} \text{\ \ and\ \ } \mathrm{div}_{g_{\RN}}(\vec{E}_{\RN})=0.$$
Reissner-Nordstr\"om metrics $g_{\RN}$ have non-negative scalar curvature. 

When the parameters $\alpha$ and $\beta$ are both positive, $g_{\RN}$ has two asymptotically Euclidean ends with the mass and the electric charge given by 
$$m_{_\mathrm{ADM}}(g_{\RN})=\alpha+\beta\text{\ \ and\ \ } Q(g_{\RN},\vec{E}_{\RN})= \beta-\alpha.$$
Please refer to the image in  the left half of Figure \ref{fig1}. Note that there is a minimal surface located at $r=\sqrt{\alpha\beta}$. The area of the minimal surface is given by $4\pi (\sqrt{\alpha}+\sqrt{\beta})^4$, while the length of the neck is computed to be on the order of $1+(\alpha+\beta)(1+|\ln(\alpha\beta)|)$. The latter is an important observation to make because the relationship between $\alpha$ and $\beta$ controls the presence of ``deep gravity wells"; these are crucial for our work in Section \ref{deepwells:section}. As the reader is about to experience, the presence of such deep gravity wells adds a substantial complication to our work. 

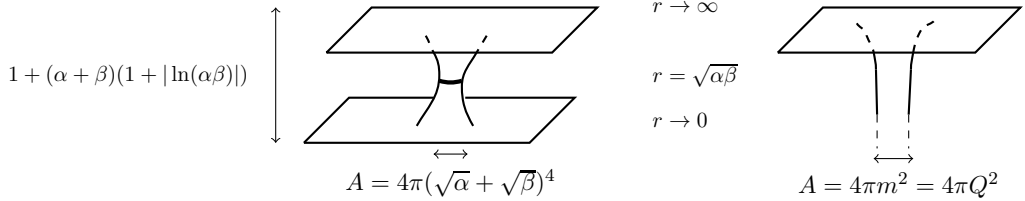
\begin{figure}[h]
\centering
\begin{tikzpicture}[scale=.3]
\draw[thick] (-5,-5) to (5, -5) to (7,-3) to (-3, -3) to (-5, -5);
\draw[thick] (-6,-9) to (4, -9) to (6,-7) to (1.15, -7); 
\draw[thick] (-0.25,-7) to (-4, -7) to (-6, -9);
\draw[thick, dashed] (-.6, -4.25) to [out=-60, in=120] (-0.25, -5);
\draw[thick] (-.25, -5) to [out=-60, in=90] (0, -6) to  [out=-90, in=60](-1, -8.25);
\draw[thick, dashed] (2.1, -4.25) to [out=-120, in=60] (1.65, -5);
\draw[thick] (1.65, -5) to [out=-120, in=90] (1, -6.5) to  [out=-90, in=120](1.5, -8.25);
\draw[ultra thick] (0,-6.25) to [out=-15, in=-165] (1,-6.25);
\node[right] at (9,-8) {\scalebox{0.75}{$r\to 0$}};
\node[right] at (9,-6) {\scalebox{0.75}{$r=\sqrt{\alpha\beta}$}};
\node[right] at (9,-3) {\scalebox{0.75}{$r\to \infty$}};

\draw[<->] (-7.25, -9) to (-7.25, -3);
\node[left] at (-8, -6) {\scalebox{0.75}{$1+(\alpha+\beta)(1+|\ln(\alpha\beta)|)$}};
\draw[<->] (-0.25,-9.5) to (1.25,-9.5);
\node[below] at (0.5, -9.75) {\scalebox{0.85}{$A=4\pi(\sqrt{\alpha}+\sqrt{\beta})^4$}};

\draw[thick] (15,-5) to (23.75, -5) to (25.75,-3) to (17, -3) to (15, -5);

\draw[thick, dashed] (18.5, -3.7) to [out=-60, in=120] (18.9, -4.1) to [out=-80, in=100] (19.25, -5);
\draw[thick] (19.25, -5) to [out=-80, in=90] (19.35, -5.75);

\draw[thick] (19.35, -5.75) to (19.4, -7.75);

\draw[thin, dashed] (19.4, -7.75) to (19.4, -9.25);

\draw[thin, dashed] (20.8, -7.75) to (20.8, -9.25);

\draw[thick] (20.9, -5.75) to [out=90, in=-100] (21, -5);

\draw[thick] (20.9, -5.75) to (20.8, -7.75);

\draw[thick, dashed] (21, -5) to [out=80, in=-120] (21.35, -4) to [out=60, in=-120] (22, -3.5);

\draw[<->] (19.25,-9.7) to (21,-9.7);
\node[below] at (20.35, -9.85) {\scalebox{0.85}{$A=4\pi m^2=4\pi Q^2$}};

\end{tikzpicture}
\caption{$\alpha\beta\neq 0$ on the left; $\alpha\beta=0$ on the right.}\label{fig1}
\end{figure}

In the case when $\alpha=\beta$ we obtain the Schwarzschild point-mass $m=2\alpha=2\beta$. 
In the situation in which we permit exactly one of $\alpha$ and $\beta$ to be zero we have {\it extreme Reissner-Nordstr\"om} initial data where $m_{_\mathrm{ADM}}=|Q|$ and where there is no minimal surface. Instead, extreme Reissner-Nordstr\"om geometry features an asymptotically cylindrical end; see the image on the right of Figure \ref{fig1}. The area of the coordinate sphere at $r=0$ is $4\pi m^2$.

\subsection{Brill-Lindquist metrics}\label{BLintro}

A ``superposition" of point-objects of Reissner-Nordstr\"om-type is studied in detail in the work of D. Brill and R. Lindquist, \cite{BL}. The metric and the electric field for Brill-Lindquist data are given by 
\begin{equation}\label{BL-basic}
g_{\BL}=(\chi_{\BL}\psi_{\BL})^2\geucl,\ \ \vec{E}_{\BL}=\grad_{g_{\BL}}\left(\ln\left(\frac{\chi_{\BL}}{\psi_{\BL}}\right)\right)
\end{equation}
where the functions $\chi_{\BL}$ and $\psi_{\BL}$ take the form of 
\begin{equation}\label{BL-basic-2}
\chi_{\BL}(x)=1+\sum \frac{\alpha_i}{|x-p_i|}
\text{\ \ and\ \ } 
\psi_{\BL}(x)=1+\sum \frac{\beta_i}{|x-p_i|}
\end{equation}
for some $\alpha_i, \beta_i\ge 0$ and some \underline{\emph{finite}} set 
$\P_*$ of points $p_i$. The constraints satisfied by the metric $g_{\BL}$ and the electric field $\vec{E}_{\BL}$ are
$$\mathrm{Scal}(g_{\BL})=2|\vec{E}_{\BL}|^2_{g_{\BL}} \text{\ \ and\ \ } \mathrm{div}_{g_{\BL}}(\vec{E}_{\BL})=0.$$
In particular, $g_{\BL}$ is of non-negative scalar curvature. 

We always assume that at least one of $\alpha_i$ and $\beta_i$ is non-zero.
For geometric reasons which shall become apparent shortly, we often times need to distinguish the subset of sources 
$$\P_{**}=\{p_i\,\big{|} \alpha_i\beta_i\neq 0\}\subseteq \P_*$$
from the remaining sources. In addition to $\alpha_i$ and $\beta_i$ another quantity which substantially impacts the geometry near $p_i$ is  
$$\sigma_i:=\min_{j\neq i} |p_i-p_j|.$$ 

The metric $g_{\BL}$ has an asymptotically Euclidean end at $|x|\to \infty$. There one computes
\begin{equation}\label{massandchargeBL}
m_{_\mathrm{ADM}}(g_{\BL})=\sum (\alpha_i+\beta_i),\ \ Q(g_{\BL},\vec{E}_{\BL})=\sum (\beta_i-\alpha_i).
\end{equation}
Thus in some sense the \emph{effective mass} and the \emph{effective charge} of each individual point source are given by $\alpha_i+\beta_i$ and $\beta_i-\alpha_i$, respectively. 

For each $p_i\in \P_{**}$ we have an additional asymptotically Euclidean end at $x\to p_i$, with ADM mass and charge equal to 
\begin{equation}\label{BL-mass-charge}
\begin{aligned}
m_i&:=\alpha_i + \beta_i + \sum_{j\neq i} \frac{\alpha_j\beta_i + \alpha_i\beta_j}{|p_j-p_i|}\\
Q_i&:=\beta_i - \alpha_i + \sum_{j\neq i} \frac{\alpha_j \beta_i- \alpha_i \beta_j}{|p_j-p_i|}.
\end{aligned}
\end{equation}
For each $p_i\in \P_{*}\smallsetminus\P_{**}$ we have an asymptotically cylindrical end at $x\to p_i$. In this case the areas of concentric coordinate spheres converging to $p_i$ are found to approach $4\pi (m_i)^2=4\pi Q_i^2$ where either
\begin{equation}\label{extremeBL-mass-charge}
m_i=-Q_i=\alpha_i+\sum_{j\neq i} \frac{\alpha_i\beta_j}{|p_i-p_j|} \text{\ \ or\ \ } m_i=Q_i=\beta_i+\sum_{j\neq i} \frac{\alpha_j\beta_i}{|p_i-p_j|},
\end{equation}
depending of whether $\beta_i=0$ or (respectively) $\alpha_i=0$. Note that expressions \eqref{BL-mass-charge} still cover the expressions \eqref{extremeBL-mass-charge} as a special case. Overall, we see that (in some sense) Brill-Lindquist metrics \eqref{BL-basic} -- \eqref{BL-basic-2} serve as initial data for vacuum with Reissner-Nordst\"om-like point sources located at $p_i$ whose masses $m_i$ and charges $Q_i$ are described in \eqref{BL-mass-charge}. 

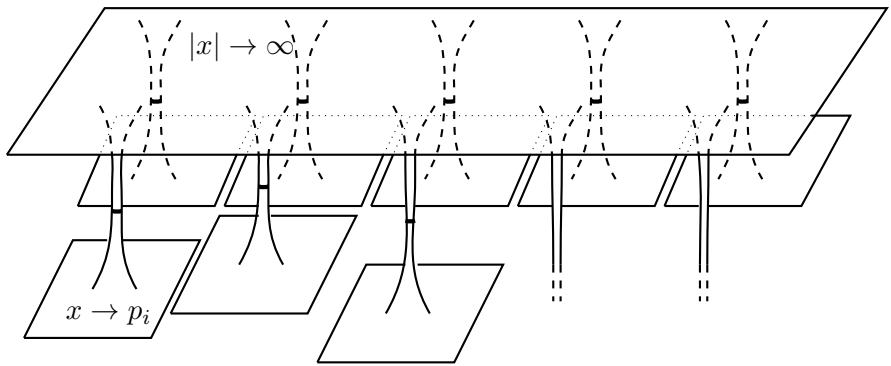
\begin{figure}[h]
\centering
\begin{tikzpicture}[scale=.65]
\draw[thick] (-8.5,-6) to (7.5, -6) to (9.5,-3) to (-6.5, -3) to (-8.5, -6);
\node[right] at (-5,-3.8) {$|x|\to \infty$};

\draw[thick] (-8.15,-9.75) to (-5.55, -9.75) to (-4.55,-7.75) to (-6.15, -7.75); 
\draw[thick] (-6.45, -7.75) to (-7.15, -7.75) to (-8.15, -9.75);
\draw[thick, dashed] (-6.6, -5) to [out=-60, in=95] (-6.35, -6);
\draw[thick] (-6.35, -6) to [out=-85, in=90] (-6.35, -6.75) to [out=-90, in=60](-6.75, -8.75);
\draw[thick, dashed] (-5.75, -5) to [out=-120, in=85] (-6.15, -6);
\draw[thick] (-6.15, -6) to [out=-95, in=90] (-6.15, -7) to [out=-90, in=120](-5.85, -8.75);
\draw[ultra thick] (-6.35,-7.15) to [out=-15, in=-165] (-6.15,-7.15);
\node[right] at (-7.5,-9.3) {$x\to p_i$};

\draw[thick] (-5.15,-9.25) to (-2.35, -9.25) to (-1.35,-7.25) to (-3.1, -7.25); 
\draw[thick] (-3.35, -7.25) to (-4.15, -7.25) to (-5.15, -9.25);
\draw[thick, dashed] (-3.6, -5) to [out=-60, in=95] (-3.35, -6);
\draw[thick] (-3.35, -6) to [out=-85, in=90] (-3.35, -6.25) to [out=-90, in=60](-3.75, -8.25);
\draw[thick, dashed] (-2.75, -5) to [out=-120, in=85] (-3.15, -6);
\draw[thick] (-3.15, -6) to [out=-95, in=90] (-3.15, -6.5) to [out=-90, in=120](-2.85, -8.25);
\draw[ultra thick] (-3.35,-6.65) to [out=-15, in=-165] (-3.15,-6.65);

\draw[thick] (-2.15,-10.25) to (0.65, -10.25) to (1.65,-8.25) to (-0.1, -8.25); 
\draw[thick] (-0.35, -8.25) to (-1.15, -8.25) to (-2.15, -10.25);
\draw[thick, dashed] (-0.6, -5) to [out=-60, in=90] (-0.35, -6);
\draw[thick] (-0.35, -6) to [out=-90, in=90] (-0.35, -6.25) to [out=-90, in=60](-0.75, -9.25);
\draw[thick, dashed] (0.25, -5) to [out=-120, in=90] (-0.15, -6);
\draw[thick] (-0.15, -6) to [out=-90, in=90] (-0.15, -6.5) to [out=-90, in=120](0.15, -9.25);
\draw[ultra thick] (-0.35,-7.35) to [out=-15, in=-165] (-0.15,-7.35);

\draw[thick, dashed] (2.4, -5) to [out=-60, in=95] (2.65, -6);
\draw[thick] (2.65, -6) to [out=-90, in=90] (2.675, -8.25);
\draw[thick, dashed] (2.675, -8.25) to (2.675, -9);
\draw[thick, dashed] (3.25, -5) to [out=-120, in=90] (2.85, -6);
\draw[thick] (2.85, -6) to [out=-90, in=90] (2.825, -8.25);
\draw[thick, dashed] (2.825, -8.25) to (2.825, -9);

\draw[thick, dashed] (5.4, -5) to [out=-60, in=95] (5.65, -6);
\draw[thick] (5.65, -6) to [out=-85, in=90] (5.675, -8.25);
\draw[thick, dashed] (5.675, -8.25) to (5.675, -9);
\draw[thick, dashed] (6.25, -5) to [out=-120, in=90] (5.85, -6);
\draw[thick] (5.85, -6) to [out=-90, in=90] (5.825, -8.25);
\draw[thick, dashed] (5.825, -8.25) to (5.825, -9);

\draw[thick] (-7.05,-7.05) to (-6.35, -7.05);
\draw[thick] (-6.15, -7.05) to (-4.25, -7.05) to (-3.8, -6);
\draw[dotted] (-3.8, -6) to (-3.4,-5.2) to (-6.25, -5.2) to (-6.6, -6);
\draw[thick] (-6.6, -6) to (-7.05, -7.05);
\draw[thick, dashed] (-5.8, -3.25) to [out=-60, in=95] (-5.55, -4.25) to [out=-85, in=90] (-5.55, -4.5) to [out=-90, in=60](-5.95, -6.5);
\draw[thick, dashed] (-4.95, -3.25) to [out=-120, in=85] (-5.35, -4.25) to [out=-95, in=90] (-5.35, -4.75) to [out=-90, in=120](-5.05, -6.5);
\draw[ultra thick] (-5.55,-4.9) to [out=-15, in=-165] (-5.35,-4.9);

\draw[thick] (-4.05,-7.05) to (-3.35, -7.05);
\draw[thick] (-3.15, -7.05) to (-1.25, -7.05) to (-0.8, -6);
\draw[dotted] (-0.8, -6) to (-0.4,-5.2) to (-3.25, -5.2) to (-3.6, -6);
\draw[thick] (-3.6, -6) to (-4.05, -7.05);
\draw[thick, dashed] (-2.8, -3.25) to [out=-60, in=95] (-2.55, -4.25) to [out=-85, in=90] (-2.55, -4.5) to [out=-90, in=60](-2.95, -6.5);
\draw[thick, dashed] (-1.95, -3.25) to [out=-120, in=85] (-2.35, -4.25) to [out=-95, in=90] (-2.35, -4.75) to [out=-90, in=120](-2.05, -6.5);
\draw[ultra thick] (-2.55,-4.9) to [out=-15, in=-165] (-2.35,-4.9);

\draw[thick] (-1.05,-7.05) to (-0.35, -7.05);
\draw[thick] (-0.15, -7.05) to (1.75, -7.05) to (2.2, -6);
\draw[dotted] (2.2, -6) to (2.6,-5.2) to (-0.25, -5.2) to (-0.6, -6);
\draw[thick] (-0.6, -6) to (-1.05, -7.05);
\draw[thick, dashed] (0.2, -3.25) to [out=-60, in=95] (0.45, -4.25) to [out=-85, in=90] (0.45, -4.5) to [out=-90, in=60](0.05, -6.5);
\draw[thick, dashed] (1.05, -3.25) to [out=-120, in=85] (0.65, -4.25) to [out=-95, in=90] (0.65, -4.75) to [out=-90, in=120](0.95, -6.5);
\draw[ultra thick] (0.45,-4.9) to [out=-15, in=-165] (0.65,-4.9);

\draw[thick] (1.95,-7.05) to (2.65, -7.05);
\draw[thick] (2.85, -7.05) to (4.75, -7.05) to (5.2, -6);
\draw[dotted] (5.2, -6) to (5.6,-5.2) to (2.75, -5.2) to (2.4, -6);
\draw[thick] (2.4, -6) to (1.95, -7.05);
\draw[thick, dashed] (3.2, -3.25) to [out=-60, in=95] (3.45, -4.25) to [out=-85, in=90] (3.45, -4.5) to [out=-90, in=60](3.05, -6.5);
\draw[thick, dashed] (4.05, -3.25) to [out=-120, in=85] (3.65, -4.25) to [out=-95, in=90] (3.65, -4.75) to [out=-90, in=120](3.95, -6.5);
\draw[ultra thick] (3.45,-4.9) to [out=-15, in=-165] (3.65,-4.9);

\draw[thick] (4.95,-7.05) to (5.65, -7.05);
\draw[thick] (5.85, -7.05) to (7.75, -7.05) to (8.75, -5.2) to (7.95, -5.2);
\draw[dotted] (7.95,-5.2) to (5.75, -5.2) to (5.4, -6);
\draw[thick] (5.4, -6) to (4.95, -7.05);
\draw[thick, dashed] (6.2, -3.25) to [out=-60, in=95] (6.45, -4.25) to [out=-85, in=90] (6.45, -4.5) to [out=-90, in=60](6.05, -6.5);
\draw[thick, dashed] (7.05, -3.25) to [out=-120, in=85] (6.65, -4.25) to [out=-95, in=90] (6.65, -4.75) to [out=-90, in=120](6.95, -6.5);
\draw[ultra thick] (6.45,-4.9) to [out=-15, in=-165] (6.65,-4.9);

\end{tikzpicture}
\caption{Geometry of Brill-Lindquist metrics}\label{fig2}
\end{figure}
 
There is a difference between the \underline{\emph{effective mass}} $\alpha_i + \beta_i$ and \underline{\emph{bare mass}} $m_i$ of the individual Reissner-Nordstr\"om particles. We have 
$$(\alpha_i + \beta_i) - m_i=-\sum_{j\neq i} \frac{\alpha_j \beta_i + \alpha_i\beta_j}{|p_j-p_i|}$$
so that the total ``effective mass" of the system is smaller than the sum of the individual ``bare masses":
\begin{equation}\label{BL-interactionenergy}
m_{_\mathrm{ADM}}-\sum_{i} m_i=-\sum_{i}\sum_{j\neq i} \frac{\alpha_j \beta_i+ \alpha_i \beta_j}{|p_j-p_i|}.
\end{equation}
The difference can be interpreted in terms of \emph{interaction energy}. 
More specifically, after restoring physical units\footnote{To perform this it suffices to introduce a multiplicative factor of $\frac{G}{c^2}$ in front of each of the parameters $\alpha_i$, $\beta_i$, $m_i$ and $q_i$.} in \eqref{BL-mass-charge} we discover that  
$\alpha_i + \beta_i =m_i+O(\frac{G}{c^2})$. The difference between the effective and the bare potential energy is the Newtonian potential energy 
$$m_{_\mathrm{ADM}}c^2 - \sum_{i} m_ic^2=-G\sum_{i,j,j\neq i} \frac{m_im_j-q_iq_j}{|p_j-p_i|}$$
up to a term of order $O(\frac{G}{c^2})$. For more details the reader is referred to \cite{BL}. It is worth emphasizing that charge does behave additively, $Q_{\BL}=\sum_{i} Q_i$.
 
We conclude this portion of the Introduction with the observation that the superposition of extreme Reissner-Nordstr\"om data evolves into the spacetime 
\begin{equation}\label{MP}
-\left(1+\sum \frac{m_i}{2|x-p_i|}\right)^{-2}dt^2+\left(1+\sum \frac{m_i}{2|x-p_i|}\right)^2\geucl
\end{equation}
which is in the family of Majumdar-Papapetrou spacetimes. Due to time-independency of the metric coefficients in \eqref{MP} these space-times are often interpreted as collections of charged black holes (extreme Reissner-Nordstr\"om bodies) in equilibrium. More information on this can, for example, be found in \cite{HartleHawking}.

\subsection{Brill-Lindquist-Riemann sums and charged dust clouds}\label{BLR-Sec}

\subsubsection{The definition of a Brill-Lindquist-Riemann sum}
The starting point are smooth non-negative functions $A(x)$ and $B(x)$ supported in some box 
$$[-D, D]^3\subseteq \mathbb{R}^3.$$ 
For each $n\in\mathbb{N}$ we form a subdivision of $[-D,D]^3$ into boxes of side length $\frac{D}{n}$ and consider the set  
$$\left\{\left(\frac{i+(1/2)}{n}D,\frac{j+(1/2)}{n}D,\frac{k+(1/2)}{n}D\right)\, \Big{|}\, -n\le i, j, k < n\right\}$$
of all the centers of all the boxes of side $D/n$. From now we enumerate and label the subdivision boxes as $\{V_{i,n}\}$ and their centers with $p_{i,n}$; note that their Euclidean volume satisfies 
$$\mathrm{Vol}(V_{i,n})=(D/n)^3.$$ 
In parallel to quantities $\alpha_i+\beta_i$ and $\beta_i-\alpha_i$ we think of $(A+B)\dvol$ and $(B-A)\dvol$ as playing the role of the \underline{\emph{effective}} mass and charge \emph{density} distribution. We now associate them to each point source. 

\begin{definition}\label{BLR}
By a \emph{Brill-Lindquist-Riemann sum of point sources} we mean the metric $g_n$ and the electric field $\vec{E}_n$ given by  
$$g_n=(\chi_n\psi_n)^2\geucl,\ \  \vec{E}_n=\grad_{g_n}\ln(\chi_n/\psi_n)$$
with
$$\chi_n(x) := 1+\sum_i\frac{\alpha_{i,n}}{|x-p_{i,n}|} \text{\ \ and \ \ } \psi_n(x) := 1+\sum_i\frac{\beta_{i,n}}{|x-p_{i,n}|},$$
where the family of parameters $\alpha_{i,n}$ and $\beta_{i,n}$ satisfies 
$$\begin{aligned}
&\C(\alpha,A):=\sup_{i,n} n^4|\alpha_{i,n}-A(p_{i,n})(D/n)^3|<\infty,\\   
&\C(\beta, B):=\sup_{i,n} n^4|\beta_{i,n}-B(p_{i,n})(D/n)^3|<\infty.
\end{aligned}$$
The phrase \emph{Brill-Lindquist-Riemann sum with no charge} assumes $A=B$ and $\alpha_{i,n}=\beta_{i,n}$ for all $n$ and $i$.
\end{definition}

The canonical choice for parameters $\alpha_{i,n}$ and $\beta_{i,n}$ is 
$$\alpha_{i,n}=A(p_{i,n})(D/n)^3,\ \beta_{i,n}=B(p_{i,n})(D/n)^3;$$
such a choice leads to what we henceforth refer to as the sequence of Brill-Lindquist-Riemann sums of midpoint type. However, just as in the case of classic Riemann integration we would like to permit the sample locations to vary within the subdivision box. Note that for a sample location $q_{i,n}$ within $V_{i,n}$ we have 
$$|A(q_{i,n})-A(p_{i,n})|\le \tfrac{D}{n}\|dA\|_{L^\infty},\ \ |B(q_{i,n})-B(p_{i,n})|\le \tfrac{D}{n}\|dB\|_{L^\infty},$$
so that the parameters
\begin{equation}\label{commonsensechoice}
\alpha_{i,n}=A(q_{i,n})(D/n)^3,\ \beta_{i,n}=B(q_{i,n})(D/n)^3
\end{equation}
do indeed satisfy 
$$\C(\alpha, A)\le D^4\|dA\|_{L^\infty}<\infty,\ \ \C(\beta, B)\le D^4\|dB\|_{L^\infty}<\infty.$$
In some other circumstances we might want to view the parameters $\alpha_{i,n}$ and $\beta_{i,n}$ as arising from approximations or ``measurements", and so building in a little room for error may pay huge dividends.   

Since we may disregard subdivision boxes where $\alpha_{i,n}=\beta_{i,n}=0$ we may, without any loss of generality, assume that $\mathcal{P}_{n,*}$ contains no $p_{i,n}$ where $\alpha_{i,n}=\beta_{i,n}=0$:
$$\mathcal{P}_{n,*}=\{p_{i,n}\big{|}\alpha_{i,n}+\beta_{i,n}\neq 0\}.$$
The separation parameter in this case is 
$$\sigma_{i,n}=D/n \text{\ \ for all\ \ } i, n.$$
Also observe that, by construction, we have $0\not\in \P_{n,*}$ and 
\begin{equation}\label{moreseparation}
|p_{i,n}|>\sigma_{i,n}/2 \text{\ \ for\ \ all\ \ } p_{i,n}\in \P_{n,*}.
\end{equation}

\subsubsection{Charged dust clouds}\label{TheIdea}
Ultimately the goal of this paper is to investigate the convergence of the sequence of Brill-Lindquist-Riemann sums towards a continuum, a charged dust cloud. The conformal factors $\chi_n$ and $\psi_n$ take the schematic form of Riemann sums for the integrals 
\begin{equation}\label{potentials1}
\chi(x)=1+\int_y \frac{A(y)}{|x-y|}\dvol,\ \ \psi(x)=1+\int_y \frac{B(y)}{|x-y|}\dvol.
\end{equation} 
The natural candidate for our continuum dust cloud is thus given by 
\begin{equation}\label{defng}
g=(\chi\psi)^2\geucl,\ \  \vec{E}=\grad_g\ln(\chi/\psi).
\end{equation}
The functions $\chi$ and $\psi$ solve the PDE's
\begin{equation}\label{potentials2}
\Delta_{\geucl}\chi=-4\pi A,\ \ \Delta_{\geucl}\psi=-4\pi B
\end{equation}
and satisfy the boundary conditions $\chi,\psi\to 1$ as $|x|\to \infty$. The constraints satisfied by $(g,\vec{E})$ are
\begin{equation}\label{chriscross}
\mathrm{Scal}(g)=16\pi \chi^{-3}\psi^{-3}(A\psi+B\chi)+2|\vec{E}|^2_{g} \text{\ \ and\ \ } \mathrm{div}_{g}(\vec{E})=-4\pi \chi^{-3}\psi^{-3}(A\psi-B\chi).
\end{equation}
The latter are constraints for electrostatic charged dust with mass-energy density $\rho_{\mathrm{mass}}$ and charge density $\rho_{\mathrm{el}}$ given by 
\begin{equation}\label{chriscross2}
\rho_{\mathrm{mass}}=\chi^{-3}\psi^{-3}(A\psi+B\chi),\ \ \rho_{\mathrm{el}}=\chi^{-3}\psi^{-3}(A\psi-B\chi).
\end{equation}
Since the functions $A$ and $B$ are non-negative we in addition have the inequality 
\begin{equation}\label{DEC}
|\rho_{\mathrm{el}}|\le \rho_{\mathrm{mass}}.
\end{equation}
This particular inequality is consistent with what has been identified in the literature as the dominant energy condition (DEC), and with results concerning DEC such as those in \cite{CRT} (Theorem 2.1), \cite{GHHP} (see its Section 4) and \cite{KW1} (Theorem 2).

Alternatively, mass-energy and charge involved in \eqref{chriscross2} can be expressed using $3$-forms: 
\begin{equation}\label{chriscross3}
\begin{cases}
&\omegam=\rho_{\mathrm{mass}}\mathrm{dvol}_g=\psi A\,\dvol + \chi B\,\dvol,\\  &\omegae=\rho_{\mathrm{el}}\mathrm{dvol}_g=\psi A\,\dvol - \chi B\,\dvol
\end{cases}
\end{equation}
There are some very compelling reasons which support the idea that charged dust clouds are better described in terms of $3$-forms $(\omega_{\mathrm{mass}},\omega_{\mathrm{el}})$ as opposed to densities $(\rho_{\mathrm{mass}},\rho_{\mathrm{el}})$. For example, consider the fact that the concept of density inherently involves the concepts of metric and volume. As a result, employment of densities in the formulations of the constraints makes it impossible to view the constraints as equations which inform us about the geometric responses to presence of matter. By expressing the ``amount" of matter present in a metric independent way, such as the one involving $3$-forms $\omega_{\mathrm{mass}}$ and $\omega_{\mathrm{el}}$, we are able to frame the solutions to the constraints as responses to presence of matter. 

Here is the precise meaning of the phrase charged dust cloud we use throughout this work.
\begin{definition}\label{matter}
By a charged dust cloud we mean:
\begin{enumerate}[label=\Roman*]
\item\label{omegadefn} The pair $(\omegam, \omegae)$ where $\omegam=\rho_{\mathrm{mass}}\dvol$ and $\omegae=\rho_{\mathrm{el}}\dvol$ are $3$-forms with $\rho_{\mathrm{mass}}\ge 0$ and $\rho_{\mathrm{el}}$ smooth and compactly supported. 
\medbreak
\item The solution $(g,\vec{E})$ of the constraint equations 
\begin{equation}\label{davidconstraints}
\mathrm{Scal}(g)\mathrm{dvol}_g=16\pi \omegam+2|\vec{E}|^2_{g}\mathrm{dvol}_g \text{\ \ and\ \ } \mathrm{div}_{g}(\vec{E})\mathrm{dvol}_g=-4\pi \omegae
\end{equation}
in the form of \eqref{defng} and with conformal factors $\chi$ and $\psi$ satisfying the asymptotic conditions
\begin{equation}\label{asymptotics}
\left|\partial_x^l\!\left(\chi(x)-1\right)\right|,\ \left|\partial_x^l\!\left(\psi(x)-1\right)\right|=O(|x|^{-l-1}),\ \ |x|\to \infty,\ \ l\ge0.
\end{equation}
\end{enumerate}
In addition, if \eqref{DEC} holds then we say that the charged dust cloud satisfies the dominant energy condition (DEC). 
\end{definition}

When expressed in terms of the conformal factors $\chi$, $\psi$ and $3$-forms $\omega_{\mathrm{mass}}$ and $\omega_{\mathrm{el}}$ the constraint equations for $g=(\chi\psi)^2 \geucl$ and $\vec{E}=\grad_{g}(\ln(\chi/\psi))$ read as follows:
\begin{equation}\label{TobySystem}
\begin{cases}
\psi \Delta_{\geucl}\chi\,\dvol=-2\pi(\omega_{\mathrm{mass}}+\omega_{\mathrm{el}})\\
\chi \Delta_{\geucl}\psi\,\dvol=-2\pi (\omega_{\mathrm{mass}}-\omega_{\mathrm{el}}).
\end{cases}
\end{equation} 
Note that the DEC makes the right hand sides of this system non-positive. The existence and the uniqueness of positive solutions  $(\chi, \psi)$ of this system is proven by T. Aldape in \cite{toby1}: 

\begin{theorem}\label{TobyThm}
Suppose $(\omegam, \omegae)$ satisfies the conditions \eqref{omegadefn} of Definition \ref{matter} and the dominant energy condition. Then there exist unique positive solutions $\chi$ and $\psi$ of \eqref{TobySystem}, subject to the asymptotic conditions \eqref{asymptotics}.
\end{theorem}

Thus Theorem \ref{TobyThm} ensures that to each pair $(\omegam, \omegae)$ which satisfies the conditions \eqref{omegadefn} of Definition \ref{matter} and the DEC we can associate a unique solution of the constraints \eqref{davidconstraints}. It is in this sense of the word that every charged dust cloud satisfying DEC is of the form \eqref{defng}. Readers familiar with the difficulties surrounding the prescribed scalar curvature problem on $\mathbb{R}^3$ probably notice that this particular outcome of Theorem \ref{TobyThm} would no longer hold if instead of the pair of $3$-forms $(\omegam, \omegae)$ we used the pair of densities $(\rho_{\mathrm{mass}}, \rho_{\mathrm{el}})$.

In summary, the idea of this paper is to examine:
\begin{itemize}
\item If (and if so, in what sense) the sequence of Brill-Lindquist-Riemann sum converges to a charged dust cloud (satisfying DEC) and,
\medbreak
\item The extent to which every charged dust cloud (satisfying DEC) can be discretized through an approximation by a Brill-Lindquist-Riemann sum. 
\end{itemize}

\subsubsection{Difficulties with our idea}
The idea we just presented rests on the symbolic passage between a Riemann-sum-looking expression and an integral. Such a maneuver is more delicate than it might seem at first. For one, observe that the constraint equations satisfied by $(g_n, \vec{E}_n)$ are 
$$\mathrm{Scal}(g_n)=2|\vec{E}_n|^2_{g} \text{\ \ and\ \ } \mathrm{div}_{g_n}(\vec{E}_n)=0,$$
which when compared to \eqref{chriscross} suggests a jump in scalar curvature. (See also the following two paragraphs.) In particular, there can be no $C^2$-like convergence along the lines of $g_n\to g$. Furthermore, the set $\cup_n \P_{n,*}$ of locations of point-objects is dense in $[-D,D]^3$ making it so that there is no subset of $[-D,D]^3$ with a non-empty interior on which the statement $g_n\to g$ even makes sense. 

The exact nature of point-wise properties of $g_n$ and their convergence is investigated in Section \ref{ConvSec}. One of the conclusions of Section \ref{ConvSec} is that the metrics $g_n$ are well-approximated at the $C^1$-level by the metric $g$ but only over sets of the form
$$\mathbb{R}^3\smallsetminus \left(\cup_i B_{\geucl} (p_{i,n}, Dn^{-\nu}\right)$$
with $1<\nu<3/2$. Excising neighborhoods of $p_{i,n}$ is absolutely necessary for a very intuitive reason: there is a sense in which the metrics $g_n$ ought to be\footnote{Another genre of results we present in Section \ref{ConvSec} address this ``convergence" towards metrics of Reissner-Nordstr\"om-type near $p_{i,n}$.} ``like" Reissner-Nordstr\"om metrics near point-sources at $p_{i,n}$. Note that the electric field
$$\vec{E}_n=\grad_{g_n}\ln(\chi_n/\psi_n)$$
involves only the first derivatives of $\chi_n$ and $\psi_n$, and because of this we do have ``convergences" 
$$\vec{E}_n\approx \vec{E} \text{\ \ and\ \ } 2|\vec{E}_n|^2_{g_n}\approx 2|\vec{E}|^2_{g}$$ 
over sets of the form $\mathbb{R}^3\smallsetminus \left(\cup_i B_{\geucl} (p_{i,n}, Dn^{-\nu}\right)$ where $\vec{E}=\grad_g\ln(\chi/\psi)$.

As discussed above there can be no statement of the form $g_n\to g$ at the $C^2$-level, but it is worth noticing that (Euclidean) second derivatives of $g_n$ do permit uniform bounds to some small extent. In Section \ref{ConvSec} we prove that for each fixed value of $c\ll 1$ there is a uniform bound on second derivatives of $g_n$ over sets of the form 
$$\mathbb{R}^3\smallsetminus \left(\cup_i B_{\geucl} (p_{i,n}, cD/n\right).$$
In other words, even though curvatures of $g_n$ are not expected to converge to those of $g$ in any straightforward sense, at least they exhibit boundedness sufficiently away from sources $p_{i,n}$. The  same kind of comment applies to boundedness/lack of convergence of $\mathrm{div}_{g_n}(\vec{E}_n)$.

In summary, one could say that there is $C^1$-convergence $g_n\to g$ over any subset of the form  
\begin{equation}\label{cantor}
K\smallsetminus \left(\bigcup_{i,n\ge N} B_{\geucl} (p_{i,n}, Dn^{-\nu})\right),\ \ 1<\nu<3/2
\end{equation}
of a compact set $K\subseteq \mathbb{R}^3$. Interiors of sets \eqref{cantor} are disjoint from $[-D,D]^3$ but their Lebesgue measure can be made arbitrarily close to that of the compact set $K$ when $\nu>4/3$, due to 
$$\sum_{n\ge N} \sum_i \mathrm{Vol}_{\geucl} B_{\geucl} (p_{i,n}, Dn^{-\nu})\le CD^3 \sum_{n\ge N} n^{3-3\nu}= O(N^{4-3\nu}).$$
This particular Cantor-esque aspect of sets \eqref{cantor} makes it unclear if and how one could make use of results of Cheeger-Gromov theory (e.g \cite{MAnderson}) or -- more importantly -- the well-posedness of the Einstein equations (e.g \cite{KRS, Ringstrom}). It seems that any investigation of the idea we presented in Section \ref{TheIdea} has to be rooted in techniques of metric geometry and/or geometric measure theory. 

Theorem \ref{GHduhduhduh}, which we are about to state, is perhaps the kind of result we need as the foundation of our investigations; the theorem itself is proven in Section \ref{GH:sec}. Before we make the statement we have to make one disclaimer. Since the convergence of non-compact domains requires extra care we have simplified the situation by restricting ourselves to working within $B_{\geucl}(0,R)\smallsetminus \P_n$ and $B_{\geucl}(0,R)$ with $R>\sqrt{3}D$ fixed. The work of \cite{CD, corvino} implies that regions near infinity can be treated as being exactly of Reissner-Nordstr\"om type anyway, and so we do not feel that by restricting to finite $R$ we sacrificed a lot of generality. This particular simplification is implemented throughout our paper. 

\begin{theorem}\label{GHduhduhduh}
Fix $R>\sqrt{3}D$. The set 
$$\U_{n,R}:=B_{\geucl}(0,R)\smallsetminus \left(\bigcup_i B_{\geucl}(p_{i,n}, \tfrac{D}{n^2})\right)$$
equipped with the metric $g_n$ converges in the Gromov-Hausdorff sense as $n\to \infty$ to the set $B_{\geucl}(0,R)$ equipped with the metric $g$ introduced in \eqref{defng}-\eqref{potentials1}. \end{theorem}

At first glance it appears that Theorem \ref{GHduhduhduh} achieves our goal. Once again, the idea of cutting out neighborhoods of locations in $\P_n$ (i.e $B_{\geucl}(p_{i,n}, \tfrac{D}{n^2})$) seems perfectly reasonable as we expect the geometries of $g_n$ to be more Reissner-Nordstr\"om type near $p_{i,n}$. 
The theorem achieves the physically worthwhile goal discussed in the opening section of this paper as it allows us to represent non-vacuum continuum as a limit of largely vacuum, discretized configurations. Stated in more geometric terms the theorem provides a source of examples for sequences of scalar-flat manifolds whose Gromov-Hausdorff limit is not scalar-flat. This particular take on Theorem \ref{GHduhduhduh} deserves to be stated explicitly as a corollary. The proof of the corollary is merely an application of \eqref{chriscross} within the context where there is no charge.

\begin{corollary}\label{BS}
Let $A\neq 0$ be a smooth non-negative function supported in $[-D,D]^3$, and let $\chi$ be the unique solution of 
$$\Delta_{\geucl}\chi=-4\pi A,\ \ \chi\big{|}_\infty=1.$$
Fix $R> \sqrt{3}D$. There exists a sequence of spaces $(\U_{n,R},g_n)$ which are scalar-flat but whose Gromov-Hausdorff limit is $(B_{\geucl}(0,R), \chi^4\geucl)$ and is of scalar curvature
$$\mathrm{Scal}(\chi^4\geucl)=32\pi \chi^{-5}A\neq 0.$$
\end{corollary}
In relation to this corollary the reader may also want to consult Remark \ref{christinasthing} below.

However, $(\U_{n,R},g_n, \vec{E}_n)$ and $(B_{\geucl}(0,R)\smallsetminus \P_n ,g_n, \vec{E}_n)$ are extremely different when viewed as \emph{relativistic initial data} and one can make an argument that consideration of $(\U_{n,R},g_n)$ is highly physically unsatisfying! Metaphorically speaking, since $D/n^2\gg \alpha_{i,n}\pm \beta_{i,n}$ consideration of $\U_{n,R}$ in place of $B_{\geucl}(0,R)\smallsetminus \P_n$ means cutting off regions which are even remotely close to ``stars" (point-sources) of a ``galaxy" (dust cloud). By doing so we are removing regions which allow us to detect classic relativistic effects (e.g gravitational lensing) in the first place! In the space-time evolution of $\U_{n,R}$ there are plenty of signals which cannot reach certain destinations within $\U_{n,R}$ because of a highly artificial boundary raised at the (Euclidean) radius $D/n^2$. The far more natural boundaries for the purposes of space-time evolutions are horizons such as minimal surfaces suggested in the diagram in Figure \ref{fig2}; all signals crossing them are lost to hypothetical observers anyway. 

For this reason it is absolutely essential to address the existence of any horizons/minimal surfaces within the (Euclidean) radius $D/n^2$ of a ``star" (point-source). If the interpretation of Brill-Lindquist metrics as collections of (charged) point-sources and the corresponding reading of Figure \ref{fig2} are not completely misleading one ought to be able to associate a somehow canonical horizon/minimal surface (henceforth denoted by $\Sigma_{i,n}$) to each point-source $p_{i,n}\in \P_{n,**}$. A substantial portion of our paper is dedicated to resolving exactly this issue. 

\subsubsection{The plan}
Theorem \ref{minsurf:thm1} establishes existence, uniqueness and further geometric properties regarding minimal surfaces $\Sigma_{i,n}\subseteq B_{\geucl}(p_{i,n},D/(2n))$. Techniques involved in proving Theorem \ref{minsurf:thm1} extend to the setting of Brill-Lindquist metrics in general, and for this reason we develop a general theorem (Theorem \ref{minsurf:thm1-taos}) first; we then extract Theorem \ref{minsurf:thm1} as a corollary. For now the reader should note that these results apply only when $n$ is sufficiently large. The majority of convergence results in our article address the domain located inside the ball $B_{\geucl}(0,R)$ but \emph{outside} all of $\Sigma_{i,n}$. Here is a precise definition.

\begin{definition}\label{defn-mn}
Fix $n$ which is sufficiently large\footnote{The exact meaning of ``sufficiently large" is spelled out in Theorem \ref{minsurf:thm1}.} and consider minimal surfaces $\Sigma_{i,n}$ of Theorem \ref{minsurf:thm1}. By the \emph{outside} of $\Sigma_{i,n}$, denoted $\mathrm{Out}_R(\Sigma_{i,n})$, we mean the connected component of $B_{\geucl}(0,R)\smallsetminus \Sigma_{i,n}$ which does not contain $p_{i,n}$. Define
$$\V_{n,R}:=\left(\bigcap_{i}\, \mathrm{Out}_R(\Sigma_{i,n})\right)\smallsetminus \P_{n,*},$$
where the intersection is over all $p_{i,n}\in \P_{n,**}$. 
\end{definition}

The diagram in Figure \ref{fig3} conveys the appearance of $\V_{n,R}$. 

\begin{figure}[h]
\centering
\begin{tikzpicture}[scale=.6]
\draw[thick] (-8,-4.25) [out=-105, in=180] to (-0.5, -7) [out=0, in=-120] to (8.25, -4.25)[out =60, in=0] to (-1, -1.75)[out=180, in=75] to (-8, -4.25);

\begin{scope}[yshift=-0.25cm]
\draw[thick, dashed] (-6.6, -5) to [out=-60, in=95] (-6.35, -6.35);
\draw[thick] (-6.35, -6.35) to [out=-85, in=90] (-6.35, -7.15);
\draw[thick, dashed] (-5.75, -5) to [out=-120, in=85] (-6.15, -6.4);
\draw[thick] (-6.15, -6.4) to [out=-95, in=90] (-6.15, -7.15);
\draw[ultra thick] (-6.35,-7.15) to [out=-15, in=-165] (-6.15,-7.15);

\draw[thick, dashed] (-3.6, -5) to [out=-60, in=95] (-3.35, -6) to [out=-85, in=90] (-3.35, -6.45);
\draw[thick, dashed] (-2.75, -5) to [out=-120, in=85] (-3.15, -6) to [out=-95, in=90] (-3.15, -6.45);
\draw[ultra thick] (-3.35,-6.5) to [out=-15, in=-165] (-3.15,-6.5);

\draw[thick, dashed] (-0.6, -5) to [out=-60, in=90] (-0.35, -6.7);
\draw[thick] (-0.35, -6.7) to [out=-90, in=90] (-0.35, -7.35);
\draw[thick, dashed] (0.25, -5) to [out=-120, in=90] (-0.15, -6.7);
\draw[thick] (-0.15, -6.7) to [out=-90, in=90] (-0.15, -7.35);
\draw[ultra thick] (-0.35,-7.35) to [out=-15, in=-165] (-0.15,-7.35);

\node[left] at (-0.15,-7.75) {$\Sigma_{i,n}$};

\draw[thick, dashed] (2.4, -5) to [out=-60, in=95] (2.65, -6.75);
\draw[thick] (2.65, -6.75) to [out=-90, in=90] (2.675, -7.5);
\draw[thick, dashed] (2.675, -7.5) to (2.675, -8.25);
\draw[thick, dashed] (3.25, -5) to [out=-120, in=90] (2.85, -6.75);
\draw[thick] (2.85, -6.75) to [out=-90, in=90] (2.825, -7.5);
\draw[thick, dashed] (2.825, -7.5) to (2.825, -8.25);

\draw[thick, dashed] (5.4, -5) to [out=-60, in=95] (5.65, -6.35);
\draw[thick] (5.65, -6.35) to [out=-85, in=90] (5.675, -7.25);
\draw[thick, dashed] (5.675, -7.25) to (5.675, -8);
\draw[thick, dashed] (6.25, -5) to [out=-120, in=90] (5.85, -6.35);
\draw[thick] (5.85, -6.35) to [out=-90, in=90] (5.825, -7.25);
\draw[thick, dashed] (5.825, -7.25) to (5.825, -8);
\end{scope}


\begin{scope}[yshift=0.75cm]
\draw[thick, dashed] (-5.8, -3.25) to [out=-60, in=95] (-5.55, -4.25) to [out=-85, in=90] (-5.55, -4.9);
\draw[thick, dashed] (-4.95, -3.25) to [out=-120, in=85] (-5.35, -4.25) to [out=-95, in=90] (-5.35, -4.9);
\draw[ultra thick] (-5.55,-4.9) to [out=-15, in=-165] (-5.35,-4.9);

\draw[thick, dashed] (-2.8, -3.25) to [out=-60, in=95] (-2.55, -4.25) to [out=-85, in=90] (-2.55, -4.9);
\draw[thick, dashed] (-1.95, -3.25) to [out=-120, in=85] (-2.35, -4.25) to [out=-95, in=90] (-2.35, -4.9);
\draw[ultra thick] (-2.55,-4.9) to [out=-15, in=-165] (-2.35,-4.9);

\draw[thick, dashed] (0.2, -3.25) to [out=-60, in=95] (0.45, -4.25) to [out=-85, in=90] (0.45, -4.9);
\draw[thick, dashed] (1.05, -3.25) to [out=-120, in=85] (0.65, -4.25) to [out=-95, in=90] (0.65, -4.9);
\draw[ultra thick] (0.45,-4.9) to [out=-15, in=-165] (0.65,-4.9);

\draw[thick, dashed] (3.2, -3.25) to [out=-60, in=95] (3.45, -4.25) to [out=-85, in=90] (3.45, -4.9);
\draw[thick, dashed] (4.05, -3.25) to [out=-120, in=85] (3.65, -4.25) to [out=-95, in=90] (3.65, -4.9);
\draw[ultra thick] (3.45,-4.9) to [out=-15, in=-165] (3.65,-4.9);

\draw[thick, dashed] (6.2, -3.25) to [out=-60, in=95] (6.45, -4.25) to [out=-85, in=90] (6.45, -4.9);
\draw[thick, dashed] (7.05, -3.25) to [out=-120, in=85] (6.65, -4.25) to [out=-95, in=90] (6.65, -4.9);
\draw[ultra thick] (6.45,-4.9) to [out=-15, in=-165] (6.65,-4.9);
\end{scope}

\end{tikzpicture}
\caption{Depiction of $(\V_{n,R}, g_n)$}\label{fig3}
\end{figure}
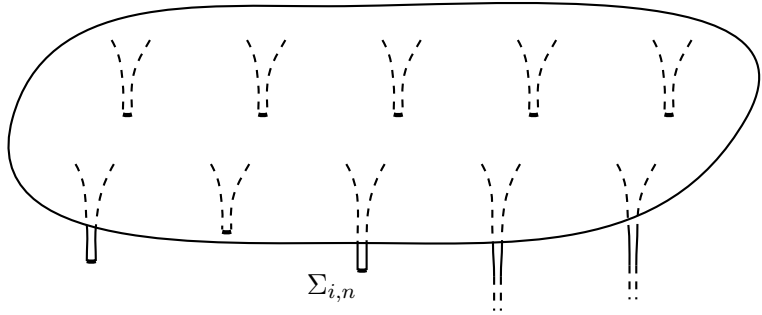

The overall plan of our paper is to study convergence of the sequence of metric spaces $(\V_{n,R}, g_n)$. We first address its Gromov-Hausdorff convergence or lack there of. We then investigate its intrinsic flat limit. Since the target audience for the article are researchers interested in relativity, we also provide a brief overview of both genres of limits in their respective sections.  

\subsection{Our main results}

Let us first introduce some terminology which helps in discussions of our results. 

\subsubsection{Control parameters and classes}\label{controlclasses}
At many places in our paper we rely on an assumption that some quantity associated to a Brill-Lindquist metric $(\chi_{\BL}\psi_{\BL})^2\geucl$ is ``sufficiently small". What exactly constitutes a sufficiently small quantity depends on parameters such as $\alpha_i$, $\beta_i$ or $\sigma_i$. The question of uniformity naturally and frequently comes up. Here is our language for this kind of a situation. 

\begin{definition}\label{defn-cs}
A constant of class $\C(i)$ is a polynomial expression in variables 
$$\hat{\chi}^{(i)}:=1+\sum_{j\neq i} \frac{\alpha_j}{|p_i-p_j|},\ \ \hat{\psi}^{(i)}:=1+\sum_{j\neq i} \frac{\beta_j}{|p_i-p_j|} \text{\ \ and\ \ } \frac{\alpha_i+\beta_i}{\sigma_i}$$
whose coefficients are some fixed (universal) non-negative real numbers. We say that a value $\e>0$ is small relative to $\C(i)$ if $\tfrac{1}{\e}$ is bounded from above by an element of class $\C(i)$.
\end{definition}

For example, our main result regarding minimal surfaces -- Theorem \ref{minsurf:thm1-taos} -- only applies when $\e>0$ is small relative to $\C(i)$.   

The reader surely notices that quantities $\hat{\chi}^{(i)}$ and $\hat{\psi}^{(i)}$ control the size of the functions 
$$\chi^{(i)}:=\chi_{\BL}-\tfrac{\alpha_i}{|x-p_i|} \text{\ \ and\ \ }\psi^{(i)}:=\psi_{\BL}-\tfrac{\beta_i}{|x-p_i|}$$ 
in the vicinity of $p_i$. Specifically, since $\partial^l(\tfrac{1}{|x-p_i|})$ can be bounded by a universal multiple of $\tfrac{1}{|x-p_i|^{l+1}}$ the following hold over 
$|x-p_i|\le \sigma_i/2$:
$$|\partial^l \chi^{(i)}|\le \frac{C}{\sigma_i^l}\hat{\chi}^{(i)} \text{\ \ and\ \ } |\partial^l \psi^{(i)}|\le \frac{C}{\sigma_i^l}\hat{\psi}^{(i)};$$
the constants $C$ depend on $l$ but are otherwise universal. 
 
When studying Brill-Lindquist-Riemann sums we are looking into an infinite family of Brill-Lindquist metrics. For this reason we are often interested in bounds on the quantities 
$$\hat{\chi}^{(i)}, \ \ \hat{\psi}^{(i)} \text{\ \ and\ \ } \frac{\alpha_i+\beta_i}{\sigma_i}$$
which are uniform across $i$ or, perhaps, a whole family of Brill-Lindquist metrics. The crucial point here is that in presence of such uniform bounds a constant of class $\C(i)$ or a value which is small relative to $\C(i)$ can be chosen independently of $i$ or the family of Brill-Lindquist metrics. In the special circumstance of the sequence of Brill-Lindquist-Riemann sums the quantities in question can be bounded uniformly by quantities $\|A\|D^2$, $\|B\|D^2$, $\C(\alpha, A)/D$ and $\C(\beta, B)/D$; for example, see discussion of Proposition \ref{BLRcontrollemma}. All the norms involved here are Euclidean $L^\infty(\mathbb{R}^3)$-norms.

\begin{definition}\label{defn-cs-2}
By constants of class $\C$ we mean polynomial expressions in variables $\|A\|D^2$, $\|B\|D^2$, $\C(\alpha, A)/D$ and $\C(\beta, B)/D$ whose coefficients are universal non-negative real numbers. We say that a property holds for all $n$ which are large enough relative to $\C$ if the property holds for all $n\ge N$ where $N$ is of class $\C$.
\end{definition}

For example, Theorem \ref{minsurf:thm1-taos} mentioned above, as well as Definition \ref{defn-mn} itself, apply when $n$ which is large enough relative to $\C$. The language we just introduced can also be helpful for keeping track of rates of convergences, especially if we modify our class of constants. 

\begin{definition}\label{defn-cs-3}
By constants of class $\C^+$ we mean polynomial expressions in variables $\|A\|D^2$, $\|B\|D^2$, $\|dA\|D^3$, $\|dB\|D^3$, $\C(\alpha, A)/D$ and $\C(\beta, B)/D$ whose coefficients are universal non-negative real numbers. We say that a property holds for all $n$ which are large enough relative to $\C^+$ if the property holds for all $n\ge N$ where $N$ is of class $\C^+$.
\end{definition}

The sequences of Brill-Lindquist-Riemann sums we study depend on real valued parameters other than the ones discussed thus far. In this paper the reader will encounter additional parameters $R$, $R'$ and $\lambda$. Many multiplicative constants, as well as the meaning of the phrase ``large enough", may depend on said parameters. For this reason we introduce an additional piece of terminology. 

\begin{definition}\label{defn-cs-4}
Let $\mathrm{Par}$ denote a set of real valued parameters. By constants of class $\C[\mathrm{Par}]$ (resp. $\C^+[\mathrm{Par}]$) we mean polynomial expressions in elements of $\mathrm{Par}$ whose coefficients are constants of class $\C$ (resp. $\C^+$). We say that a property holds for all $n$ which are large enough relative to $\C[\mathrm{Par}]$ (resp. $\C^+[\mathrm{Par}]$) if the property holds for all $n\ge N$ where $N$ is of class $\C[\mathrm{Par}]$ (resp. $\C^+[\mathrm{Par}]$). 
\end{definition}

For example, the statement of Theorem \ref{GHduhduhduh} we made above can be altered to include the claim that for each $\e>0$ there exists $N$ of the form $\tfrac{1}{\e}\C^+[R]$ such that  
$$d^{GH}((\U_{n,R},g_n), (B_{\geucl}(0,R),g))<\e$$
for all $n\ge N$.

\subsubsection{Our general result regarding Brill-Lindquist metrics}

The following is our main result regarding minimal surfaces of general Brill-Lindquist metrics. 

\begin{theorem}\label{minsurf:thm1-taos}\ 
\begin{enumerate}
\item Fix $p_i\in \P_{**}$. There exists a constant $C$ of class $\C(i)$ and a value of $\e>0$ which is small relative to $\C(i)$ with the following property: If $\frac{\alpha_i+\beta_i}{\sigma_i}<\e$ then there exists a function $\S_{i}:S^2\to (0,\infty)$ for which the image $\Sigma_{i}$ of 
$$\omega\mapsto p_i+\S_{i}(\omega)\omega$$
is a minimal surface for $g_{\BL}$. Furthermore, we have the following: 
\begin{itemize}
\item $\Sigma_{i}$ is located in the region 
$$\left((\hat{\chi}^{(i)}\hat{\psi}^{(i)})^{-1/2}-C\tfrac{\alpha_i+\beta_i}{\sigma_i}\right)\sqrt{\alpha_{i}\beta_{i}}\le |x-p_i|\le \left((\hat{\chi}^{(i)}\hat{\psi}^{(i)})^{-1/2}+C\tfrac{\alpha_i+\beta_i}{\sigma_i}\right)\sqrt{\alpha_{i}\beta_{i}}.$$
\medbreak
\item $\Sigma_{i}$ is the only minimal surface contained entirely within $|x-p_i|\le \sigma_i/C$.
\end{itemize}
\medbreak
\item Suppose $p_i\in \P_*\smallsetminus \P_{**}$. There exists a constant $C$ of class $\C(i)$ and a value of $\e>0$ which is small relative to $\C(i)$ with the following property: If $\frac{\alpha_i+\beta_i}{\sigma_i}<\e$ then there are no minimal surfaces which are contained entirely within $|x-p_i|\le \sigma_i/C$.
\end{enumerate}
\end{theorem}

We need to inform the reader that the uniqueness statement of this theorem can be improved in certain circumstances; Remark \ref{taos-afterthefact-remark1} has all the relevant details. Also note that we make absolutely no claim that the surfaces of Theorem \ref{minsurf:thm1-taos} are outermost minimal. In general these surfaces will \emph{not} be outermost minimal surfaces. The reader should perhaps contrast the situation in Figure \ref{fig2} with the situation depicted in Figure 3 of \cite{BL} where point sources in serious proximity of one another form a joint minimal surface. On the other hand, there are situations where it can be proven that surfaces of Theorem \ref{minsurf:thm1-taos} are indeed outermost minimal. For further information about one such situation the reader can consult \cite{SormaniStavrov}.

The proof of Theorem \ref{minsurf:thm1-taos} is subdivided between Sections \ref{MinSurfExist} and \ref{minsurf-sec}. In Section \ref{MinSurfExist} we explicitly construct the minimal surfaces by solving non-linear elliptic PDEs. We do so by observing that near point sources in $\P_{**}$ the geometry is well approximated by the Reissner-Nordstr\"om geometry, and by finding solutions of the relevant PDEs in the form of small perturbations of minimal surfaces of Reissner-Nordstr\"om geometry. Section \ref{minsurf-sec} is dedicated to the proof of uniqueness. To prove uniqueness we first construct foliations with well controlled sign of the mean curvature, and then use these foliations to narrow down potential locations of the minimal surfaces. Readers familiar with the subject of \cite{SormaniStavrov} surely notice that the methods presented here streamline several arguments regarding minimal surfaces in  \cite{SormaniStavrov}.

\subsubsection{Our results regarding Brill-Lindquist-Riemann sums}
In the case of Brill-Lindquist-Riemann sums the parameters $\hat{\chi}^{(i)}_n$ and $\hat{\psi}^{(i)}_n$ can be uniformly bounded by constants of class $\C$ (see Proposition \ref{BLRcontrollemma}), while the parameter $\frac{\alpha_{i,n}+\beta_{i,n}}{\sigma_{i,n}}$ can be estimated based on
$$\alpha_{i,n}+\beta_{i,n}\le (\|A\|+\|B\|)(D/n)^3+(\C(\alpha, A)+\C(\beta, B))/n^4;$$
ultimately we obtain
$$\frac{\alpha_{i,n}+\beta_{i,n}}{\sigma_{i,n}}\le \frac{C}{n^2}$$
for some constant $C$ of class $\C$. In particular, for a given $\e>0$ we can arrange that  
$$\frac{\alpha_{i,n}+\beta_{i,n}}{\sigma_{i,n}}<\e,$$ 
all provided $n$ is large enough relative to $\C$. Overall, Theorem \ref{minsurf:thm1-taos} applies to Brill-Lindquist-Riemann sums and produces Theorem \ref{minsurf:thm1} below. Any and all departures from the literal restatement of Theorem \ref{minsurf:thm1-taos} are due to Remark \ref{remark3.2forBLR} and the approximations 
$$\hat{\chi}^{(i)}_n\approx \chi(p_{i,n}) \text{\ \ and\ \ } \hat{\psi}^{(i)}_n\approx \psi(p_{i,n})$$ 
discussed in Proposition \ref{BLRcontrollemma}.

\begin{theorem}\label{minsurf:thm1}\ 
\begin{enumerate}
\item There exists a constant $C$ of class $\C$, such that for all $n$ which are large relative to $\C$ and all $i$ with $\alpha_{i,n}\beta_{i,n} \neq 0$ there is a function $\S_{i,n}:S^2\to (0,\infty)$ for which the image $\Sigma_{i,n}$ of 
$$\omega\mapsto p_i+\S_{i,n}(\omega)\omega$$
is a minimal surface for $g_n$. Furthermore, we have the following: 
\begin{enumerate}
\item $\Sigma_{i,n}$ is located in the region 
$$\sqrt{\alpha_{i,n}\beta_{i,n}}\left(\tfrac{1}{\sqrt{\chi(p_{i,n})\psi(p_{i,n})}}-\tfrac{C}{n}\right)\tfrac{D^3}{n^3}\le |x-p_i|\le \sqrt{\alpha_{i,n}\beta_{i,n}}\left(\tfrac{1}{\sqrt{\chi(p_{i,n})\psi(p_{i,n})}}+\tfrac{C}{n}\right)\tfrac{D^3}{n^3}.$$
\medbreak
\item $\Sigma_{i,n}$ is the only minimal surface contained entirely within $|x-p_{i,n}|\le D/(2n)$.
\end{enumerate}
\medbreak
\item Suppose that $n$ is large relative to $\C$ and that $\alpha_{i,n}\beta_{i,n}=0$ for some $i$. There are no minimal surfaces which are entirely contained within $|x-p_{i,n}|\le D/(2n)$.
\end{enumerate}
\end{theorem}

Given that the proof of existence of surfaces $\Sigma_{i,n}$ relies on their proximity to being Reissner-Nordstr\"om minimal spheres and given the uniqueness of $\Sigma_{i,n}$ as minimal surfaces near the point sources themselves, we consider them to be naturally associated to each point source in $\P_{n,**}$.
Once again, there is no reason to believe that in full generality the surfaces $\Sigma_{i,n}$ are outermost minimal surfaces (e.g the situation depicted in Figure 3 of \cite{BL}). 

\bigbreak

It is interesting to notice that even upon enforcing that $\P_{n,**}=\P_{n,*}$ the sets $\V_{n,R}$ might have unbounded diameter. As seen in the context of Figure \ref{fig1} a Reissner-Nordstr\"om ``neck" may be of arbitrarily long because its length is dictated by quantities such as 
$(\alpha+\beta)(1+|\ln(\alpha\beta)|)$. Likewise, the surfaces $\Sigma_{i,n}$ may be located at the end of what can be imagined as a very deep well. Such a circumstance can and indeed does make the $g_n$-diameter of a set such as $\V_{n,R}$ very large.

The diameter of $\V_{n,R}$ is studied in great detail in Section \ref{deepwells:section}. The main result in this regard (Lemma \ref{diameterlemma}) establishes that $\V_{n,R}$ is of uniformly bounded diameter as $n\to \infty$ if and only if the sequence 
\begin{equation}\label{dn:defn2}
\ell_n:=\max_i \tfrac{1}{D}(\alpha_{i,n}+\beta_{i,n})|\ln(\alpha_{i,n}\beta_{i,n}/D^2)|
\end{equation}
is bounded. 

Inspired by the content of Lemma \ref{diameterlemma} we make the following definition.

\begin{definition}\label{deepwellsdefinition}
A sequence of Brill-Lindquist-Riemann sums is said to have \emph{deep wells} if $\P_{n,**}\neq\P_{n,*}$ for some $n$ or if the sequence of quantities $\ell_n$ defined in \eqref{dn:defn2} is unbounded. Otherwise, we say that it does not have deep wells. 
\end{definition}

In many situations of interest sequences of Brill-Lindquist-Riemann sums have no deep wells. For example, when no charge is present (that is, when $\alpha_{i,n}=\beta_{i,n}$ for all $n$ and $i$) we have no deep wells. In fact, in that particular context we have 
\begin{equation}\label{limsupcondition}
\lim_{n\to \infty}\ell_n=0.
\end{equation}
The sequences for which \eqref{limsupcondition} holds are particularly well-behaved from the standpoint of convergence. It is for this reason that we find it worthwhile to make another definition.  

\begin{definition}\label{shortwellsdefinition}
A sequence of Brill-Lindquist-Riemann sums with no deep wells is said to have \emph{shallow wells} if it satisfies \eqref{limsupcondition}.
\end{definition}

Section \ref{deepwells:section} presents some examples of Brill-Lindquist-Riemann sums with neither deep nor shallow wells. The values of $\ell_n$ do depend on how we go about choosing the values of the parameters $\alpha_{i,n}$ and $\beta_{i,n}$. For example, presence of deep or shallow wells (or the lack there of) may well be tied to our choices of sample points $q_{i,n}$ in \eqref{commonsensechoice}. The overall lesson here is that the sequence $\ell_n$ can exhibit -- to put it politely -- very interesting behavior. 

While the liberty of choosing the parameters $\alpha_{i,n}$ and $\beta_{i,n}$ in any which way so long as the conditions of Definition \ref{BLR} are fulfilled could be seen as a contributing factor to the hard-to-control behavior of the sequences $\ell_n$, it can also be seen as a blessing. This is revealed in Proposition \ref{wlog?} of Section \ref{deepwells:section} in which we show that the parameters $\alpha_{i,n}$ and $\beta_{i,n}$ can always be chosen so that our sequence of Brill-Lindquist-Riemann sums has shallow wells. Here is one insight into Proposition \ref{wlog?}: should it be the case that $\limsup \ell_n\neq 0$ it would be so because of locations where functions $A$ and/or $B$ are turning to zero. Any sort of perturbation of $\alpha_{i,n}=0$ or $\beta_{i,n}=0$  -- which could be due to unavoidable errors in measurement or evaluation -- may lead to radically altered values of $\ell_{i,n}$ and $\ell_n$. For this reason values of $\alpha_{i,n}$ or $\beta_{i,n}$ which are dangerously close to zero should perhaps be treated as being unreliable ``as measured" and somehow rounded off so to not cause ``noise". The proof of Proposition \ref{wlog?} implements this idea. For more on the subject the reader is encouraged to explore  Section \ref{deepwells:section}. 

\bigbreak

We start our convergence studies by investigating Gromov-Hausdorff convergence\footnote{For the convenience of audience specializing in general relativity we begin Section \ref{GH:sec} with a brief review of Gromov-Hausdorff limit.}  of Brill-Lindquist-Riemann sums. The following theorem is our main result asserting Gromov-Hausdorff convergence.

\begin{theorem}\label{GHthm1}
Fix $R>\sqrt{3}D$. There exists a constant $C$ of class $\C^+[R]$ such that 
$$d^{GH}((\V_{n,R},g_n), (B_{\geucl}(0,R),g))<C(\tfrac{1}{n}+\ell_n)$$
for all $n$ which are large relative to $\C$. In particular, if the sequence of Brill-Lindquist-Riemann sums has shallow wells then it converges in the Gromov-Hausdorff sense to the set $(B_{\geucl}(0,R),g)$. \end{theorem}

A reader might feel comfortable with the idea that, on the basis of Proposition \ref{wlog?},  no generality is lost by assuming the sequence of Brill-Lindquist-Riemann sums has shallow wells. Should that be the case, the reader would find that Theorem \ref{GHthm1} accomplishes our stated goal. Otherwise, the estimate of Theorem \ref{GHthm1} begs the following question: what happens if the sequence of Brill-Lindquist-Riemann sums does not have shallow wells. Perhaps the best answer we can give is: ``it highly varies". Great many interesting examples can be constructed here but in the interests of brevity we present only two examples in full detail:

\begin{enumerate}
\item Section \ref{GHdne} contains an example where the sequence of Brill-Lindquist-Riemann sums does not have deep wells, but which does not converge in the Gromov-Hausdorff sense. 
\medbreak
\item Section \ref{dependencyonsamplepoints} contains an example where Gromov-Hausdorff limit exists but varies dependening of the choice of sample points $q_{i,n}$ in \eqref{commonsensechoice}.
\end{enumerate}

In situations when the sequence of Brill-Lindquist-Riemann sums has deep wells we are forced to manually enforce compactness. We do so by replacing the sets $\V_{n,R}$ with the geodesic balls $\V_{n,R,R'}$ of radius $R'\gg R$ in $\V_{n,R}$ centered at $0$:
$$\V_{n,R,R'}:=\{p\in \V_{n,R}\ \big{|}\ d_{(\V_{n,R},g_n)}(0,p)<R'\}.$$
Readers who go through the details of Sections \ref{GHdne} and \ref{dependencyonsamplepoints} will be able to create examples of Brill-Lindquist-Riemann sums with deep wells which converge and examples which do not converge in the Gromov-Hausdorff sense. We briefly touch upon this subject at the end of Section \ref{GH:sec}.

\bigbreak
Our next goal is to investigate \emph{intrinsic flat limits} of spaces $(\V_{n,R},g_n)$ under the assumption that the sequence of Brill-Lindquist-Riemann sums has no deep wells. The concept of the intrinsic flat limit is very deep and its brief survey is included at the beginning of Section \ref{IFL:sec} of this paper. The following theorem can be considered to be our main result in this portion of the article. 

\begin{theorem}\label{BIGMAMMA}\ 
\begin{enumerate}
\item If the sequence of Brill-Lindquist-Riemann sums has no deep wells, then  $\V_{n,R}$ equipped with the metric $g_n$ converges in the intrinsic flat sense as $n\to \infty$ to the set $B_{\geucl}(0,R)$ equipped with the metric $g$ is introduced in \eqref{defng}-\eqref{potentials1}.
\medbreak
\item  If the sequence of Brill-Lindquist-Riemann sums has deep wells, then the set $\V_{n,R,R'}$ equipped with the metric $g_n$ converges in the intrinsic flat sense as $n\to \infty$ to the set $B_{\geucl}(0,R)$ equipped with the metric $g$ is introduced in \eqref{defng}-\eqref{potentials1}.
\end{enumerate}
\end{theorem}

\begin{remark}\label{christinasthing}
Readers who are interested in purely geometric aspects of this work surely notice that several results along the lines of Corollary \ref{BS} can be formulated as immediate consequences of Theorems \ref{GHthm1} and  \ref{BIGMAMMA}. 
\end{remark}

The functions $A$ and $B$ which we have been using all along are tied to the mass-energy and charge of the dust cloud as in \eqref{chriscross2} and \eqref{chriscross3}; they are not the mass-energy and charge of the cloud per se but are related to them by a system of non-linear PDEs. We end this Introduction by addressing the question which is in many ways opposite to the one we considered thus far -- the question of whether every charged dust cloud satisfying DEC (see Definition \ref{matter}) can be discretized using the concept of Brill-Lindquist-Riemann sums. An immediate consequence of Theorem \ref{TobyThm} is that the answer to our question is affirmative. 

\begin{theorem}\label{FullCircleThm}
Consider the charged dust cloud $(\omegam, \omegae)$ as in Definition \ref{matter} and the corresponding solution 
$$(g,\vec{E})=\left((\chi\psi)^2\geucl,\ \grad_g\ln(\chi/\psi)\right)$$ 
of the constraints \eqref{davidconstraints}. Any sequence of Brill-Lindquist-Riemann sums $(\V_{n,R,R'}, g_n)$ corresponding to the functions 
$$A=-\tfrac{1}{4\pi}\Delta_{\geucl}\chi,\ \ B=-\tfrac{1}{4\pi}\Delta_{\geucl}\psi$$
converges in the intrinsic flat sense to $(B_{\geucl}(0,R), g)$.
\end{theorem}

\subsection{Conclusions}
Our article commences a study of convergence of discretized point-object configurations, which we call Brill-Lindquist-Riemann sums, towards a charged dust continuum from the perspective of relativistic initial data. We explain why we find applications of well-posedness results for the Einstein evolution equations or the applications of the Cheeger-Gromov theory unfeasible, and we explain why we find it necessary to first address the underlying manifolds $(\V_{n,R}, g_n)$ using methods of metric geometry and the like. We then offer a study of the Gromov-Hausdorff and the intrinsic flat limits of $(\V_{n,R}, g_n)$. We discover that Gromov-Hausdorff limit is only well-behaved in situations for which we coined the phrase ``shallow wells". In other situations the Gromov-Hausdorff limit may not exist or it may depend on the process of evaluation of mass and charge parameters (choice of sample points, rounding/measurement errors). Although shallow wells could be considered to be in some way generic (see Section \ref{wlog?sec}) there are good physical reasons, which we address shortly, to not focus solely on shallow wells. On the other hand, the convergence of the sequence of Brill-Lindquist-Riemann sums is much more straightforward under the intrinsic flat limit, always leading to the charged dust continuum we intuitively expect. One possible interpretation here is that the intrinsic flat limit is the most suitable limit to use in our context. The idea that the intrinsic flat limit might be better suited for applications within mathematical general relativity is not new: the work on stability of the rigidity portion of the Positive Mass Theorem, e.g \cite{LeeSormani}, has already demonstrated this. Finally, we show that ``every" charged dust cloud satisfying the dominant energy condition (see Definition \ref{matter}) can be discretized using sequences of Brill-Lindquist-Riemann sums. 

Studies of convergence of Lorentzian manifolds rooted in ideas of metric geometry or geometric measure theory are relatively new, and we are certainly not familiar with any completed work on the topic of convergence of relativistic initial data which would be compatible with such convergence of Lorentzian manifolds. Readers who are interested to learn more can start by looking up \cite{ABBA, Noldus} and references therein. In particular, it is unclear what if anything our work has to say about the convergence of \emph{initial data} $(\V_{n,R}, g_n, \vec{E}_n)$ towards $(B_{\geucl}(0,R), g, \vec{E})$.

That there indeed may be something to the convergence idea we outlined in our paper is evidenced by the examples of (static) extreme charged dust. These examples fit within our framework under the umbrella of $B\equiv 0$ and $\beta_{i,n}\equiv 0$. Brill-Lindquist-Riemann sums in this case are superpositions of Reissner-Nordstr\"om bodies in equilibrium which we saw earlier in relation to \eqref{MP}. Thus, in this specific case we have a very concrete expression 
$$\mathbf{g}_n= -\chi_n^{-2}dt^2+\chi_n^2\geucl$$
for the spacetime evolutions of Brill-Lindquist-Riemann sums 
$$(g_n,\vec{E}_n)=(\chi_n^2\geucl, \grad_{g_n}\ln(\chi_n))$$ viewed as relativistic initial data. By Theorem \ref{BIGMAMMA} the intrinsic flat limit of $(\V_{n,R,R'},g_n)$ is $(B_{\geucl}(0,R),g)$ while the ``naive" space-time limit of $\mathbf{g}_n$ seems to be 
$$\mathbf{g}= -\chi^2\geucl^{-2}dt^2+\chi^2\geucl.$$
(Investigating if this kind of spacetime limit is or is not compatible with the work of \cite{ABBA, Noldus} is an interesting topic for future research.) Here is the good news: one can manually verify that the metric $\mathbf{g}$ solves the Einstein evolution equations for initial data 
$$(g,\vec{E})=(\chi_n^2\geucl, \grad_{g_n}\ln(\chi_n))$$
and for matter modeled as electrostatic dust where the mass-energy and charge densities are given by  
$$\rho_{\mathrm{mass}}=\chi^{-3}A=\rho_{\mathrm{el}};$$
also compare with \eqref{chriscross2}. In fact, the metric $\mathbf{g}$ has already been identified in physics literature (e.g \cite{Xdust1}) as describing static charged dust (in equilibrium). 

The example of static charged dust suggests a possibility of a theorem where intrinsic flat convergence of initial data leads to a (suitably defined) convergence of the space-time evolutions. While it is unclear to us what kind of stability results for the Einstein equations might apply in situations when the topology of the underlying manifold itself is changing, and such is the case with topologies of $\V_{n,R}$ and $\V_{n,R,R'}$, the fact that we have at least some boundedness of the second derivatives of $g_n$ and some boundedness of $\mathrm{Ricci}(g_n)$ (see Section \ref{ConvSec} and \cite{KRS}) leaves us with some hope that stability results applicable to our framework might exist or be discovered. 

\subsection*{Acknowledgments}
The original idea behind this project goes back to informal conversations between I.S and Prof.\,C.\,Sormani. I.S is deeply grateful to Prof.\,C.\,Sormani for all the consultations and guidance over the years regarding the concept of the intrinsic flat limit. While many portions of this article have been a solo work of I.S, the material on the construction of minimal surfaces as graphs over spheres is coauthored with T.B. This one particular part of the research was funded by the 2016 John S. Rogers Science Research Program at Lewis \& Clark College.

\section{Minimal Surfaces Associated to Individual Point Sources -- Existence}\label{MinSurfExist}

The main intuition here is that, when one zooms in, the geometry near each location $p_i\in \P_*$ is more or less the same as a Reissner-Nordstr\"om geometry. Naively at least we may think of $\alpha_{i}\pm \beta_{i}$ as the effective mass and the electric charge of the Reissner-Nordstr\"om body located at $p_i$, and so naively we may expect a minimal surface at about $|x-p_i|=\sqrt{\alpha_{i}\beta_{i}}$. We inspect this idea more closely by means of the dilation 
$$\Phi_{i}:u\to p_i+\tau_{i}u \text{\ \ for\ \ } \tau_{i}:=\sqrt{\alpha_{i}\beta_{i}}$$
and an examination of the rescaled metric 
$$\tau_{i}^{-2}\Phi_{i}^*\,g_{\BL}.$$ 
It should be mentioned that this particular kind of zooming in, i.e rescaling, is essentially the same as the rescaling in the point-particle limit of Gralla and Wald \cite{sb, Gluing1}. In a sense this rescaling achieves non-dimensionalization, and the reader may benefit from thinking that the $u$-variable is non-dimensional.

\subsection{Approximation by a Reissner-Nordstr\"om metric}
The following lemma makes the approximation claims from above precise. To avoid notationally lengthy expressions we henceforth employ 
$$\chi_{i}:=\Phi_{i}^*\chi_{\BL}, \ \hat{\chi}^{(i)}:=\chi^{(i)}(p_i) \text{\ \ and\ \ } \psi_{i}:=\Phi_{i}^*\psi_{\BL},\ \hat{\psi}^{(i)}:=\psi^{(i)}(p_i).$$
Under such notational conventions we have 
$$\tau_{i}^{-2}\Phi_{i}^*\,g_{\BL}=(\chi_{i}\psi_{i})^2\geucl.$$

\begin{lemma}\label{metricconvergence-taos}
Restrict the domain of $\Phi_{i}$ to a fixed annulus centered at the origin and assume that, relative to the size of this annulus, we have $\alpha_i+\beta_i\ll\sigma_i$. 
We then also have 
\begin{itemize}
\item $\left\|\chi_{i} - \left(\hat{\chi}^{(i)}+\frac{\sqrt{\alpha_i/\beta_i}}{|u|}\right)\right\|_{L^\infty}=O\left(\frac{\tau_i}{\sigma_i}\right);$
\medbreak
\item If $l\ge 1$ then $\left\|\partial^l_u\,\chi_{i} - \partial^l_u\left(\hat{\chi}^{(i)}+\frac{\sqrt{\alpha_i/\beta_i}}{|u|}\right)\right\|_{L^\infty}=O\left(\left(\frac{\tau_i}{\sigma_i}\right)^l\right).$  
\end{itemize}
All the implied proportionality constants are of class $\C(i)$ and depend on $l$ and the choice of the annulus. 
\end{lemma}

\begin{remark}\label{oldconvergencestatement}
This is not entirely true, but there is a sense here that the metric $\tau_{i}^{-2}\Phi_{i}^*\,g_{\BL}$ ``converges" to the metric 
\begin{equation}\label{ReissNordLikeMetric}
\left(\hat{\chi}^{(i)}+\frac{\sqrt{\alpha_i/\beta_i}}{|u|}\right)^2\left(\hat{\psi}^{(i)}+\frac{\sqrt{\beta_i/\alpha_i}}{|u|}\right)^2 \geucl 
\end{equation} 
We note that in different coordinates \eqref{ReissNordLikeMetric} takes the exact form of a Reissner-Nordstr\"om metric. To be precise, \eqref{ReissNordLikeMetric} can be expressed as 
\begin{equation}\label{ReissNordLikeMetric-ver2}
\left(1+\frac{\hat{\psi}^{(i)}\sqrt{\alpha_i/\beta_i}}{|v|}\right)^2\left(1+\frac{\hat{\chi}^{(i)}\sqrt{\beta_i/\alpha_i}}{|v|}\right)^2 \geucl
\end{equation}
for $v=\hat{\chi}^{(i)}\hat{\psi}^{(i)}u$.
The metric \eqref{ReissNordLikeMetric-ver2} has a minimal surface at $|v|=\sqrt{\hat{\chi}^{(i)}\hat{\psi}^{(i)}}$, meaning that the metric \eqref{ReissNordLikeMetric} has a minimal surface at 
$$u=\left(\hat{\chi}^{(i)}\hat{\psi}^{(i)}\right)^{-1/2}\le 1.$$
\end{remark}

\begin{proof}
Consider the functions
$$\chi_{i}:=\Phi_{i}^*\chi_{\BL}=\Phi_{i}^*\chi^{(i)}+\frac{\sqrt{\alpha_i/\beta_i}}{|u|}.$$ 
As long as $\alpha_i+\beta_i\ll \sigma_i$ we have $\tau_{i}\ll \sigma_i/2$ and 
$\mathrm{Im}\left(\Phi_{i}\right)\subseteq \{x\big| |x-p_i|<\sigma_i/2\}$. 
In general, the $l$-th partial derivatives of $u\mapsto \frac{1}{|u|}$ can be expressed in the form $\frac{P_l}{|u|^{2l+1}}$ where $P_l$ is a homogeneous polynomial of degree $l$. Thus in particular we have 
$$\left|\partial_u^l\left(\frac{1}{|u|}\right)\right|\le \frac{C}{|u|^{l+1}}$$
for some universal constants $C$ which depend only on $l$. It further follows that 
$$\left|\partial_u^l\left(\frac{\alpha_{j}}{|p_i+\tau_{i}u-p_j|}\right)\right|\le C\left(\frac{\tau_{i}}{|p_i-p_j+\tau_{i}u|}\right)^l\cdot \frac{\alpha_{j}}{|p_i-p_j+\tau_{i}u|}.$$
Upon summation, and assuming $\alpha_i+\beta_i\ll \sigma_i$, we obtain an estimate of the form 
$$\left|\partial_u^l\Phi_{i}^*\chi^{(i)}\right|\le C\left(\frac{\tau_{i}}{\sigma_i}\right)^l\Phi_{i}^*\chi^{(i)}.$$
In particular, there exist constants $C\in \C(i)$ (depending on $l$) such that 
\begin{equation}\label{MetricConvergence-pt2}
\left|\partial_u^l\Phi_{i}^*\chi^{(i)}\right|\le C\left(\frac{\tau_{i}}{\sigma_i}\right)^l.
\end{equation}
It follows that 
\begin{equation}\label{MetricConvergence-pt1}
\left|\Phi_{i}^*\chi^{(i)} - \hat{\chi}^{(i)}\right|\le C\frac{\tau_i}{\sigma_i} \text{\ \ i.e\ \ }\left|\chi_{i} - \left(\hat{\chi}^{(i)}+\frac{\sqrt{\alpha_i/\beta_i}}{|u|}\right)\right|\le C\frac{\tau_i}{\sigma_i}
\end{equation}
for some constant $C\in\C(i)$ while the $l$-th derivatives of the function $\chi_{i}(u)$ are approximated by the corresponding derivatives of $\hat{\chi}^{(i)}+\frac{\sqrt{\alpha_i/\beta_i}}{|u|}$ at the rate of $O((\tau_{i}/\sigma_i)^l)$ with implied proportionality constants of class $\C(i)$. 
\end{proof}

\begin{remark} \label{metricconvergence-july2020-1}
For applications in Section \ref{TatyanaSection} below  note that 
\begin{equation}
\frac{\partial\chi_{i}}{\partial |u|} =-\frac{\sqrt{\alpha_i/\beta_i}}{|u|^2}+O\left(\frac{\tau_i}{\sigma_i}\right),
\end{equation}
where the constant implied in the $O$-term is of class $\C(i)$ and where the exponent within the $O$-term increases with any additional $\partial_u$ derivatives. In other words, the function $\frac{\partial  \chi_{i}}{\partial |u|}$ over compact annular domains converges uniformly to $-\frac{\sqrt{\alpha_i/\beta_i}}{|u|^2}$ with all the derivatives. 
\end{remark}

\subsection{The minimal surface equation}\label{TatyanaSection}

We seek a function $f:S^2\rightarrow (0,\infty)$ for which the image of the function 
$$F:S^2\rightarrow \R^3 \text{\ \ given by\ \ } F(\omega)=f(\omega)\omega$$ 
is a minimal surface for $\tau_{i}^{-2}\Phi_{i}^*\,g_{\BL}=(\chi_{i}\psi_{i})^2\geucl$.
To find the minimal surface equation for $f$ we use calculus of variations to optimize
$$\int_{S^2}(\chi_{i}\psi_{i})^2\big{|}_{\mathrm{Im}(F)}\,f\sqrt{f^2+|df|^2}\,\mathrm{dvol}_{S^2}.$$
We find that the minimal surface equation is
\begin{equation}\label{minsurf:eqn-fromnoahpaper}
\begin{aligned}
\Delta_{S^2}f-\tfrac{1}{f^2+|df|^2}\mathrm{Hess}f(\grad f, \grad f)-\left(2+\tfrac{|df|^2}{f^2+|df|^2}\right)f&\\
-\tfrac{2}{\chi_{i}\psi_{i}}\left(\tfrac{\partial (\chi_{i}\psi_{i})}{\partial f} f^2 -\langle d_{S^2}(\chi_{i}\psi_{i}), df\rangle\right)&=0
\end{aligned}
\end{equation}
where $d_{S^2}(\chi_{i}\psi_{i})$ denotes the pullback of $d(\chi_{i}\psi_{i})$ under $F$ to $S^2$.

The leading three terms of \eqref{minsurf:eqn-fromnoahpaper} capture the mean curvature\footnote{Indeed, we find it to be the multiple of $-f/\sqrt{f^2+|df|^2}$ and the stated expression.} of the graph of $f$ as a surface sitting inside the Euclidean $\mathbb{R}^3$. The remaining terms in \eqref{minsurf:eqn-fromnoahpaper} reflect the presence of the conformal factor which, in our case, increasingly becomes like that of a Reissner-Nordstr\"om metric (cf.\,\eqref{MetricConvergence-pt1}). We proceed by examining this approximation claim more closely. 

Consider the expression 
$$\lambda_{i}(f):=-\frac{2}{\chi_{i}\psi_{i}}\frac{\partial (\chi_{i}\psi_{i})}{\partial f} f^2=-2\left(\frac{\partial \chi_{i}/\partial f}{\chi_{i}}+\frac{\partial \psi_{i}/\partial f}{\psi_{i}}\right)f^2,$$
which for each given $f$ is a function on $S^2$. Under the assumptions that $\alpha_i+\beta_i\ll\sigma_i$ the approximations of Lemma \ref{metricconvergence-taos} suggest the proximity of $\lambda_{i}(f)$ to 
$$\hat{\lambda}_{i}(f):=2\left(\frac{\sqrt{\alpha_i/\beta_i}}{\hat{\chi}^{(i)}+\frac{\sqrt{\alpha_i/\beta_i}}{f}}+\frac{\sqrt{\beta_i/\alpha_i}}{\hat{\psi}^{(i)}+\frac{\sqrt{\beta_i/\alpha_i}}{f}}\right).$$
The properties of and the relationship between the functions $\lambda_{i}$
and $\hat{\lambda}_{i}$ are established in the next lemma. 

\begin{lemma}\label{propertiesoflambda}
\ 
\begin{enumerate}
\item The values of $\lambda_{i}$ and $\hat{\lambda}_{i}$ and all of their derivatives with respect to $f$ and $\omega\in S^2$ are bounded so long as the range of $f$ is in a fixed compact subset of $(0,\infty)$. Furthermore, in that case we have
$$\left\|\frac{\partial^l\lambda_{i}}{\partial f^l}-\frac{\partial^l\hat{\lambda}_{i}}{\partial f^l}\right\|_{L^\infty}=O\left(\frac{\tau_i}{\sigma_i}\right).$$
\medbreak
\item For each compact subset $K\subseteq (0,\infty)\times S^2$ and each multi-index $\mu$ we have an estimate of the form
$$\|\partial^\mu\lambda_{i}-\partial^\mu\hat{\lambda}_{i}\|_{L^\infty(K)}=O\left(\frac{\tau_i}{\sigma_i}\right).$$
The implied proportionality constant is of class $\C(i)$. 
\end{enumerate}
\end{lemma}

\begin{proof}
It suffices to explore the expressions 
$$f^2\frac{\partial \chi_{i}/\partial f}{\chi_{i}}=\frac{-\sqrt{\alpha_i/\beta_i}+f^2\,\frac{\partial \chi^{(i)}}{\partial f}}{\chi^{(i)}+\frac{\sqrt{\alpha_i/\beta_i}}{f}}$$ 
and their ``limiting" counterparts
$$\frac{-\sqrt{\alpha_i/\beta_i}}{\hat{\chi}^{(i)}+\frac{\sqrt{\alpha_i/\beta_i}}{f}}=\frac{-\sqrt{\alpha_i/\beta_i}\,f}{\hat{\chi}^{(i)}\,f+\sqrt{\alpha_i/\beta_i}}.$$
The very last expression is clearly bounded by $\|f\|_{L^\infty}$. 
Furthermore, its derivatives with respect to the $f$-variable can be directly computed: 
$$\left(\frac{\sqrt{\alpha_i/\beta_i}}{\hat{\chi}^{(i)}\,f+\sqrt{\alpha_i/\beta_i}}\right)^2,\ -2\left(\frac{\sqrt{\alpha_i/\beta_i}}{\hat{\chi}^{(i)}\,f+\sqrt{\alpha_i/\beta_i}}\right)^2\frac{\hat{\chi}^{(i)}}{\hat{\chi}^{(i)}\,f+\sqrt{\alpha_i/\beta_i}},\ ...$$
By the nature of these expressions an inductive argument can be constructed to show that so long as the range of $f$ is in a fixed compact subset of $(0,\infty)$ the values of $\hat{\lambda}_{i,n}$ and all of its derivatives with respect to $f$ are bounded. 

The quality of approximating $\lambda_{i}$ by  $\hat{\lambda}_{i}$ can be examined through perspective of the variation formula
$$\delta \left(\tfrac{Y}{Z}\right)=\tfrac{\delta Y}{Z}-\tfrac{Y}{Z} \tfrac{\delta Z}{Z} \text{\ \ at\ \ } (Y,Z)=(-\sqrt{\alpha_i/\beta_i}, \hat{\chi}^{(i)}+\tfrac{\sqrt{\alpha_i/\beta_i}}{f}).$$ Lemma \ref{metricconvergence-taos} and Remark \ref{metricconvergence-july2020-1} in essence claim that in our situation we have 
$$\delta Y,\delta Z=O(\tau_i/\sigma_i).$$ 
That approximating $\lambda_{i}$ by $\hat{\lambda}_{i}$ is at least as good as $O(\frac{\tau_i}{\sigma_i})$ in $L^\infty$ is now a consequence of the fact that in our situation we have $Z\ge 1$ as well as boundedness of $Y/Z$. Schematically speaking, the same type of reasoning extends to 
$$\delta \left((\tfrac{Y}{Z})'\right)=\delta\left(\tfrac{Y'}{Z}\right)-\delta\left(\tfrac{Y}{Z}\right)\tfrac{Z'}{Z}-\tfrac{Y}{Z}\delta\left(\tfrac{Z'}{Z}\right)$$
and all the higher order derivatives. Consequently, approximating any of the derivatives of $\lambda_{i}$ by  $\hat{\lambda}_{i}$ is at least as good as $O(\frac{\tau_i}{\sigma_i})$. This applies both to derivatives with respect to $f$ and with derivatives with respect to spherical $\omega$-variables so long as we are restricted to domains of $f$ which are compact subsets of $(0,\infty)$.
\end{proof}
 
Reasoning analogous to the one just presented also shows that 
$$\xi_{i}(f):=-\frac{2}{\chi_{i}\psi_{i}}d_{S^2}(\chi_{i}\psi_{i})=-2\left(\frac{d_{S^2} \chi_{i}}{\chi_{i}}+\frac{d_{S^2}\psi_{i}}{\psi_{i}}\right)$$
satisfies the following.

\begin{lemma}\label{propertiesofxi}
For each compact subset $K\subseteq (0,\infty)\times S^2$ and each multi-index $\mu$ we have an estimate of the form
$$\|\partial^\mu\xi_{i}\|_{L^\infty(K)}=O(\tau_i/\sigma_i).$$
The implied proportionality constant is of class $\C(i)$. 
\end{lemma}

In summary, we see that the minimal surface equation takes the form of
\begin{equation}\label{MinSurf:eqn}
\Delta_{S^2}f-\tfrac{1}{f^2+|df|^2}\mathrm{Hess}f(\grad f, \grad f)-\left(2+\tfrac{|df|^2}{f^2+|df|^2}\right)f + \langle\xi_{i}(f),df\rangle+\lambda_{i}(f)=0,
\end{equation}
where, under the assumption of $\tau_i\ll\sigma_i$, 
\begin{itemize}
\item the $1$-form $\xi_{i}(f)$ over $S^2$ is well approximated by $0$ together with all the derivatives as spelled out in Lemma \ref{propertiesofxi}.
\medbreak
\item the function $\lambda_{i}(f)$ is approximated by the function 
$\hat{\lambda}_{i}(f)$ with all the derivatives in the sense spelled out in Lemma \ref{propertiesoflambda}.
\end{itemize}
As such, the minimal surface equation is approximated by 
\begin{equation}\label{minsurfapproximate}
\Delta_{S^2}f-\tfrac{1}{f^2+|df|^2}\mathrm{Hess}f(\grad f, \grad f)-\left(2+\tfrac{|df|^2}{f^2+|df|^2}\right)f+\hat{\lambda}_{i}(f)=0,
\end{equation} 
which in turn is the minimal surface equation for the metric \eqref{ReissNordLikeMetric}. The solution of the latter is the constant function 
$$\hat{f}_{i}=\left(\hat{\chi}^{(i)}\hat{\psi}^{(i)}\right)^{-1/2}.$$
It is to be expected that the minimal surface equation \eqref{MinSurf:eqn} has a solution which is approximately equal to $\hat{f}_{i}$. 

Informed by this situation we employ a substitution 
$$f=\hat{f}_{i}e^h.$$
This substitution renormalizes the minimal surface equation so that the approximate solution to work around in $h=0$. Direct computation leads to 
\begin{equation}\label{minsurfeqn-taos}
\mathcal{T}_{i}h:=\Delta_{S^2}h-\tfrac{1}{1+|dh|^2}\mathrm{Hess}\,h\,(\grad \,h, \grad\, h)+\langle\Xi_{i}(h),dh\rangle+\Lambda_{i}(h)-2=0,
\end{equation}
where the relationship between the $1$-forms $\xi_{i}$ and $\Xi_{i}$, and the functions $\lambda_{i}$ and $\Lambda_{i}$ is as follows:
$$\Xi_{i}(h)=\xi_{i}(\hat{f}_{i}e^h),\ \ \Lambda_{i}(h)=\frac{\lambda_{i}(\hat{f}_{i}e^h)}{\hat{f}_{i}e^h}.$$
In addition, we introduce the function $\hat{\Lambda}_{i}(h)=\hat{\lambda}_{i}(\hat{f}_{i}e^h)/(\hat{f}_{i}e^h)$. The idea once again is to take the advantage of the proximity of $\Lambda_{i}$ to $\hat{\Lambda}_{i}$ in order to squeeze information about $\Lambda_{i}$.

\begin{lemma}\label{PropertiesofLambda}
\ 
\begin{enumerate}
\item The values of $\Lambda_i$ and $\hat{\Lambda}_{i}$ and all of their derivatives with respect to $h$ are bounded so long as the range of $h$ is in a fixed compact subset of $\mathbb{R}$. Furthermore, in that case we have
$$\left\|\frac{\partial^l\Lambda_{i}}{\partial h^l}-\frac{\partial^l\hat{\Lambda}_{i}}{\partial h^l}\right\|_{L^\infty}=O\left(\frac{\tau_i}{\sigma_i}\right).$$
\medbreak
\item For each compact subset $K\subseteq (0,\infty)\times S^2$ and each multi-index $\mu$ we have an estimate of the form
$$\|\partial^\mu\Lambda_{i}-\partial^\mu\hat{\Lambda}_{i}\|_{L^\infty(K)}=O\left(\frac{\tau_i}{\sigma_i}\right).$$
The implied proportionality constant is of class $\C(i)$. 
\medbreak
\item For each compact set $K\subseteq \mathbb{R}\times S^2$ there exists $\e>0$ which is small relative to control variables of class $\C(i)$ and a constant $C$ of class $\C(i)$ such that whenever $\frac{\alpha_i+\beta_i}{\sigma_i}<\e$ we have
$$\frac{\partial\Lambda_{i}}{\partial h},\ \frac{\partial\hat{\Lambda}_{i}}{\partial h}\le -\frac{\tau_i}{C(\alpha_i+\beta_i)}$$
for all $(h,\omega)\in K$.
\medbreak
\item To each fixed compact subset $K\subseteq \mathbb{R}$ and each $l\ge 1$ we can associate a constant $C$ of class $\C(i)$ for which 
$$\left|\frac{\partial^l \Lambda_{i}}{\partial h^l}\right|, \left|\frac{\partial^l \hat{\Lambda}_{i}}{\partial h^l}\right|\le C\frac{\tau_i}{\alpha_i+\beta_i}$$
whenever $h\in K$.
\end{enumerate}
\end{lemma}

\begin{proof}
By assumption there exists $C$ of class $\C(i)$ such that 
$$1\le \hat{\chi}^{(i)},\hat{\psi}^{(i)}\le C \text{\ \ and thus\ \ } \tfrac{1}{C}\le \hat{f}_{i}\le 1.$$ 
Consequently, the mappings 
$$h\mapsto \hat{f}_{i}e^h,\ \ \ h\mapsto \frac{1}{\hat{f}_{i}e^h}$$
over compact domains for $h$ are bounded. Thus, our first two claims are simply a consequence of Lemma \ref{propertiesoflambda}. We proceed by studying the derivatives of $\hat{\Lambda}_{i}$ and $\Lambda_{i}$ more explicitly. Direct computation shows that 
$$\frac{\partial \hat{\Lambda}_{i}}{\partial h}=-2\sqrt{\alpha_i\beta_i\hat{\chi}^{(i)}\hat{\psi}^{(i)}}\left(\frac{e^h}{\left(e^h\sqrt{\alpha_i\hat{\psi}^{(i)}}+\sqrt{\beta_i\hat{\chi}^{(i)}}\right)^2}+\frac{e^h}{\left(e^h\sqrt{\beta_i\hat{\chi}^{(i)}}+\sqrt{\alpha_i\hat{\psi}^{(i)}}\right)^2}\right)$$
and that for all subsequent derivatives with respect to $h$ we have estimates of the form 
$$\left|\frac{\partial^l \hat{\Lambda}_{i}}{\partial h^l}\right|\le C\sqrt{\alpha_i\beta_i\hat{\chi}^{(i)}\hat{\psi}^{(i)}}\left(\frac{e^h}{\left(e^h\sqrt{\alpha_i\hat{\psi}^{(i)}}+\sqrt{\beta_i\hat{\chi}^{(i)}}\right)^2}+\frac{e^h}{\left(e^h\sqrt{\beta_i\hat{\chi}^{(i)}}+\sqrt{\alpha_i\hat{\psi}^{(i)}}\right)^2}\right)$$
for some universal constants $C$ depending on $l$.
Thus, for each fixed compact subset $K\subseteq \mathbb{R}\times S^2$ one can find a constant $C$ of class $\C(i)$ for which 
$$\frac{\partial \hat{\Lambda}_{i}}{\partial h}\le -\frac{\tau_i}{C(\alpha_i+\beta_i)},\ \ (h,\omega)\in K.$$
We may assume that $(\alpha_i+\beta_i)/\sigma_i\ll 1$, as dictated by control variables of class $\C(i)$, and so it can be arranged that 
$$\frac{\tau_i}{\sigma_i}\ll \frac{\tau_i}{C(\alpha_i+\beta_i)}.$$
Next, take into account the fact that  
$$\left|\frac{\partial \hat{\Lambda}_{i}}{\partial h}- \frac{\partial \Lambda_{i}}{\partial h}\right|=O(\tau_i/\sigma_i).$$
It follows that, after potentially increasing the value of $C$ and imposing even stronger restrictions on $(\alpha_i+\beta_i)/\sigma_i\ll 1$, we have 
$$\frac{\partial \Lambda_{i}}{\partial h}\le -\frac{\tau_i}{C(\alpha_i+\beta_i)},\ \ (h,\omega)\in K.$$
For the same reasons, to each fixed compact subset $K\subseteq \mathbb{R}$ and each $l\ge 1$ we can associate a constant $C$ of class $\C(i)$ for which 
$$\left|\frac{\partial^l \hat{\Lambda}_{i}}{\partial h^l}\right|\le C\frac{\tau_i}{\alpha_i+\beta_i}$$
whenever $h\in K$. This completes our proof.
\end{proof}

For the record, we also have the following properties of the $1$-form 
$\Xi_{i}$. They are consequences of Lemma \ref{propertiesofxi} and the Chain Rule. 

\begin{lemma}\label{PropertiesofXi}
For each smooth $h$ the $1$-form $\Xi_{i}(h)$ is smooth and satisfies 
$$\|\Xi_{i}(h)\|_{H^l(S^2)}\le C\tfrac{\tau_i}{\sigma_i}.$$
In particular, we have:
\begin{enumerate}
\item $\|\langle \Xi_{i}(h), dh\rangle\|_{H^l(S^2)}\le C\tfrac{\tau_i}{\sigma_i}\|dh\|_{H^l(S^2)}$.
\medbreak
\item $\|\langle \Xi_{i}(h_1), dh_1\rangle - \langle \Xi_{i}(h_2), dh_2\rangle\|_{H^l(S^2)}\le C\tfrac{\tau_i}{\sigma_i}\|h_1-h_2\|_{H^{l+1}(S^2)}$.
\end{enumerate}
The constants $C$ are of class $\C(i)$, dependent on $l$. Furthermore, assuming that $l\gg1$ so that $H^l$-spaces are an algebra under multiplication the constants can be chosen uniformly across all $h$, $h_1$ and $h_2$ from a fixed ball in $H^l(S^2)$.
\end{lemma}

To be honest, the last estimate can be replaced with a slightly stronger one:
$$\|\langle \Xi_{i}(h_1), dh_1\rangle - \langle \Xi_{i}(h_2), dh_2\rangle\|
\le C\tfrac{\tau_i}{\sigma_i}\left((\|dh_1\|+\|dh_2\|)\|h_1-h_2\|+\|d(h_1-h_2)\|\right)$$
but we have not been able to make use of this stronger estimate.
All the norms here are $H^l(S^2)$-norms with $l\gg1$.

\subsection{The associated semi-linear problem} In the discussion above we hinted at using $h=0$ as an approximate solution of the minimal surface equation \eqref{minsurfeqn-taos}. Although this indeed is an option, it is possible to prove a stronger result\footnote{From our experience staying within the control class $\C(i)$ is at least very difficult if not impossible to achieve by using $h=0$ as an approximate solution.} regarding the location of minimal surfaces by using an improved approximate solution. Consider the semi-linear problem 
\begin{equation}\label{semilinearproblem}
\mathcal{R}_{i}h:=\Delta_{S^2}h +\Lambda_{i}(h)-2=0.
\end{equation}
The improved approximate solution we are alluding to is the solution $h_{i}$ of $\mathcal{R}_{i}h=0$ addressed in the following lemma.

\begin{lemma}\label{semilinearlemma}
There exists $\e>0$ which is small relative to control variables of class $\C(i)$ such that whenever $\frac{\alpha_i+\beta_i}{\sigma_i}<\e$ there exists a smooth solution $h_{i}$ of \eqref{semilinearproblem} with  
$$\|h_{i}\|_{L^\infty(S^2)} =O\left(\frac{\alpha_i+\beta_i}{\sigma_i}\right) \text{\ \ and\ \ } \|dh_{i}\|_{H^l(S^2)}=O\left(\frac{\tau_i}{\sigma_i}\right)$$
for each $l$. The implied proportionality constants are of class $\C(i)$.
\end{lemma}

\begin{proof}
The existence of $h_{i}$ is established by means of the method of sub- and super-solutions. Since $f=\hat{f}_{i}$ solves \eqref{minsurfapproximate} we have that 
$$\hat{\Lambda}_{i}(h)\big{|}_{h=0}=2.$$
Lemma \ref{PropertiesofLambda} now implies that 
$$\Lambda_{i}(0)=2+O(\tau_i/\sigma_i).$$
In addition, the estimate on $\frac{\partial \Lambda_{i}}{\partial h}$ provided in Lemma \ref{PropertiesofLambda} ensures that for some constant $C$ of class $\C(i)$ we have 
$$\Lambda_{i}\left(-C\tfrac{\alpha_i+\beta_i}{\sigma_i}\right)-2>0 \text{\ \ and\ \ } \Lambda_{i}\left(C\tfrac{\alpha_i+\beta_i}{\sigma_i}\right)-2<0.$$
In particular, constants $h_{-}=-C\tfrac{\alpha_i+\beta_i}{\sigma_i}$ and $h_{+}=C\tfrac{\alpha_i+\beta_i}{\sigma_i}$ serve as sub- and super-solutions of \eqref{semilinearproblem}. The existence of a smooth solution $h_{i}$ with 
$$-C\tfrac{\alpha_i+\beta_i}{\sigma_i}<h_{i}<C\tfrac{\alpha_i+\beta_i}{\sigma_i}$$
now follows from, for example, \cite{jim}. 

To prove the estimates on $dh_{i}$ observe that 
\begin{equation}\label{forbootstrapping}
\Delta_{S^2}h_{i}=2-\Lambda_{i}(h_{i})=\left(\hat{\Lambda}_{i}(0)-\hat{\Lambda}_{i}(h_{i})\right)+\left(\hat{\Lambda}_{i}(h_{i})-\Lambda_{i}(h_{i})\right).
\end{equation}
It follows from Lemma \ref{PropertiesofLambda} that the latter is on the order of 
$O\left(\frac{\tau_i}{\sigma_i}\right)$ in $L^2(S^2)$. 
Working within the orthogonal complement of the subspace of constant functions in $L^2(S^2)$, i.e orthogonally to the kernel of $\Delta_{S^2}$, we see that 
$$\|h_{i}-\overline{h}_{i}\|_{L^2(S^2)}\le C \frac{\tau_i}{\sigma_i}$$
where $\overline{h}_{i}$ denotes the average value of $h_{i}$ and where $C$ is of class $\C(i)$. The Elliptic Regularity Estimate for the operator $\Delta_{S^2}$ now implies 
\begin{equation}\label{middleofbootstrapping}
\|h_{i}-\overline{h}_{i}\|_{H^2(S^2)}\le C \frac{\tau_i}{\sigma_i} \text{\ \ and consequently\ \ } \|dh_{i}\|_{H^1(S^2)} \le C \frac{\tau_i}{\sigma_i}.
\end{equation}
The plan now is to bootstrap further using Elliptic Regularity. Given the estimates on $\frac{\partial}{\partial h}\hat{\Lambda}_i$ and the derivatives of $\hat{\Lambda}_{i}-\Lambda_{i}$ presented in Lemma \ref{PropertiesofLambda} we see that the $H^1$-norm of the right hand side of \eqref{forbootstrapping} is bounded by a $\C(i)$-multiple of 
$$\frac{\tau_i}{\alpha_i+\beta_i}\|dh_{i}\|_{L^2(S^2)}+\frac{\tau_i}{\sigma_i}\le C\frac{\tau_i}{\sigma_i}.$$
Elliptic Regularity now implies the improved version of \eqref{middleofbootstrapping} in which $H^2$ and $H^1$ norms are replaced by $H^3$ and $H^2$ norms, respectively. 
Next we estimate the $H^2$-norm of the right hand side of \eqref{forbootstrapping}. Due to the fact that we are in dimension $3$, the Sobolev Embedding gives a bound which is a $\C(i)$-multiple of 
$$\frac{\tau_i}{\alpha_i+\beta_i}(\|\mathrm{Hess}\,h_{i}\|_{L^2}+\|dh_{i}\|^2_{H^2}+\|dh_{i}\|_{L^2})+\frac{\tau_i}{\sigma_i}\le C\frac{\tau_i}{\sigma_i}.$$
This ultimately leads to the improved version of \eqref{middleofbootstrapping} in which $H^2$ and $H^1$ norms are replaced by $H^4$ and $H^3$ norms, respectively. From here on an inductive argument based on the Sobolev Embedding can be constructed to show that 
$$\|dh_{i}\|_{H^l(S^2)}=O\left(\frac{\tau_i}{\sigma_i}\right)$$ for all $l$. 
\end{proof}

Under the assumption on small size of $dh_{i}$ the Hessian term 
$$\mathscr{H}(h):=\tfrac{1}{1+|dh|^2}\mathrm{Hess}\,h(\grad\,h, \grad\,h)$$
from  \eqref{minsurfeqn-taos} contributes very little -- at least as long as we don't deviate much from $h_{i}$. This is evidenced by the inequalities of the form 
$$\begin{aligned}
&\|\mathscr{H}(h)\|_{H^l(S^2)}\le C \|dh\|^3_{H^{l+1}(S^2)}\\
&\|\mathscr{H}(h_1)-\mathscr{H}(h_2)\|_{H^l(S^2)}\le C (\|dh_1\|^2_{H^{l+1}(S^2)}+\|dh_2\|^2_{H^{l+1}(S^2)})\|h_1-h_2\|_{H^{l+2}(S^2)}
\end{aligned}$$
which apply so long as $h$, $h_1$ and $h_2$ are from a fixed ball in $H^{l+2}(S^2)$ while $l\gg 1$. 
By Lemma \ref{PropertiesofXi} similar estimates apply to the inner-product term in \eqref{minsurfeqn-taos}. 

The overall idea here is that on small neighborhoods of $h_{i}$ the operator $\mathcal{R}_{i}$ approximates the operator $\mathcal{T}_{i}$ of \eqref{minsurfeqn-taos} very well. We record this observation in the precise form needed later.

\begin{lemma}\label{approxlemma}
Define
$$\mathcal{E}_{i}:=\mathcal{R}_{i}-\mathcal{T}_{i}.$$
 For each $l\gg 1$ there exists a constant $C$ of class $\C(i)$ such that  
\begin{enumerate}
\item $\left\|\mathcal{E}_{i}(h)\right\|_{H^l(S^2)}\le C\left(\frac{\tau_i}{\sigma_i}+\nu\right)^3+C\frac{\tau_i}{\sigma_i}\left(\frac{\tau_i}{\sigma_i}+\nu\right)$
\medbreak
\item $\left\|\mathcal{E}_{i}(h_1)-\mathcal{E}_{i}(h_2)\right\|_{H^l(S^2)}\le C\left(\left(\frac{\tau_i}{\sigma_i}+\nu\right)^2+\frac{\tau_i}{\sigma_i}\right)\|h_1-h_2\|_{H^{l+2}(S^2)}$
\end{enumerate}
for all $h$, $h_1$ and $h_2$ from a fixed ball in $H^{l+2}(S^2)$ with 
$$dh,\ dh_1,\ dh_2\in B_\nu(dh_{i})\subseteq H^{l+1}(S^2).$$ 
\end{lemma}

Our strategy for solving $\mathcal{T}_{i}h=0$ is to solve the equivalent equation 
$$\mathcal{R}_{i}h = \mathcal{E}_{i}h$$
on a small ball centered at $h_{i}$. We do this by means of linearization of $\mathcal{R}_{i}$ at $h_{i}$.

\subsection{Linearization of the operator $\mathcal{R}_{i}$ at $h_{i}$}
We now study the linearization of the operator $\mathcal{R}_{i}$ at $h_{i}$. To do so we let $h= h_{i}+\e k$, and take the formal derivative of $\mathcal{R}_{i}(h)$ with respect to $\e$ at $\e = 0$. We arrive at the expression  
$$\mathcal{L}_{i} k:=\Delta_{S^2} k-c_{i}k$$
where, by Lemma \ref{PropertiesofLambda}, the constant function $c_{i}:=-\frac{\partial\Lambda_{i}}{\partial h}\Big{|}_{h=h_{i}}$ satisfies 
$$c_{i}\ge \frac{\tau_i}{C(\alpha_i+\beta_i)}$$
for some $C$ of class $\C(i)$. 

The operator $\mathcal{L}_{i}$ is self-adjoint elliptic and its kernel, by the Maximum Principle, is trivial. It follows that $\mathcal{L}_{i}$, viewed as an operator from $H^{l+2}(S^2)$ to $H^l(S^2)$, is invertible. The following lemma controls the norm of the inverse. 

\begin{lemma}\label{linftyinvertible}
There exists, for each $l$, a constant $C$ of class $\C(i)$ such that 
$$\|k\|_{H^{l+2}(S^2)}\le \frac{C(\alpha_i+\beta_i)}{\tau_i}\,\|\mathcal{L}_{i} k\|_{H^l(S^2)}.$$
\end{lemma}

\begin{proof}
It suffices to prove the estimate of the form  
\begin{equation}\label{miniinvertibilitylemma}
\|k\|_{H^{l+2}(S^2)}\le C\left(1+\frac{1}{c_{i}}\right)\|\mathcal{L}_{i} k\|_{H^l(S^2)};
\end{equation}
we do so using induction on $l$. By the Elliptic Regularity there exists a universal constant $C$ such that 
$$\|k\|_{H^{2}(S^2)}\le C\left(\|\Delta_{S^2} k\|_{L^2(S^2)}+\|k\|_{L^2(S^2)}\right)\le C\left(\|\mathcal{L}_{i} k\|_{H^2(S^2)}+(c_{i}+1)\|k\|_{L^2(S^2)}\right).$$
A direct examination of the eigenvalues of $\mathcal{L}_{i}$ shows that 
\begin{equation}\label{liz}
\|k\|_{L^2(S^2)}\le \tfrac{1}{c_{i}} \|\mathcal{L}_{i} k\|_{L^2(S^2)},
\end{equation} 
which in turn further implies
$$\|k\|_{H^2(S^2)}\le C(2+\tfrac{1}{c_{i}})\|\mathcal{L}_{i} k\|_{L^2(S^2)}.$$
Absorbing the factor of $2$ into the constant $C$ completes the proof of the base case. 

For the induction step assume the estimate of the form \eqref{miniinvertibilitylemma} and consider the Elliptic Regularity Estimate
$$\begin{aligned}
\|k\|_{H^{l+3}(S^2)}\le &C\left(\|\Delta_{S^2} k\|_{H^{l+1}(S^2)}+\|k\|_{L^2(S^2)}\right)\\
\le &C\left(\|\mathcal{L}_{i} k\|_{H^{l+1}(S^2)}+c_{i}\|k\|_{H^{l+1}(S^2)}+\|k\|_{L^2(S^2)}\right).
\end{aligned}$$
Since $c_{i}\,C(2+\tfrac{1}{c_{i}})=C(1+2c_{i})\le 3C$ the induction hypothesis implies
$$c_{i}\|k\|_{H^{l+1}(S^2)}\le 3C\|\mathcal{L}_{i} k\|_{H^l(S^2)}.$$
Overall, we have 
$$\|k\|_{H^{l+3}(S^2)}\le C\left((1+3C)\|\mathcal{L}_{i} k\|_{H^{l+1}(S^2)}+\|k\|_{L^2(S^2)}\right).$$
Applying \eqref{liz} and increasing the value of $C$ to $C(1+3C)$ proves the induction step. 
\end{proof}

\subsection{The quadratic error term $\mathcal{Q}_{i}$}
We continue by analyzing the error term $\mathcal{Q}_{i}$ defined by 
$$\mathcal{R}_{i}(h) = \mathcal{L}_{i}(h-h_{i})+\mathcal{Q}_{i}(h).$$
\begin{lemma}\label{Qn:eqn}
For each $l\gg 1$ there exists a constant $C$ of class $\C(i)$ independent of $\nu<1$ such that 
$$\|h_1-h_{i}\|_{H^{l+2}(S^2)}, \|h_2-h_{i}\|_{H^{l+2}(S^2)}<\nu$$
implies
$$\|\mathcal{Q}_{i}(h_1)-\mathcal{Q}_{i}(h_2)\|_{H^l(S^2)}\le C\nu\|h_1-h_2\|_{H^{l+2}(S^2)}.$$
In particular, we have 
$$\|\mathcal{Q}_{i}(h)\|_{H^l(S^2)}\le C\nu^2$$
for all $h$ with $\|h-h_{i}\|_{H^{l+2}(S^2)}\le \nu$.
\end{lemma}
\begin{proof}
Recall that the functions $\Lambda_{i}$ have bounded derivatives (cf. Lemma \ref{PropertiesofLambda}). It then follows from 
$$\mathcal{Q}_{i}\big{|}_{h=h_{i}}=\left(\partial_h \mathcal{Q}_{i}\right)\big{|}_{h=h_{i}}=0$$ 
that there is a constant $C$ uniform in $\nu<1$ and $\omega \in S^2$ such that 
$$\big{|}\mathcal{Q}_{i}(h_1, \omega)-\mathcal{Q}_{i}(h_2, \omega)\big{|}\le C\nu\big{|}h_1-h_2\big{|},$$
for all $h_1$ and $h_2$ with $\|h_1-h_{i}\|_{L^\infty}, \|h_2-h_{i}\|_{L^\infty}<\nu$. For the same reasons we have  
$$\big{|}\partial_h\mathcal{Q}_{i}(h_1, \omega)-\partial_h\mathcal{Q}_{i}(h_2, \omega)\big{|}\le C\big{|}h_1-h_2\big{|}$$
and the existence, for each $l\ge 2$, of constants $C$ such that 
$\big{|}\partial_h^l\mathcal{Q}_{i}(h,\omega)\big{|}\le C$, $l\ge 2$.
The same types of estimates apply to derivatives of $\mathcal{Q}_{i}(h,\omega)$ with respect to spherical, $\omega$-variables.
In combination with the Chain Rule all these estimates combined show that for each $l\gg 1$ one can find constants $C$ of class $\C(i)$ for which  
$$\|\mathcal{Q}_{i}(h_1)-\mathcal{Q}_{i}(h_2)\|_{H^l(S^2)}\le C\nu\|h_1-h_2\|_{H^{l}(S^2)}$$
so long $\|h_1-h_{i}\|_{H^{l}(S^2)}, \|h_2-h_{i}\|_{H^{l}(S^2)}<\nu$.
\end{proof}

\subsection{Solving the minimal surface equation}\label{solvingthepde}
We first re-write the minimal surface equation \eqref{minsurfeqn-taos} as 
$$\mathcal{R}_{i}h=\mathcal{E}_{i}h,$$
where the operator $\mathcal{E}_{i}$ is addressed in Lemma \ref{approxlemma}. Using linearization at $h=h_{i}$ the latter can be alternatively expressed as a fixed point problem $h = F_{i}h$ where 
$$F_{i}h:= h_{i}-\mathcal{L}_{i}^{-1}\left(\mathcal{E}_{i}h +\mathcal{Q}_{i}(h)\right).$$
Our strategy now is to apply the Banach Fixed Point Theorem to the mapping(s) $F_{i}$. 

We start by showing that for a fixed $l$, sufficiently small $\tfrac{\alpha_i+\beta_i}{\sigma_i}$ and sufficiently small $\nu$ the mapping $F_{i}$ takes a ball $B_\nu(h_{i})\subseteq H^{l+2}(S^2)$ of small radius $\nu<1$ into itself. To this end let $h\in B_\nu(h_{i})$. Lemmas \ref{approxlemma}, \ref{linftyinvertible} and \ref{Qn:eqn} guarantee that $\|\mathcal{L}_{i}^{-1}\left(\mathcal{E}_{i}h +\mathcal{Q}_{i}(h)\right)\|_{H^{l+2}(S^2)}$ can be bounded by a $\C(i)$-multiple of 
$$\frac{\alpha_i+\beta_i}{\tau_i}\left(\left(\frac{\tau_i}{\sigma_i}+\nu\right)^3+\frac{\tau_i}{\sigma_i}\left(\frac{\tau_i}{\sigma_i}+\nu\right)+\nu^2\right)$$
or, more simply, a $\C(i)$-multiple of 
$$\frac{\tau_i(\alpha_i+\beta_i)}{\sigma_i^2}+\frac{\alpha_i+\beta_i}{\sigma_i}\nu+\frac{\alpha_i+\beta_i}{\tau_i}\nu^2.$$
Assuming that $\frac{\alpha_i+\beta_i}{\sigma_i}\le \frac{1}{3C}$ is sufficiently small, we can find values of $\nu\ll1$ such that  
\begin{equation}\label{n-vs-nu}
3C\left(\frac{\alpha_i+\beta_i}{\sigma_i}\right)^2\le\frac{\alpha_i+\beta_i}{\tau_i}\nu\le \frac{1}{3C};
\end{equation} 
it is for such values of $\nu\ll 1$ that we now have 
$$\|F_{i}h-h_{i}\|_{H^{l+2}(S^2)}\le \nu \text{\ \ i.e\ \ }F_{i}:B_\nu(h_{i})\to B_\nu(h_{i}).$$
Likewise, Lemmas \ref{approxlemma}, \ref{linftyinvertible} and \ref{Qn:eqn} imply   
$$\begin{aligned}
\|F_{i}h_1 - F_{i}h_2\|_{H^{l+2}(S^2)}\\
\le & C\frac{\alpha_i+\beta_i}{\tau_i}\left(\left(\frac{\tau_i}{\sigma_i}+\nu\right)^2+\frac{\tau_i}{\sigma_i}+\nu\right) \, \|h_1-h_2\|_{H^{l+2}(S^2)}\\
\le &C\left(\frac{\alpha_i+\beta_i}{\sigma_i}+\frac{\alpha_i+\beta_i}{\tau_i}\nu\right) \, \|h_1-h_2\|_{H^{l+2}(S^2)}
\end{aligned}$$
for all $h_1, h_2 \in B_\nu(h_{i})$.
Bounds on $\frac{\alpha_i+\beta_i}{\sigma_i}$ and $\nu$ of type \eqref{n-vs-nu} with a larger value of $C$ ensure that the multiplicative factor above is no more than $1/2$,
meaning that $F_{i}:B_\nu(h_{i})\to B_\nu(h_{i})$ is a contraction. We now see from the Banach Fixed Point Theorem that $F_{i}$ has a unique fixed point in $B_\nu(h_{i})$. In particular, by choosing the smallest possible $\nu$ in \eqref{n-vs-nu}, with $n$ sufficiently large relative to control variables of class $\C(i)$,
we obtain a solution $h$ of $\mathcal{T}_{i}h=0$ with 
$$\|h-h_{i}\|_{H^l(S^2)}=O\left(\frac{\tau_i(\alpha_i+\beta_i)}{\sigma_i^2}\right).$$
The reader should note that we do have the freedom to use larger values of $\nu$ such as $\nu=\e\tau_i/(\alpha_i+\beta_i)$ where $\e$ is small enough - as determined by control variables of class $\C(i)$. Such a choice then proves the uniqueness of solutions $h$ within a larger class: 
\begin{equation}\label{hunique}
\|h-h_{i}\|_{H^l(S^2)}\le \e\tau_i/(\alpha_i+\beta_i).
\end{equation}
Also worthy of notice is the fact that Lemma \ref{semilinearlemma} implies that $\|h\|_{L^\infty}=O\left(\frac{\alpha_i+\beta_i}{\sigma_i}\right)$ and $\|dh\|_{H^l(S^2)}=O\left(\frac{\tau_i}{\sigma_i}\right)$. 
Overall, we have proven the following theorem.

\begin{theorem}\label{actualminsurfthm}
If $\frac{\alpha_i+\beta_i}{\sigma_i}$ is sufficiently small relative to control variables of class $\C(i)$, then the minimal surface equation \eqref{MinSurf:eqn} has a smooth solution $f$ with 
$$\|f-\hat{f}_{i}\|_{L^\infty(S^2)}=O\left(\frac{\alpha_i+\beta_i}{\sigma_i}\right).$$
In addition, for each $l$ the function $f$ satisfies an estimate of the form 
$$\|df\|_{H^l(S^2)}=O\left(\frac{\tau_i}{\sigma_i}\right).$$
The implied proportionality constants are of class $\C(i)$.
\end{theorem}

The existence portion of Theorem \ref{minsurf:thm1-taos} is simply a streamlined (and weaker) version of our Theorem \ref{actualminsurfthm}.

\section{Minimal Surfaces Associated to Individual Points Sources -- Uniqueness}\label{minsurf-sec}

The strategy in this section is to use judiciously chosen foliations to gradually narrow down the locations of any and all minimal surfaces within $B_{\geucl}(p_i,\sigma_i/2)$. The foliations are chosen so we can have a very well controlled sign of the mean curvature along the leaves. Throughout the section we use the formula which relates the mean curvature $H_{g_{\BL}}$ of a surface computed with respect to the ambient metric $g_{\BL}$ and the mean curvature $H_{\geucl}$ of the same surface computed with respect to the ambient Euclidean metric: 
\begin{equation}\label{genH}
H_{g_{\BL}}=(\chi_{\BL}\psi_{\BL})^{-1}H_{\geucl}+2(\chi_{\BL}\psi_{\BL})^{-2}\grad\left(\chi_{\BL}\psi_{\BL}\right)\cdot \vec{N}_{\geucl}.
\end{equation}

\subsection{Foliation by spheres} 
The following formula is immediate from \eqref{genH} and the fact that $|x-p_i|=\varrho$ has the mean curvature of $H_{\geucl}=\tfrac{2}{\varrho}$. 

\begin{lemma}\label{meancurvformula} The mean curvature $H_{g_{\BL}}$ of the coordinate sphere $|x-p_i|=\varrho$ computed with respect to the metric $g_{\BL}$ and the outward pointing normal satisfies 
$$\begin{aligned}
\frac{(\chi_{\BL}\psi_{\BL})^2}{2} H_{g_{\BL}}=&\grad(\chi^{(i)}\psi^{(i)})\cdot \vec{N}_{\geucl}\\
+&\frac{1}{\varrho}\left(\chi^{(i)}\psi^{(i)}+\alpha_{i}\,\grad\,\psi^{(i)}\cdot \vec{N}_{\geucl}+\beta_{i}\,\grad\,\chi^{(i)}\cdot  \vec{N}_{\geucl}\right)-\frac{\alpha_{i}\beta_{i}}{\varrho^3},
\end{aligned}$$
with all the gradients and all the unit normals on the right hand sides are computed with respect to the Euclidean metric. 
\end{lemma}

Our next goal is to establish that far enough out the coordinate spheres have positive mean curvature while close in their mean curvature is negative. 

\begin{lemma}\label{oldChrLemma2} 
There exist a constant $C$ of class $\C(i)$ and a value of $\e>0$ which is small relative to control variables of class $\C(i)$ such that the following hold whenever $\frac{\alpha_i+\beta_i}{\sigma_i}<\e$. 
\medbreak
\begin{enumerate}
\item\label{part3/2minsurfprop-ver2} If $p_i\in \P_{**}$ the annular region 
$$0<|x-p_i|<\left((\hat{\chi}^{(i)}\hat{\psi}^{(i)})^{-1/2}-C\tfrac{\alpha_i+\beta_i}{\sigma_i}\right)\sqrt{\alpha_{i}\beta_{i}}$$
is foliated by coordinate spheres of negative outward / positive inward mean curvature. 
\medbreak
\item\label{part2minsurfprop-ver2} The annular region
$$\left((\hat{\chi}^{(i)}\hat{\psi}^{(i)})^{-1/2}+C\tfrac{\alpha_i+\beta_i}{\sigma_i}\right)\sqrt{\alpha_{i}\beta_{i}}< |x-p_i|<\sigma_i/C$$
are foliated by coordinate spheres of positive (outward) mean curvature. 
\end{enumerate}
\end{lemma}

\begin{remark}\label{taos-afterthefact-remark1}
The result \eqref{part2minsurfprop-ver2} may be improved in circumstances when we have further information about $|\grad \chi^{(i)}|+|\grad \psi^{(i)}|$ near $p_i$. As is, we only know that over the ball $|x-p_i|\le \sigma_i/2$ we have 
$$|\grad\chi^{(i)}|+|\grad\psi^{(i)}|\le C/\sigma_i$$
for some constant $C$ of class $\C(i)$. Replacing $\tfrac{\sqrt{\alpha_i\beta_i}}{\sigma_i}$ in the proof below by 
$$\sqrt{\alpha_i\beta_i} \left\||\grad\chi^{(i)}|+|\grad \psi^{(i)}|\right\|_{L^\infty(B_{\geucl}(p_i,\sigma_i/2))}$$ proves that annular regions
$$\left((\hat{\chi}^{(i)}\hat{\psi}^{(i)})^{-1/2}+C\tfrac{\alpha_i+\beta_i}{\sigma_i}\right)\sqrt{\alpha_{i}\beta_{i}}< |x-p_i|<R_i,$$
where $R_i\le \sigma_i/2$ is any value for which
$$C\,\left\||\grad\chi^{(i)}|+|\grad \psi^{(i)}|\right\|_{L^\infty(B_{\geucl}(p_i,\sigma_i/2))} \le 1/R_i,$$
are foliated by coordinate spheres of positive (outward) mean curvature.
\end{remark}

\begin{proof} Our proof splits into two cases depending of whether $p_i\in \P_{**}$ or $p_i\not\in \P_{**}$. In both cases the strategy is to control the sign of the expressions $\mathscr{H}_{\pm}$ defined as follows:
\begin{equation}\label{aux-ineq1-oct2020}
\mathscr{H}_\pm:=\pm |\grad(\chi^{(i)}\psi^{(i)})|_{\geucl}+\tfrac{1}{\varrho}\left(\chi^{(i)}\psi^{(i)}\pm \beta_{i}|\grad(\chi^{(i)})|_{\geucl}\pm\alpha_{i}|\grad(\psi^{(i)})|_{\geucl}\right)-\tfrac{\alpha_{i}\beta_{i}}{\varrho^3}.
\end{equation}
By Lemma \ref{meancurvformula} we have $\mathscr{H}_-\le \frac{(\chi_{\BL}\psi_{\BL})^2}{2} H_{g_{\BL}}\le \mathscr{H}_+$; thus positivity of $\mathscr{H}_-$ implies positivity of $H$ and negativity of $\mathscr{H}_+$ implies negativity of $H_{g_{\BL}}$. 
\medbreak
{\sc The case of $p_i\in \P_{**}$ i.e $\alpha_{i}\beta_{i}\neq 0$.} The idea is to set 
\begin{equation}\label{defnK}
\varrho=|x-p_i|=\tau_iK\le \sigma_i/2
\end{equation} 
for $\tau_i=\sqrt{\alpha_i\beta_i}$ and a judicious choice of $K$. Note that, by the  Mean Value Theorem, we have:
$$|\chi^{(i)}(x)\psi^{(i)}(x)-\hat{\chi}^{(i)}\hat{\psi}^{(i)}|\le \frac{C}{\sigma_i}|x-p_i|$$
for all $x$ with $|x-p_i|\le \sigma_i/2$.
We get 
$$\tau_i\,\mathscr{H}_\pm=\pm\tau_i\left|\grad(\chi^{(i)}\psi^{(i)})\right|+\frac{1}{K}\left(\chi^{(i)}\psi^{(i)}\pm \alpha_{i}|\grad\,\psi^{(i)}|\pm \beta_{i}\,|\grad\,\chi^{(i)}|\right)-\frac{1}{K^3},$$
from which we obtain 
$$\tau_i\,\mathscr{H}_\pm- \left(\frac{\hat{\chi}^{(i)}\hat{\psi}^{(i)}}{K}-\frac{1}{K^3}\right)\in \frac{\tau_i}{\sigma_i}\C(i)+\frac{1}{K}\frac{\alpha_{i}\,+\beta_{i}}{\sigma_i}\C(i).$$
It follows there is some constant $C\in \C(i)$ for which 
\begin{equation}\label{meancurv-mainestimate}
\begin{aligned}
&\tau_i\,\mathscr{H}_{+}<\left(\frac{\hat{\chi}^{(i)}\hat{\psi}^{(i)}+C(\alpha_i+\beta_i)/\sigma_i}{K}-\frac{1}{K^3}\right)+C\frac{\tau_i}{\sigma_i},\\
&\tau_i\,\mathscr{H}_{-}>\left(\frac{\hat{\chi}^{(i)}\hat{\psi}^{(i)}-C(\alpha_i+\beta_i)/\sigma_i}{K}-\frac{1}{K^3}\right)- C\frac{\tau_i}{\sigma_i}.
\end{aligned}
\end{equation}

The function 
$$f_+(K)=\frac{\hat{\chi}^{(i)}\hat{\psi}^{(i)}+C(\alpha_i+\beta_i)/\sigma_i}{K}-\frac{1}{K^3}$$ 
has a root at 
$$K_0=\left(\hat{\chi}^{(i)}\hat{\psi}^{(i)}+C(\alpha_i+\beta_i)/\sigma_i\right)^{-1/2},$$
and is increasing and concave down on $(0,\sqrt{3}K_0)$. Thus there is an interval of the form $(0,K_0-C'\tfrac{\tau_i}{\sigma_i})$ on which $f_+(K)<-C\tfrac{\tau_i}{\sigma_i}$. In particular, if $(\alpha_i+\beta_i)/\sigma_i$ is sufficiently small (relative to control variables of class $\C(i)$) then for some constant $C'$ of class $\C(i)$ and all 
$$K<\left(\hat{\chi}^{(i)}\hat{\psi}^{(i)}\right)^{-1/2}\!\!\!-C'\frac{\alpha_i+\beta_i}{\sigma_i}$$
the quantities $\mathscr{H}_+$ and $H_{g_{\BL}}$ are both negative. 

Similar analysis applies to the function 
$$f_-(K)=\frac{\hat{\chi}^{(i)}\hat{\psi}^{(i)}-C(\alpha_i+\beta_i)/\sigma_i}{K}-\frac{1}{K^3},$$
although admittedly one also has to pay a little extra attention to the fact that $f_-(K)\to 0$ as $K\to \infty$. For example, we may want to impose a restriction that $(\alpha_i+\beta_i)/\sigma_i$ be small enough so that 
\begin{equation}\label{aux-ineq2-oct2020}
\hat{\chi}^{(i)}\hat{\psi}^{(i)}-C(\alpha_i+\beta_i)/\sigma_i>1/2,
\end{equation}
which in turn ensures that 
$$f_-(K)>1/(4K) \text{\ \ when\ \ } K>2.$$ 
Thus if $(\alpha_i+\beta_i)/\sigma_i$ is sufficiently small we have $f_-(K)>C\tfrac{\tau_i}{\sigma_i}$ whenever $K$ belongs to an interval of the form 
$$\left(\hat{\chi}^{(i)}\hat{\psi}^{(i)}\right)^{-1/2}+C'\frac{\alpha_i+\beta_i}{\sigma_i}<K<\frac{\sigma_i}{4C\tau_i}.$$
In particular, for values of $K$ in this range the quantities $\mathscr{H}_-$ and $H_{g_{\BL}}$ are both positive. 

{\sc The case when $p_i\not\in \P_{**}$ i.e $\alpha_{i}\beta_{i}=0$.} Going back to \eqref{aux-ineq1-oct2020} and arguing as above leads to  
$$\mathscr{H}_{-}>\frac{\hat{\chi}^{(i)}\hat{\psi}^{(i)}-C(\alpha_i+\beta_i)/\sigma_i}{\varrho}- \frac{C}{\sigma_i}.$$
Under the assumption of type \eqref{aux-ineq2-oct2020} the latter becomes 
\begin{equation}\label{meancurv-mainestimate-ver2}
\mathscr{H}_{-}>\frac{1/2}{\varrho}- \frac{C}{\sigma_i}.
\end{equation}
The right hand side of this inequality is clearly positive for all $0<\varrho<\frac{\sigma_i}{2C}$. In other words, the signs of $\mathscr{H}_-$ and $H_{g_{\BL}}$ are both positive. This completes our proof.
\end{proof}

Equipped with the knowledge of the mean curvature along the leaves of the spherical foliation allows us to narrow down possible locations of the minimal surfaces. 

\begin{lemma}\label{BLR-minsurf} There exists a constant $C$ of class $\C(i)$ and a value of $\e>0$ small relative to control variables of class $\C(i)$ such that the following hold whenever $\frac{\alpha_i+\beta_i}{\sigma_i}$. 
\begin{enumerate}
\item\label{BLR-part3/2minsurfprop} Let $p_i\in \P_{**}$. Any (smooth, immersed) minimal surface which is contained in $B_{\geucl}(p_i, \sigma_i/C)$ is necessarily contained in  
$$B_{\geucl}\left(p_i, \left((\hat{\chi}^{(i)}\hat{\psi}^{(i)})^{-1/2}+C\tfrac{\alpha_i+\beta_i}{\sigma_i}\right)\sqrt{\alpha_{i}\beta_{i}}\right).$$
\medbreak
\item\label{BLR-part2minsurfprop} Let $p_i\in \P_{*}\smallsetminus\P_{**}$. There are no minimal surfaces which are completely contained within $B_{\geucl}(p_i, \sigma_i/C)$.
\end{enumerate}
\end{lemma}

\begin{remark}\label{taos-afterthefact-remark3}
The value of $\sigma_i/C$ could in certain circumstances be increased -- please refer to Remark \ref{taos-afterthefact-remark1} for further details.
\end{remark}

\begin{proof}
Suppose $\Sigma$ is a minimal surface contained entirely in $B_{\geucl}(p_i, \sigma_i/C)$. Furthermore, suppose $\max_{x\in \Sigma}|x-p_i|=\varrho_{\mathrm{out}}$ is reached at a point $x_0$. Then at $x_0$ we must have 
$$H_{\geucl}(\Sigma)\ge H_{\geucl}(\{|x-p_i|=\varrho_{\mathrm{out}}\})=\tfrac{2}{\varrho_{\mathrm{out}}}.$$
Since $\vec{N}_{\geucl}$ for $\Sigma$ at $x_0$ is the same as it would be for the sphere $|x-p_i|=\varrho_{\mathrm{out}}$ it follows that 
$$0=H_{g_{\BL}}(\Sigma)\ge H_{g_{\BL}}(\{|x-p_i|=\varrho_{\mathrm{out}}\})$$
i.e that the coordinate sphere $|x-p_i|=\varrho_{\mathrm{out}}$ has nonpositive outward mean curvature $H_{g_{\BL}}$ at $x_0$. Thus, the coordinate sphere $|x-p_i|=\varrho_{\mathrm{out}}$ cannot have everywhere positive outward mean curvature $H_{g_{\BL}}$. In the case when $p_i\in\P_{**}$ Lemma \ref{oldChrLemma2} implies that 
$$\varrho_{\mathrm{out}}< \left((\hat{\chi}^{(i)}\hat{\psi}^{(i)})^{-1/2}+C\tfrac{\alpha_i+\beta_i}{\sigma_i}\right)\sqrt{\alpha_{i}\beta_{i}},$$
while in the case of $p_i\not\in\P_{**}$ we have a contradiction. 
\end{proof}

\subsection{An alternative foliation}
The idea now is to consider foliations of close neighborhoods of $p_i$ by surfaces which are dilated versions of the minimal surface $\Sigma_{i}$. We shall abuse the notation and let $k\Sigma_{i}$ denote the surface determined by the function $k\cdot S_{i}$ where $S_{i}$ is the function of Theorem \ref{minsurf:thm1-taos}. In view of Lemma \ref{BLR-minsurf} we are mainly interested in values of $k$ near $k\le C$. The intuition behind the next several steps is as follows. Concentric spheres in Reissner-Nordstr\"om geometry form a foliation in which the outward mean curvature transitions from being negative to positive as we go from the inside towards the outside of the minimal surface. The geometry near $\Sigma_{i}$ is approximately that of Reissner-Nordstr\"om body with $\Sigma_{i}$ corresponding to the minimal sphere in the middle of the Reissner-Nordstr\"om ``neck". Thus for $k<1$ we expect $k\Sigma_{i}$ to have negative outward mean curvature and for $k>1$ we expect it to have positive outward mean curvature. The proof of this fact is based on a following computation. 

\begin{lemma}\label{lastminutecomputation}
Consider a vector field $\vec{N}$ which is parallel, in the Euclidean sense of the word, in the $\partial_r$-direction stemming from $p_i\in \P_{**}$. If $\partial_r\cdot \vec{N}>0$ then the functions
$$r\,\frac{\grad \chi_{\BL}\cdot \vec{N}}{\chi_{\BL}} \text{\ \ and\ \ } r\,\frac{\grad \psi_{\BL}\cdot \vec{N}}{\psi_{\BL}}$$
are increasing in $r=|x-p_i|$, all assuming $\tfrac{\alpha_i+\beta_i}{\sigma_i}$ is sufficiently small and 
$r\le C\sqrt{\alpha_{i}\beta_{i}}$
with $C\in \C(i)$. 
\end{lemma}

\begin{proof}
Decompose $\chi_{\BL}=\frac{\alpha_{i}}{r}+\chi^{(i)}$. We then have:
$$r\,\frac{\grad \chi_{\BL}\cdot \vec{N}}{\chi_{\BL}}=\frac{-\alpha_{i}(\partial_r\cdot \vec{N})+r^2\,\grad \chi^{(i)}\cdot \vec{N}}{\alpha_{i}+r\chi^{(i)}}.$$
To proceed we need estimates on $r\chi^{(i)}$, $\partial_r(r\chi^{(i)})$, $r^2\,\grad \chi^{(i)}\cdot \vec{N}$ and 
$\partial_r(r^2\,\grad \chi^{(i)}\cdot \vec{N})$ over the region where $r=O(\sqrt{\alpha_{i}\beta_{i}})$. Such estimates can be obtained from the observation that $\chi^{i}$ is bounded on the ball $B_{\geucl}(p_i, \sigma_i/2)$, which in turn means that
$$\|\grad \chi^{(i)}\|_{L^\infty(B_{\geucl}(p_i, \sigma_i/2))}\le \tfrac{C}{\sigma_i} \text{\ \ and\ \ } \|\mathrm{Hess} \chi^{(i)}\|_{L^\infty(B_{\geucl}(p_i, \sigma_i/2))}\le \tfrac{C}{\sigma_i^2}$$
for some constant $C$ of class $\C(i)$. Specifically, we have:
\begin{itemize}
\item $r\chi^{(i)}=O(\sqrt{\alpha_{i}\beta_{i}})$;
\medbreak
\item $\partial_r(r\chi^{(i)})=\chi^{(i)}+O(\tfrac{\sqrt{\alpha_{i}\beta_{i}}}{\sigma_i})$;
\medbreak
\item $r^2\,\grad \chi^{(i)}\cdot \vec{N}=O(\tfrac{\alpha_{i}\beta_{i}}{\sigma_i})$;
\medbreak
\item $\partial_r(r^2\,\grad \chi^{(i)}\cdot \vec{N})=O(\tfrac{\sqrt{\alpha_{i}\beta_{i}}}{\sigma_i}) + O(\tfrac{\alpha_{i}\beta_{i}}{\sigma_i^2})=O(\tfrac{\sqrt{\alpha_{i}\beta_{i}}}{\sigma_i})$.
\end{itemize}
The sign of $\partial_r\left(r\,\frac{\grad \chi_{\BL}\cdot \vec{N}}{\chi_{\BL}}\right)$ is determined by 
$$\begin{aligned}
&\partial_r(r^2\,\grad \chi^{(i)}\cdot \vec{N})(\alpha_{i}+r\chi^{(i)}) - (-\alpha_{i}(\partial_r\cdot \vec{N})+r^2\,\grad \chi^{(i)}\cdot \vec{N})\partial_r(r\chi^{(i)})\\
=&\alpha_{i}\left(\chi^{(i)}(\partial_r\cdot \vec{N})+O(\tfrac{\alpha_i+\beta_i}{\sigma_i})\right)
\end{aligned}$$
The leading term $\chi^{(i)}(\partial_r\cdot \vec{N})$ is positive and bounded away from zero, while $O(\tfrac{\alpha_i+\beta_i}{\sigma_i})$ can be made as small as needed. Thus  
$$\partial_r\left(r\,\frac{\grad \chi_{\BL}\cdot \vec{N}}{\chi_{\BL}}\right)>0$$
assuming $\tfrac{\alpha_i+\beta_i}{\sigma_i}$ is small enough relative to $\C(i)$.
\end{proof}

\begin{lemma}\label{alternativefoliation}
Let $k\Sigma_{i}$ denote the surface determined by the function $k\cdot S_{i}$ as in the statement of Theorem \ref{minsurf:thm1-taos}. There exists a constant $C$ of class $\C(i)$ and a value of $\e>0$ small relative to $\C(i)$ such that the following hold whenever $\frac{\alpha_i+\beta_i}{\sigma_i}<\e$: 
\begin{itemize}
\item $k\Sigma_{i}$ has positive mean curvature for all $1<k\le C$ and
\medbreak
\item $k\Sigma_{i}$ has negative mean curvature for all $0<k<1$.
\end{itemize}
\end{lemma}

\begin{proof}
In our application of Lemma \ref{lastminutecomputation} the vector field $\vec{N}=\vec{N}_{\geucl}$ is the Euclidean unit normal vector field to the surfaces $k\Sigma_{i}$. It can be computed that 
$$\partial_r\cdot \vec{N}_{\geucl}=\tfrac{f}{\sqrt{f^2+|df|^2}}>0$$
for $f$ from Theorem \ref{actualminsurfthm}. We see from \eqref{genH} that 
\begin{equation}
0=\frac{\chi_{\BL}\psi_{\BL}}{2}H_{g_{\BL}}(\Sigma_{i})=\frac{1}{2}H_{\geucl}(\Sigma_{i})+\frac{\grad \chi_{\BL}\cdot \vec{N}_{\geucl}}{\chi_{\BL}}\big{|}_{\Sigma_{i}}+ \frac{\grad \psi_{\BL}\cdot \vec{N}_{\geucl}}{\psi_{\BL}}\big{|}_{\Sigma_{i}}.
\end{equation}
Since $H_{\geucl}(k\Sigma_{i})=\tfrac{1}{k}H_{\geucl}(\Sigma_{i})$, we further have:
\begin{equation}
\begin{aligned}
\frac{\chi_{\BL}\psi_{\BL}}{2}H_{g_{\BL}}(k\Sigma_{i})=&\frac{1}{2k}H_{\geucl}(\Sigma_{i})+\frac{\grad \chi_{\BL}\cdot \vec{N}_{\geucl}}{\chi_{\BL}}\Big{|}_{k\Sigma_{i}}+ \frac{\grad \psi_n\cdot \vec{N}_{\geucl}}{\psi_{\BL}}\Big{|}_{k\Sigma_{i}}\\
=&\frac{1}{kS_{i}}\left(kS_{i}\cdot \frac{\grad \chi_{\BL}\cdot \vec{N}_{\geucl}}{\chi_{\BL}}\Big{|}_{k\Sigma_{i}}-S_{i}\frac{\grad \chi_{\BL}\cdot \vec{N}_{\geucl}}{\chi_{\BL}}\Big{|}_{\Sigma_{i}}\right)\\
&+\frac{1}{kS_{i}}\left(kS_{i}\cdot \frac{\grad \psi_{\BL}\cdot \vec{N}_{\geucl}}{\psi_{\BL}}\Big{|}_{k\Sigma_{i}}-S_{i}\frac{\grad \psi_{\BL}\cdot \vec{N}_{\geucl}}{\psi_{\BL}}\Big{|}_{\Sigma_{i}}\right)
\end{aligned}
\end{equation}
The $\pm$ signs of the differences in the last two lines are addressed in Lemma \ref{lastminutecomputation}: they are positive when $k>1$ and negative when $k<1$.
\end{proof}

\subsection{The proof of uniqueness of $\Sigma_{i}$}
The following is the proof of the uniqueness portion of Theorem \ref{minsurf:thm1-taos}. The reader should note that the uniqueness statement can be improved in certain circumstances; Remark \ref{taos-afterthefact-remark1} has all the relevant details. 

\begin{proof}
Suppose that $\Sigma$ is any (other) minimal surface contained in $B_{\geucl}(p_i, \sigma_i/C)$. By Lemma \ref{BLR-minsurf} the surface is actually contained within a ball $B_{\geucl}(p_i, C\sqrt{\alpha_{i}\beta_{i}})$ with $C\in\C(i)$. 

Let 
$$k_+=\inf\{k\,\big{|}\, \Sigma\subseteq \mathrm{Int}(k\Sigma_{i})\},$$
where $\mathrm{Int}$ denotes the connected component of the complement of $k\Sigma_{i}$ containing $p_i$. The surface $k_+\Sigma_{i}$ is tangential to the surface $\Sigma$ with 
$$\Sigma\subseteq \overline{\mathrm{Int}}(k_+\Sigma_{i}).$$
At the point of tangency we thus have 
$$H_{\geucl}(k_+\Sigma_{i})\le H_{\geucl}(\Sigma) \text{\ \ and thus\ \ } 
H_{g_{\BL}}(k_+\Sigma_{i})\le H_{g_{\BL}}(\Sigma)=0.$$
It follows from Lemma \ref{alternativefoliation} that $k_+\le 1$ i.e that  
$$\Sigma\subseteq \overline{\mathrm{Int}}(\Sigma_{i}).$$

Now let
$$k_-=\sup\{k\,\big{|}\, \Sigma\subseteq \mathrm{Out}(k\Sigma_{i})\},$$
where $\mathrm{Out}$ denotes the connected component of the complement of $k\Sigma_{i}$ not containing $p_i$. The surface $k_-\Sigma_{i}$ is tangential to the surface $\Sigma$ with 
$$\Sigma\subseteq \overline{\mathrm{Out}}(k_-\Sigma_{i}).$$
In what follows let the $\pm$-sign be in correspondence to whether $\Sigma$ does or does not contain $p_i$, respectively.
At the point of tangency of $\Sigma$ and $k_-\Sigma_{i}$ we have 
$$\pm H_{\geucl}(\Sigma)\le  H_{\geucl}(k_-\Sigma_{i})\text{\ \ and thus\ \ } 
\pm H_{g_{\BL}}(\Sigma)=0\le H_{g_{\BL}}(k_-\Sigma_{i}),$$
It follows from Lemma \ref{alternativefoliation} that $k_-\ge 1$ i.e that  
$$\Sigma\subseteq \overline{\mathrm{Out}}(\Sigma_{i}).$$
Since we already established $\Sigma\subseteq \overline{\mathrm{Int}}(\Sigma_{i})$, our proof is now complete.
\end{proof}

\section{Brill-Lindquist-Riemann sums: point-wise convergence}\label{ConvSec}

The conformal factors $\chi_n$, $\psi_n$ for Brill-Lindquist-Riemann sums were introduced in Definition \ref{BLR}. Recall that the broad goal of this paper is to examine geometric consequences of convergences 
$$\chi_n\to \chi,\ \ \psi_n\to \psi,$$
where functions $\chi$ and $\psi$ are defined as in \eqref{potentials1} -- \eqref{potentials2}. We begin this process by investigating point-wise behavior of the sequence of functions $\chi_n$ and $\psi_n$ and their derivatives. 

\subsection{Behavior at the $C^0$-level}\label{C0Sec}

This portion of Section \ref{ConvSec} is dedicated to the analysis of the $C^0$-behavior of $\chi_n$, $\psi_n$ and $g_n$. We investigate two different regimes. One regime addresses locations which are ``far enough" from individual point-objects even though they are (potentially) within the dust cloud itself; the other regime is about locations which are ``pretty close" to one individual point-object. The reason for this  has to do with the relative sizes of terms such as 
$$\frac{\alpha_{i,n}}{|x-p_{i,n}|} \text{\ \ and\ \ } \chi^{(i)}_n(x):=\chi_n(x)-\frac{\alpha_{i,n}}{|x-p_{i,n}|}$$
within a small neighborhood of $p_{i,n}$, as well as the relative sizes of their derivatives. 

\begin{proposition}\label{BLRcontrollemma}
\ 
\begin{enumerate}
\item\label{boundednessoffactors} 
There exists a constant $C$ of class $\C$ such that 
$$\begin{cases}
\chi_n(x) +\psi_n(x)\le C & \text{if}\ \ x\not \in \bigcup_i B(p_{i,n}, \frac{D}{n^2}),\\
\chi_n^{(i)}(x)+ \psi_n^{(i)}(x)\le C & \text{if}\ \ x \in B(p_{i,n}, \frac{D}{2n}).
\end{cases}$$
for all $n$ and all $i$.
\medbreak
\item\label{qualityofapproximation}
There exists a constant $C$ of class $\C^+$ for which 
$$\begin{cases}
\left|\chi_n(x) -\chi(x) \right|+ \left|\psi_n(x) -\psi(x) \right|\le \frac{C}{n} & \text{if}\ \ x\not \in \bigcup_i B(p_{i,n}, \frac{D}{n^2}),\\
\left|\chi_n^{(i)}(x) -\chi(x) \right|+\left|\psi_n^{(i)}(x) -\psi(x) \right|\le\frac{C}{n} & \text{if}\ \ x \in B(p_{i,n}, \frac{D}{2n}).
\end{cases}$$
for all $n$ and all $i$.
\end{enumerate}
\end{proposition}

One significance of Proposition \ref{BLRcontrollemma} is that it establishes the existence of a constant $C$ of class $\C$ such that 
$$\hat{\chi}_n^{(i)}:=\chi_n^{(i)}(p_{i,n})\le C \ \ \text{and}\ \ \hat{\psi}_n^{(i)}:=\psi_n^{(i)}(p_{i,n})\le C,$$
as was promised in part \ref{controlclasses} of the Introduction. In addition, we may assume that 
$$\left|\hat{\chi}_n^{(i)} -\chi(p_{i,n}) \right|+\left|\hat{\psi}_n^{(i)} -\psi(p_{i,n}) \right|\le\tfrac{C}{n}.$$

While part \eqref{qualityofapproximation} of Proposition \ref{BLRcontrollemma} indicates that $g_n\approx g$ far away from point-objects, one has to be much more careful with the statement that $g_n$ is well approximated by Reissner-Norstr\"om metrics near point-objects. As tempting as it may be to claim that near $p_{i,n}$ we have $g_n\approx g_{\RN,i}$ for
$$g_{\RN,i}:=\left(\chi(p_{i,n}) +\frac{\alpha_{i,n}}{|x-p_{i,n}|}\right)^2\left(\psi(p_{i,n}) +\frac{\beta_{i,n}}{|x-p_{i,n}|}\right)^2\geucl,$$
this may not actually be true! The source of difficulty here lies in potential unboundedness of terms $\frac{\alpha_{i,n}}{|x-p_{i,n}|}$ and $\frac{\beta_{i,n}}{|x-p_{i,n}|}$.
(What is true, however, is that 
$$(1-\tfrac{C}{n})^4 g_{\RN,i}\le g_n\le (1+\tfrac{C}{n})^4 g_{\RN,i}$$
for some constant $C$ of class $\C^+$.) 

\bigbreak

We begin our proof of Proposition \ref{BLRcontrollemma} by recording universal bounds on integrals which frequently appear in our arguments. The proof of Lemma \ref{universalboundslemma} is a simple integration exercise. 

\begin{lemma}\label{universalboundslemma}
The following hold for all $x\in \mathbb{R}^3$:
\begin{enumerate}
\item $\int_{y\in [-D,D]^3}\frac{1}{|x-y|}\,\dvol \le 100D^2$;
\medbreak
\item $\int_{y\in [-D,D]^3}\frac{1}{|x-y|^2}\,\dvol \le 70D$;
\end{enumerate}
\end{lemma}

One consequence of Lemma \ref{universalboundslemma}, for example, is that 
$$\chi\le 1+100\|A\|D^2,\ \ \psi\le 1+100\|B\|D^2,\ \ \text{etc.}$$
Thus, there exists a constant $C$ of class $\C$ such that 
\begin{equation}\label{usedtobe39}
\geucl\le g\le C^2\geucl.
\end{equation}

\begin{lemma}\label{firstRSlemma}
\  
\begin{enumerate}
\item\label{RSlemmaPT1}  Assume that $x\in\mathbb{R}^3\smallsetminus \left(\bigcup_i B_{\geucl}(p_{i,n}, \tfrac{D}{n^2})\right)$. We then have 
$$\left|\sum_i \frac{(D/n)^3}{|x-p_{i,n}|} - \int_{y\in[-D,D]^3} \frac{1}{|x-y|}\,\dvol\right|\le 700\frac{D^2}{n}.$$
\medbreak
\item\label{RSlemmaPT2} Assume that $x\in B_{\geucl}(p_{i,n}, \tfrac{D}{2n})$ for some $p_{i,n}$. We then have
$$\left|\sum_{j\neq i} \frac{(D/n)^3}{|x-p_{j,n}|} - \int_{y\in[-D,D]^3} \frac{1}{|x-y|}\,\dvol\right|\le 700\frac{D^2}{n}.$$
\end{enumerate}
\end{lemma}

\begin{proof} 
We start the proof of part \eqref{RSlemmaPT1} by observing that  
$$\left|\sum_i \frac{(D/n)^3}{|x-p_{i,n}|} - \int_{y\in[-D,D]^3} \frac{1}{|x-y|}\,\dvol\right|\le \sum_i \int_{y\in V_{i,n}}\left|\frac{1}{|x-p_{i,n}|}-\frac{1}{|x-y|}\right|\,\dvol.$$
We break the sum into two constituents: one being the summation over $i$'s for which $x$ and $p_{i,n}$ are ``close" and the other being the summation over $i$'s for which $x$ and $p_{i,n}$ are ``not close". The latter case could arise from situations when $x$ is still within the vicinity of $[-D,D]^3$, and it could arise from situations when $x$ is far from $[-D,D]^3$ altogether. 

{\sc The case of $p_{i,n}$'s which are close to $x$.}
Specifically, there are at most $27$ boxes $V_{i,n}$ where $|x-y|< \tfrac{D}{n}$ for some $y\in V_{i,n}$. Summation over such $V_{i,n}$ yields 
$$\sum \int_{y\in V_{i,n}}\frac{1}{|x-y|}\,\dvol\le 4\pi \left(\int_0^{2\sqrt{3}D/n}\!\!r\,dr\right)\le 24\pi\frac{D^2}{n^2}.$$
With the exception of at most one box we may estimate $|x-p_{i,n}|\ge \tfrac{D}{2n}$ and the summation over such boxes yields
$$\sum \int_{y\in V_{i,n}}\frac{1}{|x-p_{i,n}|}\,\dvol\le 54\frac{D^2}{n^2}.$$
However, in the exceptional case when $x\in V_{i,n}$ for some $i$ and yet $\tfrac{D}{n^2}\le |x-p_{i,n}|$ the most dominant term is 
$$\int_{y\in V_{i,n}}\frac{1}{|x-p_{i,n}|}\,\dvol\le\frac{D^2}{n}.$$
Overall, in this case we have $\sum_i \int_{y\in V_{i,n}}\left|\frac{1}{|x-p_{i,n}|}-\frac{1}{|x-y|}\right|\,\dvol\le  100\frac{D^2}{n}$.

\medbreak
{\sc The case of $p_{i,n}$'s which are not close to $x$.} On the remaining boxes $V_{i,n}$ we employ the Mean Value Theorem:
$$\left|\frac{1}{|x-p_{i,n}|}-\frac{1}{|x-y|}\right|\le \frac{1}{|x-z|^2}\cdot\frac{D\sqrt{3}}{2n}$$
where $z$ is some point on the line segment joining $p_{i,n}$ and $y$. We now have that
$$y,z\in V_{i,n} \text{\ \ with\ \ } |x-z|,|x-y|\ge \tfrac{D}{n}.$$
It then follows that 
$$|x-z|\le |x-y|+\sqrt{3}\tfrac{D}{n}\le 3|x-y| \text{\ \ and\ \ } \tfrac{1}{3}|x-y|\le |x-z|\le 3|x-y|,$$
which further results in 
$$\sum \int_{y\in V_{i,n}}\left(\frac{1}{|x-p_{i,n}|}-\frac{1}{|x-y|}\right)\,\dvol\le \frac{D\sqrt{3}}{2n} \int_{y\in [-D,D]^3}\frac{9}{|x-y|^2}\,\dvol\le 630\frac{D^2}{n}.$$
Our proof of part \eqref{RSlemmaPT1} is now complete. 

The proof of part \eqref{RSlemmaPT2} relies on the same idea. By setting aside the integration over the very $V_{i,n}$ we obtain 
$$\left|\sum_{j\neq i} \frac{(D/n)^3}{|x-p_{j,n}|} - \int_y \frac{1}{|x-y|}\,\dvol\right|\le 630\frac{D^2}{n}+\int_{y\in V_{i,n}} \frac{1}{|x-y|}\,\dvol.$$
Using spherical coordinates the value of the latter integral is found to satisfy 
$$\int_{y\in V_{i,n}} \frac{1}{|x-y|}\,\dvol\le 4\pi \left(\int_0^{\sqrt{3}D/n}\!\!r\,dr\right)\le 20 \frac{D^2}{n^2}.$$
Overall, we obtain $\left|\sum_{j\neq i} \frac{(D/n)^3}{|x-p_{j,n}|} - \int_y \frac{1}{|x-y|}\,\dvol\right|\le 650\frac{D^2}{n}$.
\end{proof}

When combined Lemma \ref{universalboundslemma} and Lemma \ref{firstRSlemma} provide us with  bounds on the Riemann sums themselves. Specifically, we can now prove part \eqref{boundednessoffactors} of Proposition \ref{BLRcontrollemma}.

\subsubsection*{Proof of part \eqref{boundednessoffactors} of Proposition \ref{BLRcontrollemma}}
Lemmas \ref{universalboundslemma} and \ref{firstRSlemma} provide the following bounds on Riemann sums:
\begin{equation}\label{RSrndestimate}
\begin{cases}
\sum_i \frac{(D/n)^3}{|x-p_{i,n}|} \le 800D^2 & \text{if}\ \ x\not \in \bigcup_i B(p_{i,n}, \frac{D}{n^2}),\\
\sum_{j\neq i} \frac{(D/n)^3}{|x-p_{j,n}|} \le 800D^2 & \text{if}\ \ x \in B(p_{i,n}, \frac{D}{2n}).
\end{cases}
\end{equation}
On the other hand, it follows from Definition \ref{BLR} that 
$$\alpha_{i,n}\le \left(\|A\|+\frac{\C(\alpha, A)}{nD^3}\right)\left(\frac{D}{n}\right)^3.$$
Therefore, we have 
$$\chi_{n}(x)\le 1+\left(\|A\|+\frac{\C(\alpha, A)}{nD^3}\right) \sum_i \frac{(D/n)^3}{|x-p_{i,n}|}\le 1+800\left(\|A\|D^2+\frac{\C(\alpha, A)}{D}\right)$$
for all $n$, so long as $x\not \in \bigcup_iB(p_{i,n}, \frac{D}{n^2})$. The estimate on $\chi_n^{(i)}(x)$ when $x \in B(p_{i,n}, \frac{D}{2n})$ can be proven in the exactly same way. \qed

\bigbreak

Next, we have a Product Rule-inspired application of Lemmas \ref{universalboundslemma} and \ref{firstRSlemma}. 

\begin{lemma}\label{secondRSlemma}
Let $A$ be a smooth function supported on $[-D,D]^3$. 
\begin{enumerate}
\item If $x\in\mathbb{R}^3\smallsetminus \left(\bigcup_i \bar{B}_{\geucl}(p_{i,n}, \tfrac{D}{n^2})\right)$ then  
$$\left|\sum_i \frac{A(p_{i,n})}{|x-p_{i,n}|} (D/n)^3- \int_{y\in[-D,D]^3} \frac{A(y)}{|x-y|}\,\dvol\right|\le \frac{700}{n}\left(\|A\|D^2+\|dA\|D^3\right).$$
\medbreak
\item If $x\in B_{\geucl}(p_{i,n}, \tfrac{D}{2n})$ for some $p_{i,n}$ then $$\left|\sum_{j\neq i} \frac{A(p_{j,n})}{|x-p_{j,n}|} (D/n)^3 - \int_{y\in[-D,D]^3} \frac{A(y)}{|x-y|}\,\dvol\right|\le \frac{700}{n}\left(\|A\|D^2+\|dA\|D^3\right).$$
\end{enumerate}
All the norms involved are $L^\infty(\mathbb{R}^3)$-norms. 
\end{lemma}

\begin{proof} 
The proof here relies on the observation that 
$$\begin{aligned}
\left|\frac{A(p_{i,n})}{|x-p_{i,n}|} - \frac{A(y)}{|x-y|}\right|
\le &A(p_{i,n})\left|\frac{1}{|x-p_{i,n}|} - \frac{1}{|x-y|}\right| + \frac{1}{|x-y|}\left|A(p_{i,n})-A(y)\right|\\
\le &\|A\|\left|\frac{1}{|x-p_{i,n}|} - \frac{1}{|x-y|}\right| + \|dA\|\frac{|p_{i,n}-y|}{|x-y|}.
\end{aligned}$$
For the purposes of integration with respect to $y\in V_{i,n}$ the last expression can be replaced by 
$$\|dA\|\,\frac{\sqrt{3}D/n}{|x-y|}.$$
Upon integration over $V_{i,n}$ and summing over $i$ we obtain 
$$\begin{aligned}
&\left|\sum_i \frac{A(p_{i,n})}{|x-p_{i,n}|} (D/n)^3- \int_y \frac{A(y)}{|x-y|}\,\dvol\right|\\
\le &\|A\| \left|\sum_i \frac{(D/n)^3}{|x-p_{i,n}|} - \int_{y\in[-D,D]^3} \frac{1}{|x-y|}\,\dvol\right| + \|dA\| \frac{\sqrt{3}D}{n}  \int_{y\in[-D,D]^3} \frac{1}{|x-y|}\,\dvol.
\end{aligned}$$
Our result is now a consequence of Lemmas \ref{universalboundslemma} and \ref{firstRSlemma}.
\end{proof}

\subsubsection*{Proof of part \eqref{qualityofapproximation} of Proposition \ref{BLRcontrollemma}}
Using Definition \ref{BLR} we have 
$$\left|\sum \frac{\alpha_{i,n}}{|x-p_{i,n}|} - \sum \frac{A(p_{i,n})}{|x-p_{i,n}|}\right|\le \frac{\C(\alpha, A)}{nD^3}\sum_i \frac{(D/n)^3}{|x-p_{i,n}|}.$$
The summation appearing at the end is the Riemann sum which can, by Lemmas \ref{universalboundslemma} and \ref{firstRSlemma}, be bounded by $800D^2$; compare with step \eqref{RSrndestimate} in the proof of part \eqref{boundednessoffactors} of this very proposition. In particular, we have: 
$$\begin{cases}
\left|\sum_i \frac{\alpha_{i,n}}{|x-p_{i,n}|} - \sum_i \frac{A(p_{i,n})(D/n)^3}{|x-p_{i,n}|} \right|\le \frac{800}{nD}\C(\alpha, A) & \text{if}\ \ x\not \in \bigcup_i B(p_{i,n}, \frac{D}{n^2}),\\
\left|\sum_{j\neq i} \frac{\alpha_{j,n}}{|x-p_{j,n}|} - \sum_{j\neq i} \frac{A(p_{j,n})(D/n)^3}{|x-p_{j,n}|} \right|\le \frac{800}{nD}\C(\alpha, A) & \text{if}\ \ x\in B(p_{i,n}, \frac{D}{2n}).
\end{cases}$$
Statements made in part \eqref{qualityofapproximation} of Proposition \ref{BLRcontrollemma} are now immediate from Lemma \ref{secondRSlemma}. \qed

\subsection{$C^1$-behavior}\label{C1Sec}

Lemma \ref{firstRSlemma} is largely about convergence of Riemann sums towards to the integral $\int_{y\in[-D,D]^3} \tfrac{1}{|x-y|}\,\dvol$. The integral $\int_{y\in[-D,D]^3} \tfrac{1}{|x-y|^2}\,\dvol$ is also convergent and one may wonder about the corresponding statement regarding the latter integral. Modifying the strategy from the proof of Lemma \ref{firstRSlemma} to accommodate for the larger exponents we can prove the following:
\begin{enumerate}
\item There is some universal constant $C$ such that  
$$\left|\sum_i \frac{(D/n)^3}{|x-p_{i,n}|^2} - \int_{y\in[-D,D]^3} \frac{1}{|x-y|^2}\,\dvol\right|\le C\frac{D}{n^{3-2\nu}}$$
so long as
$$x\not\in \bigcup_i B_{\geucl}\left(p_{i,n}, Dn^{-\nu}\right)$$
for some $1<\nu<3/2$;
\medbreak
\item If $x\in B_{\geucl}(p_{i,n}, D/(2n))$ then  
$$\left|\sum_{j\neq i} \frac{(D/n)^3}{|x-p_{j,n}|^2} - \int_{y\in[-D,D]^3} \frac{1}{|x-y|^2}\,\dvol\right|\le CD\frac{\ln(n)}{n}$$
for some universal constant $C$.
\end{enumerate}
Note that, relative to Lemma \ref{firstRSlemma}, the $C^1$-context requires us to cut out somewhat larger neighborhoods of locations $p_{i,n}$. The only ``non-obvious" modification in the proof of Lemma \ref{firstRSlemma} happens towards the very end when estimating $\sum \int_{y\in V_{i,n}}\left(\frac{1}{|x-p_{i,n}|^2}-\frac{1}{|x-y|^2}\right)$. At that stage we employ the fact that said sum is bounded by \begin{equation}\label{lnestimate}
\frac{D\sqrt{3}}{2n} \int_{D/n\le |x-y|\le 2D}\frac{27}{|x-y|^3}\,\dvol=O(D\ln(n)/n).
\end{equation}
Owing to the fact that one can take the derivative of $\chi$ and $\psi$ under the integral sign once, the remaining arguments of Section \ref{C0Sec} follow a predictable course. They lead to:

\begin{proposition}\label{BLRcontrollemmaC1}
Fix a parameter $\nu$ with $1<\nu<3/2$.
\begin{enumerate}
\item
There exists a constant $C$ of class $\C$ such that 
$$\begin{cases}
\left|d\chi_n(x)\right| +\left|d\psi_n(x)\right|\le C/D & \text{if}\ \ x\not \in \bigcup_i B(p_{i,n}, Dn^{-\nu}),\\
\left|d\chi_n^{(i)}(x)\right|+\left|d\psi_n^{(i)}(x)\right|\le C/D & \text{if}\ \ x \in B(p_{i,n}, D/(2n)).
\end{cases}$$
for all $n$ and all $i$.
\medbreak
\item
There exists a constant $C$ of class $\C^+$ for which 
$$\begin{cases}
\left|d\chi_n(x) -d\chi(x) \right|+ \left|d\psi_n(x) -d\psi(x) \right|\le \frac{C/D}{n^{3-2\nu}} & \text{if}\ \ x\not \in \bigcup_i B(p_{i,n}, Dn^{-\nu}),\\
\left|d\chi_n^{(i)}(x) - d\chi(x) \right|+\left|d\psi_n^{(i)}(x) - d\psi(x) \right|\le\frac{C/D}{n} & \text{if}\ \ x \in B(p_{i,n}, D/(2n)).
\end{cases}$$
for all $n$ and all $i$.
\end{enumerate}
\end{proposition}

Once again, we can see that there is a very good approximation of Euclidean derivatives $\partial g_n\approx \partial g$ so long we are far enough away from point-objects. 

\begin{remark}\label{remark3.2forBLR}
Remarks \ref{taos-afterthefact-remark1} and \ref{taos-afterthefact-remark3} make reference to situations where an improvement can be made to the uniqueness portion of Lemma \ref{BLR-minsurf} / Theorem \ref{minsurf:thm1-taos}. In the context of Brill-Lindquist-Riemann sums the value $\sigma_{i,n}/C$ of Lemma \ref{BLR-minsurf} / Theorem \ref{minsurf:thm1-taos} could, for example, be replaced by $\sigma_{i,n}/2=D/(2n)$ provided we can arrange 
$$C\||d\chi_n^{(i)}|+|d\psi_n^{(i)}|\| \le 2n/D$$
for $n$ which are large relative to $\C$ and for all $i$. Proposition \ref{BLRcontrollemmaC1} above assures us that indeed is the case. 
\end{remark}

\subsection{$C^2$-behavior}
The second and the higher order derivatives of $\chi_n$ and $\psi_n$ are not as well behaved. The strongest statement we can make (and prove) here is that over sets of the form $\mathbb{R}^3\smallsetminus\left(\cup_i B_{\geucl}(p_{i,n}, cD/n)\right)$ with $c$ (small and) fixed we have uniform boundedness of the second derivatives of $\chi_n$ and $\psi_n$. No ``convergence" statement towards second derivatives of $\chi$ or $\psi$ is expected or even possible: the functions $\chi_n$ and $\psi_n$ are harmonic far away from sources $p_{i,n}$ while $\chi$ and $\psi$ satisfy 
$$\Delta_{\geucl} \chi = -4\pi A \text{\ \ and\ \ }  \Delta_{\geucl} \psi = -4\pi B.$$

Consider $x\not\in \cup_i B_{\geucl}(p_{i,n}, cD/n)$, with $c\ll 1$ fixed. Since there is nothing to be concerned about if $x\not\in[-D,D]^3$, assume $x\in V_{i,n}$ for some $i$. Note that under such a premise the term $\mathrm{Hess}\left(\frac{\alpha_{i,n}}{|x-p_{i,n}|}\right)=O\left(\frac{\alpha_{i,n}}{|x-p_{i,n}|^3}\right)$ and the corresponding $\beta$-term are both bounded. In fact, the same applies for sources $p_{j,n}$ which are in immediate vicinity of $V_{i,n}$. Thus in the remaining discussion we may assume that $j\neq i$ is such that $|x-y|\ge D/n$ for all $y\in V_{j,n}$. Also note that 
$$\begin{aligned}
\left|\mathrm{Hess}_{\geucl}\Big{|}^{\xi=x}_{\xi=p_{i,n}}\left(\sum_{j\neq i}\frac{\alpha_{j,n}}{|\xi-p_{j,n}|}\right)\right|\le &C\sum_{j\neq i} \frac{\alpha_{j,n}}{|p_{i,n}-p_{j,n}|^4}\cdot \frac{D\sqrt{3}}{2n}\\
\le &\frac{CD\sqrt{3}}{2n} \int_{D/n\le |x-y|\le 2D}\frac{1}{|x-y|^4}\,\dvol=O(1).
\end{aligned}$$
This means that the second derivatives of $\chi_n$ at $x$ permit a uniform bound if and only if the second derivatives of $\chi_n^{(i)}$ at $p_{i,n}$ do. As the unboundedness concern arises from the sources $p_{j,n}$ which are close to $p_{i,n}$, our situation here can be additionally simplified by assuming that the parameters $\alpha_{j,n}$ are non-zero only when $p_{j,n}$ are in a fixed small ball around $p_{i,n}$.  An even further simplification consists of the replacement of said parameters $\alpha_{j,n}$ with a constant, namely $A(p_{i,n})(D/n)^3$. To see that such a simplification is appropriate consider the fact that 
$$\left|\alpha_{j,n}-A(p_{i,n})(D/n)^3\right|\le \|dA\| (D/n)^3 |p_{i,n}-p_{j,n}| +\C(\alpha, A)/n^4$$
and that, by Proposition \ref{BLRcontrollemmaC1} and estimation as in \eqref{lnestimate}, the Hessian  
$$\begin{aligned}
&\left|\mathrm{Hess}_{\geucl}\left(\sum_{j\neq i}\frac{\|dA\| (D/n)^3 |p_{i,n}-p_{j,n}| +\C(\alpha, A)/n^4}{|p_{i,n}-p_{j,n}|}\right)\right|\\
\le&\|dA\|\sum_{j\neq i} \frac{(D/n)^3}{|p_{i,n}-p_{j,n}|^2} +C\frac{\C(\alpha, A)}{D^3} \frac{\ln(n)}{n}
\end{aligned}$$
is bounded. Overall, it remains to investigate boundedness of 
\begin{equation}\label{iwantthistobeover}
\mathrm{Hess}_{\geucl}\Big{|}_{\xi=p_{i,n}}\left(\sum_{j\neq i} \frac{1}{|\xi-p_{j,n}|}\right),
\end{equation}
with $p_{j,n}$ only ranging in a fixed small ball around $p_{i,n}$.

Let $\vec{\nu}$ denote any unit vector based at $p_{i,n}$. Direct computation shows that 
$$\mathrm{Hess}_{\geucl}\Big{|}_{\xi=p_{i,n}}\left(\frac{1}{|\xi-p_{j,n}|}\right)(\vec{\nu}, \vec{\nu})=3\frac{((p_{i,n}-p_{j,n})\cdot \vec{\nu})^2}{|p_{i,n}-p_{j,n}|^5}-\frac{1}{|p_{i,n}-p_{j,n}|^3}.$$
In principle, we are to insert this expression into \eqref{iwantthistobeover}. Before doing so note that, 
for symmetry reasons, the Hessian in \eqref{iwantthistobeover} is invariant under replacement of $p_{j,n}$ with $p_{j,n}'$ or $p_{j,n}''$ where the latter two are such that 
$$\{p_{i,n}-p_{j,n},\ p_{i,n}-p_{j,n}',\ p_{i,n}-p_{j,n}''\}$$
form an orthogonal basis of vectors of equal length. Since 
$$\begin{aligned}
((p_{i,n}-p_{j,n})\cdot \vec{\nu})^2+&((p_{i,n}-p_{j,n}')\cdot \vec{\nu})^2+((p_{i,n}-p_{j,n}'')\cdot \vec{\nu})^2=|p_{i,n}-p_{j,n}|^2.
\end{aligned}$$
we see that the Hessian in \eqref{iwantthistobeover} vanishes! 

In conclusion, there exists a uniform bound on the second derivatives of $\chi_n$ and $\psi_n$ over sets of the form $\mathbb{R}^3\smallsetminus\left(\cup_i B_{\geucl}(p_{i,n}, cD/n)\right)$ with $c\ll 1$ fixed\footnote{The bound depends on $c$.}. Such uniform bounds ensure uniform bounds on the curvature of $g_n$, as noted in the Introduction.

\section{Brill-Lindquist-Riemann sums: The spaces $(\V_{n,R},g_n)$ and $(\V_{n,R,R'},g_n)$}\label{deepwells:section}

Recall the definition of the set $\V_{n,R}$ from Definition \ref{defn-mn}.
It follows from Theorem \ref{minsurf:thm1} that 
\begin{equation}\label{inclusionsMn-ver2}
B_{\geucl}(0,R)\!\smallsetminus\!\left(\cup \bar{B}_{\geucl}(p_{i,n},C\sqrt{\alpha_{i,n}\beta_{i,n}})\right) \!\subseteq\! \V_{n,R}\! \subseteq\! B_{\geucl}(0,R)\!\smallsetminus\! \left(\cup \bar{B}_{\geucl}(p_{i,n},\sqrt{\alpha_{i,n}\beta_{i,n}}/C)\right).
\end{equation}
for some constant $C$ of class $\C$; note that the stated inclusions apply even if $p_{i,n}\not\in\P_{n,**}$.
Throughout the remainder of this article we assume that $n$ is suitably large not only so that Theorem \ref{minsurf:thm1} applies but also so that 
$$B_{\geucl}(0,R)\smallsetminus \left(\bigcup B_{\geucl}(p_{i,n},\tfrac{D}{n^2})\right) \subseteq \V_{n,R}.$$
This is possible by virtue of the fact that $\sqrt{\alpha_{i,n}\beta_{i,n}}=O(\tfrac{D}{n^3})$ with the implied proportionality constant of class $\C$.
In fact, what is true in this situation is that 
\begin{equation}\label{towardsIFL}
B_{\geucl}(0,R)\smallsetminus \left(\bigcup B_{\geucl}(p_{i,n},\tfrac{D}{n^2})\right) = \V_{n,R} \smallsetminus \left(\bigcup B_{\geucl}(p_{i,n},\tfrac{D}{n^2})\right) 
\end{equation}
This is a good moment to point out to the reader that $0\in \V_{n,R}$ for all $n$ due to the fact that $0$ is simply a corner of some of the subdivision boxes.

\subsection{Length and diameter estimates}

The front row of point sources in the Figures \ref{fig2} and \ref{fig3} illustrates the fact that the individual ``necks" could be quite long - even in situations where we have $\mathcal{P}_{n,**}=\P_{n,*}$ for all $n$. The following lemma provides explicit estimates on this length.

\begin{lemma}\label{necklengthlemma}
Let $p_{i,n}\in \P_{n,**}$ and let\footnote{The factor of $\tfrac{1}{D}$ is included so to non-dimensionalize as many terms as possible in the long run.} 
$$\ell_{i,n}:= \tfrac{1}{D}(\alpha_{i,n}+\beta_{i,n})|\ln(\alpha_{i,n}\beta_{i,n}/D^2)|.$$
Furthermore, assume the points $q,q'\in \V_{n,R}$ are collinear with $p_{i,n}$ and that $q$ is between $p_{i,n}$ and $q'$.
\begin{enumerate}
\item If $|q-p_{i,n}|\le 2 \sqrt{\alpha_{i,n}\beta_{i,n}}$ and $|q'-p_{i,n}|\ge \tfrac{D}{2n^2}$, and if $n$ is large relative to $\C$ then 
$$d_{g_n}(q,q')\ge \tfrac{D}{12}(\tfrac{1}{n^2}+\ell_{i,n});$$
\medbreak
\item There is a constant $C$ of class $\C$ such that if 
$$|q-p_{i,n}|\ge \tfrac{1}{2}\sqrt{\alpha_{i,n}\beta_{i,n}} \text{\ \ and\ \ } |q'-p_{i,n}|\le 2\tfrac{D}{n^2},$$ 
and if $n$ is large relative to $\C$ then 
$$d_{g_n}(q,q')\le CD(\tfrac{1}{n^2}+\ell_{i,n}).$$
\end{enumerate}
\end{lemma}

\begin{proof}
We begin by proving the lower bound on $d_{g_n}(q,q')$. Since 
$$g_n\ge (1+\tfrac{\alpha_{i,n}}{|x-p_{i,n}|})^2(1+\tfrac{\beta_{i,n}}{|x-p_{i,n}|})^2\geucl,$$
the $g_n$-distance between $q$ and $q'$ is not less than the distance between $q$ and $q'$ with respect to said Reissner-Nordstr\"om metric. The latter is spherically symmetric about $p_{i,n}$ and thus the minimizing geodesic connecting $q$ and $q'$ follows the Euclidean (radial) line segment with endpoints at $q$ and $q'$. It follows that 
$$\begin{aligned}
d_{g_n}(q,q')\ge&\int_{r=2\sqrt{\alpha_{i,n}\beta_{i,n}}}^{r=D/(2n^2)} (1+\tfrac{\alpha_{i,n}}{r})(1+\tfrac{\beta_{i,n}}{r})\,dr\\
=&\tfrac{D}{2n^2}-(\alpha_{i,n}+\beta_{i,n})\ln\left(\tfrac{4n^2}{D}\sqrt{\alpha_{i,n}\beta_{i,n}}\right)-\tfrac{3}{2}\sqrt{\alpha_{i,n}\beta_{i,n}}-\tfrac{2n^2}{D}\alpha_{i,n}\beta_{i,n}.
\end{aligned}$$
To proceed observe that the inequalities
$$\begin{aligned}
&\tfrac{D}{2n^2}-\tfrac{3}{2}\sqrt{\alpha_{i,n}\beta_{i,n}}-\tfrac{2n^2}{D}\alpha_{i,n}\beta_{i,n}\ge \tfrac{D}{12n^2}\\
&4n^2\le (\alpha_{i,n}\beta_{i,n})^{-5/12}D^{5/6}=\left(A(p_{i,n})B(p_{i,n})D^4\right)^{-5/12} \cdot n^{5/2}
\end{aligned}$$ 
apply when $n$ is sufficiently large relative to $\C$. Overall, we obtain 
$$d_{g_n}(q,q')\ge \tfrac{D}{12n^2} -\tfrac{1}{12}(\alpha_{i,n}+\beta_{i,n})\ln(\alpha_{i,n}\beta_{i,n}/D^2)= \tfrac{D}{12}(\tfrac{1}{n^2}+\ell_{i,n}).$$

We continue by proving the upper bound on $d_{g_n}(q,q')$. Here is an absolutely crucial (and yet very delicate) observation: Theorem \ref{minsurf:thm1} implies\footnote{Had the minimal surface $\Sigma_{i,n}$ not been as controllable as Theorem \ref{minsurf:thm1} makes it, the shortest path from $q$ to $q'$ within $\V_{n,R}$ could be quite roundabout. This would make it hard to place an upper bound on its length, and it is for this reason that we decided to solve the minimal surface equation in the way we did.} that the Euclidean line segment $\gamma$ which joins $q$ and $q'$ is contained in $\V_{n,R}$! It is for this reason that we have 
$$d_{g_n}(q,q')\le L(\gamma).$$
To estimate $L(\gamma)$ we make use of Proposition \ref{BLRcontrollemma}, according to which there exists a constant $C$ of class $\C$ with 
$$g_n\le C^2(1+\tfrac{\alpha_{i,n}}{|x-p_{i,n}|})^2(1+\tfrac{\beta_{i,n}}{|x-p_{i,n}|})^2\geucl.$$
along $\gamma$. It follows that 
$$L(\gamma)\le C\int_{r=\frac{1}{2}\sqrt{\alpha_{i,n}\beta_{i,n}}}^{r=2D/n^2} (1+\tfrac{\alpha_{i,n}}{r})(1+\tfrac{\beta_{i,n}}{r})\,dr.$$
Evaluation of the integral leads to 
$$\begin{aligned}
&\tfrac{2D}{n^2}-(\alpha_{i,n}+\beta_{i,n})\ln\left(\tfrac{n^2}{4D}\sqrt{\alpha_{i,n}\beta_{i,n}}\right)+\tfrac{3}{2}\sqrt{\alpha_{i,n}\beta_{i,n}}-\tfrac{n^2}{2D}\alpha_{i,n}\beta_{i,n}\\
\le&\tfrac{3D}{n^2}-\tfrac{1}{2}(\alpha_{i,n}+\beta_{i,n})\ln(\alpha_{i,n}\beta_{i,n}/D^2),
\end{aligned}$$
provided $n$ is large relative to $\C$. Overall, we have
$$d_{g_n}(q,q')\le L(\gamma)\le C\left(\tfrac{D}{n^2}-(\alpha_{i,n}+\beta_{i,n})\ln(\alpha_{i,n}\beta_{i,n}/D^2)\right)\le CD\left(\tfrac{1}{n^2}+\ell_{i,n}\right)$$
for some (larger) constant $C$ of class $\C$. 
\end{proof}

\begin{remark}\label{goingaround}
This is a great moment to point to a related but complementary (pun!) computation: Suppose $q$ and $q'$ are points in $\V_{n,R}\cap  B_{\geucl}(p_{i,n}, \tfrac{2D}{n^2})$ with $|q-p_{i,n}|=|q'-p_{i,n}|$ and suppose that the circular arc joining $q$ and $q'$ is contained in $\V_{n,R}$ we have that 
$$d_{g_n}(q,q') \le C\pi (1+\tfrac{\alpha_{i,n}}{\varrho})(1+\tfrac{\beta_{i,n}}{\varrho})\varrho,$$
where $\varrho$ is the shared value of $|q-p_{i,n}|=|q'-p_{i,n}|$. Given that $\tfrac{1}{2}\sqrt{\alpha_{i,n}\beta_{i,n}}\le \varrho \le \tfrac{2D}{n^2}$, the maximum of the stated expression in $\varrho$ is achieved when $\varrho=\tfrac{2D}{n^2}$ and is on the order of $O(\tfrac{D}{n^2})$. If one or both of the points $q$ and / or $q'$ are in $\V_{n,R}\cap B_{\geucl}(p_{i,n}, 2\sqrt{\alpha_{i,n}\beta_{i,n}})$ we may need to append one or two radial connectors to reach the circular arc $|x-p_{i,n}|=2\sqrt{\alpha_{i,n}\beta_{i,n}}$. The length of these connectors is no more than 
$$\left(1+\frac{2\alpha_{i,n}}{\sqrt{\alpha_{i,n}\beta_{i,n}}}\right)\left(1+\frac{2\beta_{i,n}}{\sqrt{\alpha_{i,n}\beta_{i,n}}}\right)\cdot \frac{3}{2}\sqrt{\alpha_{i,n}\beta_{i,n}}\le 12(\alpha_{i,n}+\beta_{i,n})$$
and therefore the estimate 
$$d_{g_n}(q,q')=O(D/n^2)$$
still applies. Overall, the point is that even though the ``necks" associated with individual point sources could be quite long at least they are very thin. 
\end{remark}

Lemma \ref{necklengthlemma} suggests that the diameter of the set obtained by truncating the asymptotically Euclidean end $|x|\to \infty$ (compare with Figure \ref{fig3}) is governed by the quantity
\begin{equation}\label{dn:defn}
\ell_n:=\begin{cases}
\tfrac{1}{D}\max_{i}(\alpha_{i,n}+\beta_{i,n})|\ln(\alpha_{i,n}\beta_{i,n}/D^2)| & \text{if\ \ } \P_{n,**}=\P_{n,*};\\
\infty &\text{if\ \ } \P_{n,**}\neq \P_{n,*},
\end{cases}
\end{equation} 
at least for very large $n$. Investigating this point further is crucial to get any study of convergence of Brill-Lindquist-Riemann sums off the ground. 

\begin{lemma}\label{diameterlemma}
Suppose a Brill-Lindquist-Riemann sum with $\mathcal{P}_{n,**}=\P_{n,*}$ for all $n$. There exist a constant $C$ of class $\C$ for which  
$$\tfrac{1}{C}(R+\ell_nD)\le \mathrm{diam}_{g_n}(\V_{n,R})\le C(R+\ell_nD)$$
for all $n$ which are large relative to $\C$.
\end{lemma}

\begin{proof}
To prove the lower bound we exhibit a pair of points whose $g_n$-distance is at least (a $\C$-multiple) of $R$, and a pair of points whose $g_n$-distance is at least $\tfrac{1}{12}\ell_n$. For the first pair consider the points where the line containing an edge of the box $[-D,D]^3$ pierces the Euclidean sphere whose radius is (just short of) $R$. Since $g_n\ge \geucl$ the $g_n$-length of any path in $\V_{n,R}$ joining these two points is not less than its Euclidean length, which in turn is not less than the Euclidean length of the Euclidean line segment joining the two points. In other words, the $g_n$ distance between these two points is at least $2R\cos(\arctan(\sqrt{2}))=O(R)$. For the second pair of points consider $p_{i,n}$ for which the maximum in the definition of $\ell_n$ is reached. Let $q$ be a point on the Euclidean sphere of radius $2\sqrt{\alpha_{i,n}\beta_{i,n}}$ centered at $p_{i,n}$ (cf. inclusions of \eqref{inclusionsMn-ver2}) and let $q'$ be the point where the ray from $p_{i,n}$ towards $q$ pierces the Euclidean sphere of radius $D/(2n^2)$ centered at $p_{i,n}$. By Lemma \ref{necklengthlemma} we have 
$$d_{g_n}(q,q')\ge \tfrac{D}{12}\ell_{i,n}=\tfrac{D}{12}\ell_n.$$

To prove the upper bound recall that by part \eqref{boundednessoffactors} of Proposition \ref{BLRcontrollemma} the metrics $g_n$ are uniformly equivalent to the Euclidean metric outside $\cup_i B_{\geucl}(p_{i,n}, \tfrac{D}{n^2})$. Since any two points in $p,q\in \V_{n,R}\smallsetminus \left(\cup_i B_{\geucl}(p_{i,n}, \tfrac{D}{n^2})\right)$ can be connected by a broken line segment contained in $\V_{n,R}\smallsetminus \left(\cup_i B_{\geucl}(p_{i,n}, \tfrac{D}{n^2})\right)$ whose sides are parallel to the coordinate axes we have
$$d_{g_n}(p,q)\le CR$$
for some $C$ of class $\C$. Thus, it suffices to prove an estimate of the form 
$$d_{g_n}(p,q)\le CD\left(1+\ell_{i,n}\right)$$
when $p\in B_{\geucl}(p_{i,n}, \tfrac{D}{n^2})$ for some $i$. In fact, it suffices to focus solely on the configuration described in Lemma \ref{necklengthlemma} in which $q$ is the point where the ray from $p_{i,n}$ towards $p$ pierces the Euclidean sphere of radius $D/n^2$ centered at $p_{i,n}$. At this stage the upper bound we need is an immediate consequence of Lemma \ref{necklengthlemma}.
\end{proof}

This is a very good moment to remind the reader of Definitions \ref{deepwellsdefinition} and \ref{shortwellsdefinition} made in the Introduction:  A sequence of Brill-Lindquist-Riemann sums is said to have \emph{deep wells} if $\P_{n,**}\neq\P_{n,*}$ for some $n$ or if the sequence of quantities $\ell_n$ defined in \eqref{dn:defn} is unbounded. Otherwise, the sequence is said to have no deep wells. Furthermore, a sequence which has no deep wells is said to have \emph{shallow wells} provided $\lim_{n\to \infty}\ell_n=0$.

\subsection{Behavior of the quantity $\ell_n$ with respect to parameters $\alpha_{i,n}$ and $\beta_{i,n}$}

Boundedness of expressions such as $x\ln(x)$ over $(0,\infty)$ proves that in situations when no charge is present (i.e when $\alpha_{i,n}=\beta_{i,n}$ for all $n$ and all $i$) we have no deep wells. In fact, since $\alpha_{i,n}=\beta_{i,n}=O(\tfrac{D}{n^3})$ while $\lim_{x\to 0^+} x\ln(x)=0$, the sequence of Brill-Lindquist-Riemann sums has shallow wells. 

It might be tempting to think that the long-term behavior of $\ell_n$ could be addressed in terms of functions $A$ and $B$ alone. In some simple situations this indeed is true. For example, for Brill-Lindquist-Riemann sums where the parameters $\alpha_{i,n}$ and $\beta_{i,n}$ are found via evaluation at some sample points with $q_{i,n}\in V_{i,n}$ (see \eqref{commonsensechoice}) the quantity $\ell_n$ can alternatively be expressed as 
$$\ell_n=\max_i \tfrac{(A(q_{i,n})D^2)+(B(q_{i,n})D^2)}{n^3}|\ln((A(q_{i,n})D^2)(B(q_{i,n})D^2)/n^6)|.$$
Due to the fact that $\lim_{n\to \infty}\frac{\ln n^6}{n^3}=0$ we have that 
$$\limsup_{n\to\infty}\ell_n=\limsup_{n\to \infty}\ell_n'$$
where the quantity $\ell_n'$ is defined by 
$$\ell_n':=\max_i \tfrac{1}{n^3}\left(A(q_{i,n})D^2|\ln(B(q_{i,n})D^2)|+B(q_{i,n})D^2|\ln(A(q_{i,n})D^2)|\right).$$
From here it is easy to see that 
\begin{equation}\label{suffcondnodeepwells}
\sup\, \{(AD^2)|\ln(BD^2)|+(BD^2)|\ln(AD^2)|\}<\infty,
\end{equation}
with the supremum taken over the region where $AB\neq 0$ is a sufficient condition for such a sequence of Brill-Lindquist-Riemann sums to have shallow wells. The condition \eqref{suffcondnodeepwells} is, for example, fulfilled whenever there exists a constant $c$ such that 
$$\tfrac{1}{c}B\le A\le c B.$$

However, the long-term behavior of $\ell_n$ is quite dependent on how we choose the parameters $\alpha_{i,n}$ and $\beta_{i,n}$. We now examine some examples. 

\subsubsection{Evaluation of $A$ and $B$ at different sample points}\label{Example1}
Consider the bump function 
$$B(p)=
\begin{cases}
\exp(\frac{1}{|p|^2-1}) & \text{if}\ |p|<1\\
0 & \text{if}\ |p|\ge 1,
\end{cases}$$
which is supported in $[-1,1]^3$, and let 
$$A=\exp(-1/B).$$
In this situation the condition \eqref{suffcondnodeepwells} does hold. Yet, if we permit the evaluation of $A$ and $B$ to be at distinct sample points then for any given $L\in (0,\infty)$ we can arrange that $\lim_{n\to \infty}\ell_n=L$. Indeed, consider the parameters $\alpha_{i,n}$ and $\beta_{i,n}$ to be as in \eqref{commonsensechoice} except at one instance where\footnote{Though the value of $D$ in our example is $D=1$ we choose to keep here for dimensional reasons.}
$$\alpha_{i,n}=\tfrac{D^3}{n^3}A(1-\tfrac{1}{2n}+\tfrac{3\ln n}{2n^2}, \tfrac{\lambda_1}{2n},0),\ \ \ \beta_{i,n}=\tfrac{D^3}{n^3}B(1-\tfrac{1}{2n}, \tfrac{\lambda_2}{2n},0)$$
for some judiciously chosen fixed real values $\lambda_1$ and $\lambda_2$. A direct computation shows that 
$$\begin{aligned}
\ell_n=&\lim_{n\to \infty} \frac{B(1-\frac{1}{2n}, \frac{\lambda_2}{2n},0)}{n^3B(1-\frac{1}{2n}+\frac{3\ln n}{2n^2}, \frac{\lambda_1}{2n},0)}\\
=&\lim_{n\to \infty}\frac{1}{n^3}\exp\left(3\ln n+\frac{\lambda_1^2-\lambda_2^2}{4}+O\left(\frac{\ln n}{n}\right)\right)=\exp\left(\frac{\lambda_1^2-\lambda_2^2}{4}\right).
\end{aligned}$$
For any given $0<L<\infty$ values of $\lambda_1$ and $\lambda_2$ can be chosen in a small neighborhood of an odd integer so that $\exp\left(\frac{\lambda_1^2-\lambda_2^2}{4}\right)=L$; this then proves our claim that $\lim_{n\to \infty}\ell_n=L$. It is now also clear that if in place of $3\ln n$ we used $2\ln n$ or $4\ln n$ we would have gotten 
$$\lim_{n\to \infty}\ell_n=0 \text{\ \ and\ \ } \lim_{n\to \infty}\ell_n=\infty,$$
respectively. 

\subsubsection{Sensitivity to sample points}\label{Example2}
Example \ref{Example1} could serve as a motivation to consider Brill-Lindquist-Riemann sums in which parameters $\alpha_{i,n}$ and $\beta_{i,n}$ are chosen according to \eqref{commonsensechoice}. However, the quantity $\ell_n$ can exhibit rich behavior even in that particular context! Consider the bump function 
$$B(p)=
\begin{cases}
\exp\left(\frac{1}{1-|p|^2}\right) & \text{if}\ |p|<1\\
0 & \text{if}\ |p|\ge 1,
\end{cases}$$
consider $A=\exp\left(\frac{\ln B}{B}\right)$. Note that under such choices the supremum in \eqref{suffcondnodeepwells} is infinite:
$$\sup\left\{(AD^2)|\ln(BD^2)|+(BD^2)|\ln(AD^2)|\right\}\ge \sup (BD^2)|\ln(AD^2)| =\sup_{|p|<1} |\ln(BD^2)|=\infty.$$
Consider Brill-Lindquist-Riemann sum where $\alpha_{i,n}$ and $\beta_{i,n}$ are chosen as in \eqref{commonsensechoice} with sample points being the midpoints $p_{i,n}$ except at one location where\footnote{This example assumes we are working with $n$ such that $n^3\ge\lambda$. This is not a substantial restriction as the interesting aspects of this example increase when $\lambda\approx 0$. } 
$$q_{i,n}=\left(1-\tfrac{\lambda}{2n^3}, 0,0\right)$$
for some positive real number $\lambda$. Note that $1-|p_{i,n}|^2=\frac{K-(3/4)}{n^2}$ for some integer $K$. This observation leads to 
$$1-|p_{i,n}|^2\ge \frac{1}{4n^2} \text{\ \ and\ \ } \frac{1}{n^3}B(p_{i,n})D^2|\ln(A(p_{i,n})D^2)|=\frac{1}{n^3(1-|p_{i,n}|^2)}\le \frac{4}{n}.$$
It follows that size of $\ell_n$ in this example is dictated by the term 
$$\frac{1}{n^3}B(q_{i,n})D^2|\ln(A(q_{i,n})D^2)|=\frac{1}{n^3(1-|q_{i,n}|^2)}=\frac{1}{\lambda-\frac{\lambda^2}{4n^3}}.$$
Overall, we see that 
$$\lim_{n\to \infty}\ell_n= \tfrac{1}{\lambda}.$$
Despite \eqref{suffcondnodeepwells} failing, our sequence of Brill-Lindquist-Riemann sums has no deep wells. Much as in Example \ref{Example1} it is easy to see how our expression for $q_{i,n}$ can be altered to produce an example with shallow wells (e.g $q_{i,n}=\left(1-\tfrac{1}{2n^2}, 0, 0\right)$) or an example with deep wells (e.g $q_{i,n}=\left(1-\tfrac{1}{2n^4}, 0, 0\right)$).    

\subsubsection{Interesting midpoint example}\label{Example3}
We now examine a related example of midpoint type. Consider the bump function 
$$B(p)=
\begin{cases}
\exp((1-|p|^2)^{-3/2}) & \text{if}\ |p|<1\\
0 & \text{if}\ |p|\ge 1,
\end{cases}$$
and once again consider $A=\exp\left(\frac{\ln B}{B}\right)$. The supremum in \eqref{suffcondnodeepwells} is infinite for the same reasons as in Example \ref{Example2}. Once again, we have $1-|p_{i,n}|^2\ge \frac{1}{4n^2}$ which then further leads to 
$$\limsup_{n\to\infty}\ell_n=\limsup_{n\to\infty} \frac{1}{n^3(1-|p_{i,n}|^2)^{3/2}}\le 8.$$

Next, we prove that the value of $\limsup \ell_n$ is indeed equal to $8$. This can be done in many different ways but we are choosing a slightly more involved way which shows that such a value of $\limsup$ is not simply due to a single odd-ball point-source as it was the case in Examples \ref{Example1} and \ref{Example2}. Our motivation for doing so is revealed in Section \ref{GH:sec} below. In particular, here we have an example of a sequence of Brill-Lindquist-Riemann sums of midpoint type with no deep wells, which is not an example where we have shallow wells. 

Consider integers $n_1, n_2, ....$ of the form 
$$n_k=\frac{1}{2}\left(1+\prod_{i=1}^k(\lambda_i^2+1)^2\right),$$
where $\lambda_i$ are even integers for which $\lambda_i^2+1$ are pairwise coprime\footnote{Worse come to worse, this property can be arranged by the ``Euclidean trick" of inductively constructing $\lambda_{i+1}:=(\lambda_1^2+1)...(\lambda_i^2+1)$.}. Such values of $n_k$ are interesting to us because the values 
$$w_{j,k}:=\frac{-1+(\lambda_j^2+2\lambda_j-1)\tilde{n}_{j,k}}{2},\ z_{j,k}:=\frac{-1+(\lambda_j^2-2\lambda_j-1)\tilde{n}_{j,k}}{2},\ 1\le j\le k$$
where $\tilde{n}_{j,k}=\frac{1}{(\lambda_j^2+1)^2}\prod_{i=1}^k(\lambda_i^2+1)^2$ are such that 
$$n_k^2=(n_k-1)^2 + (n_k-1) + w_{j,k}^2 + w_{j,k} + z_{j,k}^2 + z_{j,k}+1.$$
Stated differently, if
$$p=(\tfrac{n_k-(1/2)}{n_k}, \tfrac{w_{j,k}+(1/2)}{n_k}, \tfrac{z_{j,k}+(1/2)}{n_k})$$ 
then $1-|p|^2=\tfrac{1}{4n_k^2}$. It follows (assuming $n_k$ is large so that $\alpha_{i,n_k},\beta_{i,n_k}\le 1$) that at at least $k$ distinct locations $p_{i,n_k}$ we have
$$\begin{aligned}
\ell_{i,n_k}=&\frac{1}{D}(\alpha_{i,n_k}+\beta_{i,n_k})|\ln(\alpha_{i,n_k}\beta_{i,n_k}/D^2)|\\
\ge &\frac{1}{D}\alpha_{i,n_k}|\ln(\beta_{i,n_k}/D)|+\frac{1}{D}\beta_{i,n_k}|\ln(\alpha_{i,n_k}/D)|
\ge \frac{1}{D}\beta_{i,n_k}|\ln(\alpha_{i,n_k}/D)|\\
\ge &\frac{B(p_{i,n_k})D^2}{n_k^3}|\ln A(p_{i,n_k})D^2|=\frac{1}{n_k^3(1-|p_{i,n_k}|^2)^{3/2}}=8.
\end{aligned}$$
Consequently, we have $\ell_{n_k}\ge 8$ for $n_k$ large and $\limsup \ell_n=8$. 

It is interesting to notice that the sequence $\ell_n$ does not converge: For even values of $n$ we have $1-|p_{i,n}|^2=\tfrac{K-(3/4)}{n^2}$ with $K$ an even integer, meaning that 
$$1-|p_{i,n}|^2\ge\tfrac{5}{4n^2} \text{\ \ and\ \ } \limsup_{k\to\infty}\ell_{2k}\le \tfrac{8}{5^{3/2}}.$$
In fact, by using the values $n=2n_k$ with $n_k$ as in the previous paragraph we see that 
$$\lim_{k\to \infty}\ell_{2n_k}=\tfrac{8}{5^{3/2}}.$$

After all of these explicit computations it is probably easy to see that any usage of an exponent smaller than $-3/2$ in our expression for $B$ will lead to a sequence with deep wells while the exponent larger than $-3/2$ will lead to a sequence with shallow wells.

\subsubsection{Altering parameters to obtain shallow wells}\label{wlog?sec}
Recall that our definition of Brill-Lindquist-Riemann sums leaves some room for ``error" regarding the parameters $\alpha_{i,n}$ and $\beta_{i,n}$. The example of Section \ref{Example1} is included to show just how sensitive the quantity $\ell_n$ is to choices of $\alpha_{i,n}$ and $\beta_{i,n}$. The example was intentionally written in the style which suggests that having so much freedom in choices of parameters is not necessarily a ``good thing". In contrast, here is a result which encourages us to allow for the wiggle-room in Definition \ref{BLR}.

\begin{proposition}\label{wlog?}
For all pairs $(A,B)$ of non-negative, smooth functions supported in $[-D,D]^3$ and all $\lambda>0$ there exists a Brill-Lindquist-Riemann sum with shallow wells for which $\C(\alpha, A),  \C(\beta, B) \le \lambda D$. 
\end{proposition}

\begin{proof}
Consider the quantity $\sup\, \{(AD^2)|\ln(BD^2)|+(BD^2)|\ln(AD^2)|\}$ as in \eqref{suffcondnodeepwells}. If this supremum is finite there is nothing to show: Brill-Lindquist-Riemann sums of midpoint type have shallow wells. From now on we assume that said supremum is infinite. Without loss of generality also assume $\lambda\le 1$.

Fix a value of $n$ (which is large relative to $\C$). At midpoint locations $p_{i,n}$ where $A(p_{i,n})B(p_{i,n})\neq 0$ and 
$$A(p_{i,n})D^2|\ln(B(p_{i,n}D^2)|+B(p_{i,n})D^2|\ln(A(p_{i,n}D^2)|\le 2n^2$$
simply set $\alpha_{i,n}=A(p_{i,n})(D/n)^3$ and $\beta_{i,n}=B(p_{i,n})(D/n)^3$. 
This is sufficient due to 
$$\begin{aligned}
\ell_{i,n}=&\tfrac{1}{D}(\alpha_{i,n}+\beta_{i,n})|\ln(\alpha_{i,n}\beta_{i,n}/D^2)|\\
\le&C\tfrac{\ln(n)}{n^3}+
\tfrac{1}{n^3}\left(A(p_{i,n})D^2|\ln(B(p_{i,n}D^2)|+B(p_{i,n})D^2|\ln(A(p_{i,n}D^2)|\right)\le C\tfrac{\ln(n)}{n^3}+\tfrac{2}{n}.
\end{aligned}$$

The remaining midpoint locations $p_{i,n}$ fall under one of the following two scenarios:

{\sc The case of $A(p_{i,n})D^2, B(p_{i,n})D^2\le \tfrac{\lambda}{n}$:} In this situation set $\alpha_{i,n}=\beta_{i,n}=0$.

{\sc The case when $A(p_{i,n})D^2\ge \tfrac{\lambda}{n}$ or $B(p_{i,n})D^2\ge \tfrac{\lambda}{n}$:} Without loss of generality assume that $B(p_{i,n})D^2\ge \tfrac{\lambda}{n}$. If $A(p_{i,n})\neq 0$ then we must have at least one of 
$$A(p_{i,n})D^2|\ln(B(p_{i,n})D^2)|>n^2 \text{\ \ or\ \ } B(p_{i,n})D^2|\ln(A(p_{i,n})D^2)|>n^2.$$
The inequality $A(p_{i,n})D^2|\ln(B(p_{i,n}D^2)|>n^2$ in this situation implies 
$$A(p_{i,n})D^2>n^2/\ln (n/\lambda),$$ 
which is unsustainable for $n$ which is large relative to $\tfrac{1}{\lambda}\C$. Thus, we have 
$$B(p_{i,n})D^2|\ln(A(p_{i,n})D^2)|>n^2 \text{\ \ and thus\ \ } A(p_{i,n})D^2\le \exp(-Cn^2)$$
for a constant $C$ of class $\C$. Regardless if $A(p_{i,n})\neq 0$ or not, set 
$$\alpha_{i,n}=\tfrac{\lambda D}{n^4} \text{\ \ and\ \ } \beta_{i,n}=B(p_{i,n})(D/n)^3.$$
Note that, by virtue of $\exp(-Cn^2)\ll \lambda/n$ we have $n^4|\alpha_{i,n}-A(p_{i,n})(D/n)^3|\le \lambda D$. In addition, we have
$$\ell_{i,n}=\tfrac{1}{D}(\alpha_{i,n}+\beta_{i,n})|\ln(\alpha_{i,n}\beta_{i,n}/D^2)|\le \tfrac{C}{n^3}|\ln(B(p_{i,n})D^2/n^7)|\le 8C\tfrac{\ln(n)}{n^3}.$$
Our proof is now complete. 
\end{proof}

To summarize, the spirit behind the proof of Proposition \ref{wlog?} is the following.  If our ``measured" values $\alpha_{i,n}$ or $\beta_{i,n}$ are too close to zero then (as in Examples \ref{Example1} -- \ref{Example3}) we may simple be picking up ``noise" in the geometry of $\V_{n,R}$ in the form of very long ``necks" and an, informally speaking, chaotic behavior of $\mathrm{diam}(\V_{n,R},g_n)$. Perhaps all choices of  $\alpha_{i,n}$ or $\beta_{i,n}$ which are too close to zero (in the sense that is spelled out within the proof of Proposition \ref{wlog?}) are unreliable and are to be ``rounded off". Yet, there are very meaningful and physically relevant examples which arise when one of the function $A$ or $B$ identically vanishes and where the corresponding parameters $\alpha$ or $\beta$ are identically zero. (See the discussion at the end of the Introduction.) It is for this reason that we are not content studying shallow wells only.

\subsection{Presence of deep wells}
In situations when the sequence of Brill-Lindquist-Riemann sums has deep wells, that is, in situations when the $g_n$-diameters of sets $\V_{n,R}$ are unbounded as $n\to \infty$ we are forced to study the geodesic balls $\V_{n,R,R'}$ of radius $R'$ in $\V_{n,R}$ centered at $0$:
\begin{equation}\label{vnrrprime:defn}
\V_{n,R,R'}:=\{p\in \V_{n,R}\ \big{|}\ d_{(\V_{n,R},g_n)}(0,p)<R'\}.
\end{equation}

The following lemma addresses $\V_{n,R,R'}$ when $R'\gg R$. Requiring that $R'$ be big relative to $R$ ensures that the modification we are making to $\V_{n,R}$ is concentrated where the issue with the unboundedness of the diameters arises: near the point-sources. This is one possible interpretation of the property \eqref{partiideepwell} in the following lemma, which in turn the reader may want to contrast with \eqref{towardsIFL} and its consequence
\begin{equation}\label{towardsIFL-2}
B_{\geucl}(0,R)\smallsetminus \left(\bigcup \bar{B}_{\geucl}(p_{i,n},\tfrac{D}{n^2})\right) = \V_{n,R} \smallsetminus \left(\bigcup \bar{B}_{\geucl}(p_{i,n},\tfrac{D}{n^2})\right).
\end{equation}
 
\begin{lemma}\label{deepwell}
Fix $R>\sqrt{3}D$. For each $p_{i,n}\in \P_{n,*}$ define
\begin{equation}\label{defn-si}
s_{i,n,R'}:=\tfrac{D}{2n}\exp(-R'/(\alpha_{i,n}+\beta_{i,n})).
\end{equation}
There exists a constant $C$ of class $\C$ such that 
\begin{enumerate} 
\item\label{partideepwell} $\V_{n,R,R'}\subseteq B_{\geucl}(0,R)\smallsetminus \left(\bigcup B_{\geucl}(p_{i,n}, s_{i,n,R'})\right)$
\medbreak
\item\label{partiideepwell} $B_{\geucl}(0,R)\smallsetminus \left(\bigcup B_{\geucl}(p_{i,n}, \frac{D}{n^2})\right)=\V_{n,R,R'}\smallsetminus\left(\bigcup B_{\geucl}(p_{i,n}, \frac{D}{n^2})\right)$
\end{enumerate}
for all $R'\ge CR$.
\end{lemma}

\begin{proof} To prove the inclusion in part \eqref{partideepwell} consider a point within $B_{\geucl}(p_i, s_{i,n,R'})$ and a path $\gamma$ joining it and $0$. Since $s_{i,n,R'}<D/(2n)<|p_{i,n}|$ we know that a portion of $\gamma$ connects a point $q$ on Euclidean sphere $|x-p_{i,n}|=D/(2n)$ and a point $p$ on the Euclidean sphere $|x-p_i|=s_{i,n,R'}$. Using polar/spherical coordinates centered at $p_i$ we see that $|\dot{\gamma}|_{g_n}$ on the portion from $q$ to $p$ is bounded from below as follows:
$$|\dot{\gamma}|_{g_n}\ge \left(1+\frac{\alpha_{i,n}}{r}\right)\left(1+\frac{\beta_{i,n}}{r}\right)|\dot{\gamma}|_{\geucl}\ge \left(1+\frac{\alpha_{i,n}}{r}\right)\left(1+\frac{\beta_{i,n}}{r}\right)|\dot{r}|\ge \frac{\alpha_{i,n}+\beta_{i,n}}{r}|\dot{r}|.$$
Thus the length of $\gamma$ is bounded from below by 
$$\int_{r=s_{i,n,R'}}^{r=D/(2n)}\frac{\alpha_{i,n}+\beta_{i,n}}{r}\,dr=R'.$$
In other words, every path from $0$ to a point in $B_{\geucl}(p_{i,n}, s_{i,n,R'})$ has the length of at least $R'$. This completes the proof of the inclusion. 

Since $\exp(-R'/(\alpha_{i,n}+\beta_{i,n}))\le \tfrac{C}{n^3}$ for a constant $C$ of class $\C$, we see that  
$$s_{i,n,R'} < \tfrac{D}{n^2}$$
for all $n$ sufficiently large relative to $\C$, for all $i$ and all $R'\ge D$. To prove the equality \eqref{partiideepwell} we now, by virtue of part \eqref{partideepwell} of this lemma, only need to prove  
$$B_{\geucl}(0,R)\smallsetminus \left(\bigcup B_{\geucl}(p_{i,n}, \tfrac{D}{n^2})\right)\subseteq \V_{n,R,R'}$$
for some suitably large $R'$. To that end consider $p\in B_{\geucl}(0,R)\smallsetminus \left(\bigcup B_{\geucl}(p_i, \tfrac{D}{n^2})\right)$. A Euclidean broken line segment $\gamma$ largely following edges of the subdivision boxes and contained entirely in 
$\V_{n,R}$ can be constructed joining $0$ and $p$. It follows from Proposition \ref{BLRcontrollemma} that  
$$d_{(\V_{n,R},g_n)}(0,p)\le L_{g_n}(\gamma) \le CL_{\geucl}(\gamma)< 7CR  $$
for some constant $C$ of class $\C$. Taking $R'\ge 7CR$ completes the proof of the lemma. 
\end{proof}

\subsection{Volume estimates for $\V_{n,R}$ and $\V_{n,R,R'}$}
As a consequence of \eqref{inclusionsMn-ver2}, \eqref{towardsIFL-2} and Lemma \ref{deepwell} we have the following ``sandwiching" inclusions:
$$\begin{aligned}
& B_{\geucl}(0,R)\smallsetminus \left(\bigcup B_{\geucl}(p_{i,n}, \tfrac{D}{n^2})\right)\subseteq \V_{n,R}\subseteq B_{\geucl}(0,R)\smallsetminus \left(\bigcup \bar{B}_{\geucl}(p_{i,n}, \tfrac{1}{C}\sqrt{\alpha_{i,n}\beta_{i,n}})\right)\\
& B_{\geucl}(0,R)\smallsetminus \left(\bigcup B_{\geucl}(p_{i,n}, \tfrac{D}{n^2})\right)\subseteq \V_{n,R,R'}\subseteq B_{\geucl}(0,R)\smallsetminus \left(\bigcup B_{\geucl}(p_{i,n}, \hat{s}_{i,n,R'})\right)
\end{aligned}$$
where 
$$\hat{s}_{i,n,R'}:=\max\left\{\tfrac{1}{C}\sqrt{\alpha_{i,n}\beta_{i,n}}, s_{i,n,R'}\right\}.$$
For all practical purposes this means that any and all discrepancy between $\V_{n,R}$ (or $\V_{i,n,R'}$) and ``perforated" Euclidean balls is located entirely within the ``necks" of Brill-Lindquist-Riemann sums and can be easily controlled. For example, the control we have here enables us to easily estimate the volume of sets such as 
$$\V_{n,R}\cap \left(\bigcup B_{\geucl}(p_{i,n}, \tfrac{D}{n^2})\right).$$
We do so below. First, we record the following upper bound on the volume along an individual ``neck" of a general Reissner-Nordstr\"om metric. The proof of the estimate is a direct computation in spherical coordinates. 

\begin{lemma}\label{volumecomputation}
Let $0<s<t$. The volume of the region 
$$s\le |x|\le t$$
with respect the Reissner-Nordstr\"om metric
$(1+\tfrac{\alpha}{r})^2(1+\tfrac{\beta}{r})^2\geucl$
is bounded by a universal multiple of 
\begin{equation}\label{bigvolest}
t^3+ (\alpha+\beta)t^2+(\alpha+\beta)^2t
+(\alpha+\beta)^3 \ln (t/s)
+ \alpha\beta(\alpha+\beta)^2\tfrac{1}{s}+\alpha^2\beta^2(\alpha+\beta)\tfrac{1}{s^2}+
\alpha^3\beta^3\tfrac{1}{s^3}.
\end{equation}
\end{lemma}

In the case when the sequence of Brill-Lindquist-Riemann sums does not have deep wells (that is, when $\sup_n \ell_n<\infty$) the volume of $\V_{n,R}\smallsetminus \left(\bigcup B(p_{i,n},\tfrac{D}{n^2})\right)$ can be bounded as follows:
$$\mathrm{Vol}_{g_n}\left(\V_{n,R}\cap \left(\bigcup B(p_{i,n},\tfrac{D}{n^2})\right)\right)\le \sum_i \mathrm{Vol}_{g_n}\left(\left\{\tfrac{1}{C}\sqrt{\alpha_{i,n}\beta_{i,n}}\le |x-p_{i,n}|\le \tfrac{D}{n^2}\right\}\right).$$
It follows from Proposition \ref{BLRcontrollemma} that the constant $C$ may be increased so that (over the regions involved) we have 
$\mathrm{Vol}_{g_n}\le C\,\mathrm{Vol}_{g_{\RN,i}}$ 
where $g_{\RN,i}$ is the Reissner-Nordstr\"om metric 
$$g_{\RN,i}=\left(1+\frac{\alpha_{i,n}}{r}\right)^2\left(1+\frac{\beta_{i,n}}{r}\right)^2\geucl.$$
The estimate recorded in Lemma \ref{volumecomputation}, together with the fact that the summation contains $O(n^3)$-terms and the fact that $\alpha_{i,n}$, $\beta_{i,n}$ and $\sqrt{\alpha_{i,n}\beta_{i,n}}$ are uniformly of class $\tfrac{D}{n^3}\C$, ultimately leads to: 
$$\mathrm{Vol}_{g_n}\left(\V_{n,R}\cap \left(\bigcup B(p_{i,n},\tfrac{D}{n^2})\right)\right)\le C\left(\tfrac{D^3}{n^3}+(\alpha_{i,n}+\beta_{i,n})\big{|}\ln(\tfrac{n^2 \sqrt{\alpha_{i,n}\beta_{i,n}}}{CD})\big{|}\cdot\tfrac{D^2}{n^3}\right)$$
for some (yet larger) constant $C$ of class $\C$. Due to boundedness of expressions such as $\tfrac{1}{n^3}|\ln(n/C)|$ in $n$ we obtain  
\begin{equation}\label{volest-IFL1}
\begin{aligned}
\mathrm{Vol}_{g_n}\left(\V_{n,R}\cap\left(\bigcup B(p_{i,n},\tfrac{D}{n^2})\right)\right)\le &C\left(\tfrac{D^3}{n^3}+(\alpha_{i,n}+\beta_{i,n})|\ln(\alpha_{i,n}\beta_{i,n}/D^2)|\cdot \tfrac{D^2}{n^3}\right)\\
\le &CD^3\left(\tfrac{1}{n^3}+\tfrac{\ell_n}{n^3}\right).
\end{aligned}
\end{equation}
The same reasoning also leads to 
\begin{equation}\label{volest-IFL2}
\begin{aligned}
\mathrm{Vol}_{g_n}\left(\V_{n,R,R'}\cap\left(\bigcup B(p_i,\tfrac{D}{n^2})\right)\right)\le &C\left(\tfrac{D^3}{n^3}+(\alpha_{i,n}+\beta_{i,n})|\ln(n^2s_{i,n,R'}/D)|\cdot \tfrac{D^2}{n^3}\right)\\
\le &CD^3\left(\tfrac{1}{n^3}+\tfrac{R'/D}{n^3}\right).
\end{aligned}
\end{equation}

\section{The Gromov-Hausdorff limit}\label{GH:sec}

\subsection{A review of Gromov-Hausdorff convergence}
Tubular neighborhood of radius $r$ about a subset $A$ of a metric space $(Z,d_Z)$ is defined as 
$$\mathscr{T}_r(A)=\left\{z\in Z\,\big{|}\,\exists a\in A, d_Z(a,z)<r\right\}=\bigcup_{a\in A}B_{d_Z}(a, r).$$
If $A_1\subseteq \mathscr{T}_r(A_2)$, then the entire set $A_1$ is located within $r$ from $A_2$.
Under certain assumptions, such as boundedness of $Z$ or precompactness of $A_1, A_2\subseteq Z$, it is guaranteed that 
$$A_1\subseteq \mathscr{T}_r(A_2) \text{\ \ and\ \ } A_2\subseteq \mathscr{T}_r(A_1)$$
for some real number $r$. Smallness of such a value of $r$ communicates that $A_1$ and $A_2$ are in proximity of one another. This idea motivates us to define what is called \emph{Hausdorff distance} between $A_1,A_2\subseteq Z$:
$$d_Z^H(A_1, A_2)=\inf \{r\, |\, A_1\subseteq \mathscr{T}_r(A_2), A_2\subseteq \mathscr{T}_r(A_1)\}.$$
Note that $d_Z^H$ does not capture the distance/``difference"/discrepancy between \emph{metric spaces} $(A_1, d_Z|_{A_1\times A_1})$ and $(A_2, d_Z|_{A_2\times A_2})$; indeed, these two metric spaces could be (just about) isometric and yet for some reason located far away from each other within $Z$. To address such situations it is beneficial to consider metric isometric embeddings into a variety of metric spaces $(Z,d_Z)$, and then examine the behavior of the induced Hausdorff distances. Here are some precise definitions. 

By a \emph{metric isometric embedding} $\Psi$ of a metric space $(M,d_M)$ into a metric space $(Z,d_Z)$ we mean a function $\Psi:M\to Z$ such that 
$$d_Z(\Psi(p), \Psi(q))=d_M(p,q) \text{\ \ for all\ \ } p,q\in M.$$
It is very important to notice that the concept of a metric isometric embedding is different from that of a Riemannian isometric embedding. An example which explains this distinction is the inclusion mapping $\iota:S^1 \hookrightarrow \mathbb{R}^2$ between the unit circle $(S^1,d\theta^2)$ and $(\mathbb{R}^2, \geucl)$. Although it is a Riemannian isometric embedding, $\iota$ is not a metric isometric embedding because 
$$d_{(\mathbb{R}^2, \geucl)}(p,q)=2\neq \pi=d_{(S^1,d\theta^2)}(p,q)$$
whenever $p$ and $q$ are diametrically opposed. Very roughly speaking, the distance between diametrically opposed $p$ and $q$ within $(\mathbb{R}^2, \geucl)$ is achieved by shortcutting and bleeding out of $(S^1,d\theta^2)$. In general, when dealing with open submanifolds $(M,g)$ of a Riemannian manifold $(Z,g)$ we have to be careful and not assume that the inclusion $\iota:M\hookrightarrow Z$ is a metric isometric embedding of the $(M,g)$ into $(Z,g)$. (For further insight compare with Lemma \ref{distancecomparisonlemma} below.)

The \emph{Gromov-Hausdorff distance} $d^{GH}$ between two compact metric spaces $(M_1, d_1)$ and $(M_2, d_2)$ is defined by
$$d^{GH}((M_1, d_1),(M_2, d_2))=\inf_{Z, \Psi_1, \Psi_2,} d^{H}_Z(\Psi_1(M_1), \Psi_2(M_2))$$
where the infimum is taken over all metric isometric embeddings 
$$\Psi_1:M_1\to Z \text{\ \ and\ \ } \Psi_2:M_2\to Z$$ 
of $(M_1, d_1)$ and $(M_2, d_2)$ into a common metric space $(Z,d_Z)$. It can be shown that the Gromov-Hausdorff distance between two compact metric spaces vanishes if and only if the two spaces are isometric. In fact, Gromov-Hausdorff distance equips the set of isometry (equivalence) classes of compact metric spaces with a structure of a metric space; in the literature this metric space is often referred to as \emph{the Gromov-Hausdorff space}. Gromov-Hausdorff convergence refers to convergence within the Gromov-Hausdorff space, although in practice we often talk of convergence of sequences of compact metric spaces. In other words, \emph{Gromov-Hausdorff convergence} of $(M_j,d_j)$ towards a compact $(M,d)$ means convergence relative to $d^{GH}$.

The Gromov-Hausdorff distance is described particularly well by the concept of \emph{$\e$-isometry}, which we now define. A function $F:(M_1,d_1)\to (M_2,d_2)$ is called an $\e$-isometry if: 
\begin{itemize}
\item $M_2=\mathscr{T}_{\e}(\mathrm{Im}F)$;
\medbreak
\item We have $|d_1(x,y)-d_2(F(x),F(y))|<\e$ for all $x,y\in M_1$.
\end{itemize}
It is worth emphasizing that continuity of $F$ is not a requirement here. The following two properties connect the concepts of $\e$-isometries and Gromov-Hausdorff distance $d^{GH}$.
\begin{enumerate}
\item If $d^{GH}((M_1, d_1),(M_2, d_2))<\e$ then there exists a $2\e$-isometry $$F:(M_1, d_1)\to(M_2, d_2);$$
\medbreak
\item\label{yay!} If there exists an $\e$-isometry $F:(M_1, d_1)\to(M_2, d_2)$ then 
$$d^{GH}((M_1, d_1),(M_2, d_2))<2\e.$$
\end{enumerate}
We employ property \eqref{yay!} at several places in our article. 

The (pre)compactness theorem of Gromov is another result which is highly relevant to our work. In order to state the theorem efficiently we first introduce \emph{$r$-capacity of a compact metric space} $(M,d)$: 
$$\mathrm{Cap}_{(M,d)}(r)=\max \left\{k\,\big{|}\, \exists x_1, ..., x_k\in M, \forall i\neq j, d(x_i, x_j)\ge r\right\}.$$
Informally speaking, $\mathrm{Cap}_{(M,d)}(r)$ measures how spread out $M$ is by measuring the maximal number of points we can place in $M$ at distance of at least $r$ from one another. 
(The fact that the maximal number here is achieved is a consequence of the fact that $(M,d)$ is compact.)
Alternatively, $r$-capacity can be defined as 
$$\mathrm{Cap}_{(M,d)}(r)=\max \left\{k\,\big{|}\, \exists x_1, ..., x_k\in M, i\neq j\longrightarrow B_{(M,d)}(x_i,r/2)\cap B_{(M,d)}(x_j,r/2)=\emptyset \right\}.$$
Viewed from this perspective $\mathrm{Cap}_{(M,d)}(r)$ measures the maximal number or disjoint balls of radius $r/2$ which can be placed in $M$.

\begin{theorem}[Gromov's Precompactness Theorem]\label{GPT}
A subset $\mathscr{K}$ of the Gromov-Hausdorff space is precompact if and only if 
\begin{enumerate}
\item For all $r>0$ there exists $N(r)\in \mathbb{N}$ such that $$\mathrm{Cap}_{(M,d)}(r)\le N(r)$$ for all $(M,d)\in \mathscr{K}$;
\medbreak
\item There exists $D>0$ such that $\mathrm \mathrm{diam}_{(M,d)}\le D$ for all $(M,d)\in \mathscr{K}$.
\end{enumerate}
\end{theorem}

In relation to this theorem it is sometimes helpful to know how $r$-capacity and diameters behave under Gromov-Hausdorff limits. One can show that if a sequence $(M_n, d_n)$ of compact metric spaces  converges in the Gromov-Hausdorff sense to the compact metric space $(M,d)$ then 
\begin{enumerate}
\item $\limsup_{n\to \infty} \mathrm{Cap}_{(M_n, d_n)}(r)\le \mathrm{Cap}_{(M,d)}(r)$ for all $r>0$;
\medbreak
\item $\lim_{n\to \infty} \mathrm{diam}_{(M_n,d_n)} = \mathrm{diam}_{(M,d)}$.
\end{enumerate}

For further details and helpful examples the reader is referred to \cite{ChristinaExpository} and references therein. 
 
\subsection{Distance Comparison Lemma}

In our review of Gromov-Hausdorff convergence we mentioned that the inclusion $\iota:\U\hookrightarrow \V$ of an open submanifold $(\U,g)$ into a Riemannian manifold $(\V,g)$ need not be a metric isometric embedding of $(\U,g)$ into $(\V,g)$. This is particularly true when $\U$ is some kind of a perforated version $\V$. It is clear that ``perforated context" is very relevant to studies of (truncated) Brill-Lindquist-Riemann sums. In fact, we rely on the following distance comparison result at several key places in our paper. For example, the result can be used to prove that an inclusion is at least an $\e$-isometry if not a metric isometric embedding. To accommodate a variety of applications within this paper we keep the language of the lemma pretty general. Its proof is a modification of an argument used in \cite{SormaniStavrov}.

\begin{lemma}\label{distancecomparisonlemma}
Let $\V\subseteq \mathbb{R}^3$ be an open set and let $g$ be a metric on $\V$. Consider a finite union 
$$\P=\bigcup \bar{B}_{\geucl}(p_i, r_i) \text{\ \ with\ \ } r_\P\le \tfrac{1}{4}\sigma_\P,$$  
where 
$$r_{\P}:=\max_{i}r_i \text{\ \ and\ \ }\sigma_\P:=\min_{i,j}|p_i-p_j|.$$ 
Assume that $g$ is equivalent to $\geucl$ over the set $$\U:=\V\smallsetminus \P,$$ in the sense that for some constant $c$ we have $c^{-2}\geucl\le g\le c^2\geucl$ over $\U$. 
Then for all $x,y\in \U$ we have 
\begin{equation}\label{babylambda}
0\le d_{(\U, g)} (x,y)- d_{(\V,g)}(x,y)\le 2\pi c^2\,\frac{r_\P}{\sigma_\P}\,d_{(\V,g)}(x,y)+ \pi c\, r_\P.
\end{equation}
\end{lemma}

\begin{remark}\label{alttakegarricklemma}
The proof below basically consists of finding, for a given piece-wise smooth path $\gamma$ in $\V$ which is connecting points $x,y\in \U$, a piece-wise smooth path $\varphi$ in $\U$ which is still connecting $x,y\in \U$ and whose length with respect to $g$ satisfies
$$L_g(\varphi)-L_g(\gamma) \le 2\pi c^2\,\tfrac{r_\P}{\sigma_\P}\,L(\gamma)+ \pi c\, r_\P.$$
This is a worthy result in its own right, and we make use of it later on.
\end{remark}

\begin{proof}
Let $x,y\in \U$ and let $\gamma:[0,1]\to \V$ denote a (piece-wise smooth) path joining $x$ to $y$. As one traverses $\gamma$ from $x$ to $y$ one punctures (i.e. transversally meets, without loss of generality) a certain number $Q$ of Euclidean spheres $S_{\geucl}(p_i, r_i)$; this yields a subdivision 
$$0=t_0< t_1 \le t_2 < t_3 \le t_4 < ... < t_{2Q+1}=1$$
where the restriction $\gamma{|}_{[t_{2j-1},  t_{2j}]}$ is located inside the $j$-th sphere along $\gamma$ and where the restriction $\gamma{|}_{[t_{2j},  t_{2j+1}]}$ is located in $\V\smallsetminus \P$ (and in particular: outside of all of the spheres). Note that some of the spheres may appear more than once; in fact, $\gamma$ can immediately re-puncture $S_{\geucl}(p_i, r_i)$. (See the Figure \ref{fig5} accompanying this proof.) For this reason we distinguish the number $Q'$ of indices $j$, with $0<j<Q$, such that $\gamma(t_{2j})$ and $\gamma(t_{2j+1})$ are on distinct spheres $S_{\geucl}(p_i, r_i)$. Our next step is to control the value of $Q'$; the goal is to obtain an estimate in terms of the length $L_{g}(\gamma)$ and the separation parameter $\sigma_\P$. 

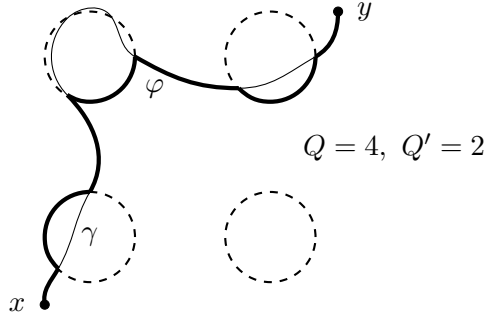
\begin{figure}
\centering
\begin{tikzpicture}[scale=0.6]
\draw[thick, dashed] (-2,0) circle (1);
\draw[thick, dashed] (2,0) circle (1);
\draw[thick, dashed] (-2,4) circle (1);
\draw[thick, dashed] (2,4) circle (1);
\draw (-3, -1.5) to [out=100, in=-120] (-2.7, -0.7) to [out=60, in =-120] (-2, 1) to [out=60, in =-45] (-2.5, 3.15) to [out =135, in =-135] (-2.5, 4.85) to [out=45, in =120] (-1.5, 4.85) to [out=-60, in=150] (-1, 4) to [out=-30, in=180] (1.3, 3.3) to [out =0, in =-150] (3,4) to [out=30, in =-90] (3.5,5);
\draw[ultra thick] (-3, -1.5) to [out=100, in=-120] (-2.7, -0.7); 
\draw[ultra thick] (-2.7, -0.7) to [out=135, in =-90] (-3, 0) to [out=90, in =-180] (-2, 1);
\draw[ultra thick] (-2, 1) to [out=60, in =-45] (-2.5, 3.15);
\draw[ultra thick]  (-2.5, 3.15) to [out=-30, in =-180] (-2, 3) to [out=0, in =-90] (-1, 4);
\draw[ultra thick] (-1, 4) to [out=-30, in=180] (1.3, 3.3);
\draw[ultra thick] (1.3, 3.3) to [out=-45, in =-180] (2, 3) to [out =0, in =-90] (3, 4);
\draw[ultra thick] (3,4) to [out=30, in =-90] (3.5,5);
\draw[fill] (-3, -1.5) circle (0.1);
\node[left] at (-3.2, -1.5) {$x$};
\draw[fill] (3.5, 5) circle (0.1);
\node[right] at (3.7, 5) {$y$};
\node at (-2,0) {$\gamma$};
\node[right] at (-1, 3.3) {$\varphi$};
\node[right] at (2.5, 2) {$Q=4,\ Q'=2$};
\end{tikzpicture}
\caption{Illustration for Lemma \ref{distancecomparisonlemma}}\label{fig5}
\end{figure}

Let $0< j< Q$ be an index such that $\gamma(t_{2j})$ and $\gamma(t_{2j+1})$ are on distinct spheres $S_{\geucl}(p_i, r_i)$. As $g$ is equivalent to $\geucl$ and as $r_i\le r_\P\le \tfrac{1}{4}\sigma_\P$ for all $i$ we see that
$$L_{g}\left(\gamma\big{|}_{[t_{2j}, t_{2j+1}]}\right)\ge \tfrac{1}{c}\,L_{\geucl}\left(\gamma\big{|}_{[t_{2j}, t_{2j+1}]}\right)\ge \tfrac{1}{c}\left(\sigma_\P-2r_\P\right)\ge \tfrac{\sigma_\P}{2c}.$$
It follows that 
$$L_{g}(\gamma)\ge Q'\tfrac{\sigma_\P}{2c}, \text{\ \ i.e\ \ }
Q'\le \tfrac{2c}{\sigma_\P} L_{g}(\gamma).$$ 

Next, consider a piece-wise smooth path $\varphi$ joining $x$ and $y$, which lies entirely in $\U$ and consists of: 
\begin{itemize}
\item The restrictions $\gamma\big{|}_{[t_0, t_1]}$ and $\gamma\big{|}_{[t_{2Q}, t_{2Q+1}]}$, 
\medbreak
\item $Q'$ restrictions $\gamma\big{|}_{[t_{2j}, t_{2j+1}]}$, joining  distinct spheres $S_{\geucl}(p_i, r_i)$, and 
\medbreak
\item In the case when $Q>0$: $Q'+1$ detours joining $\gamma(t_{2j-1})$ and some $\gamma(t_{2k})$ along a single sphere $S_{\geucl}(p_i, r_i)$. We may assume the detours are (at most semi-) circular in Euclidean sense.
\end{itemize}
The distinction between paths $\gamma$ and $\varphi$ is exactly in these $Q'+1$ detours.  Using the equivalence of $g$ to $\geucl$ we obtain the following bound on the total $g$-length of all the detours: 
$$\pi (Q'+1)\, c\, r_\P\le  2\pi c^2\,\tfrac{r_\P}{\sigma_\P} L_{g}(\gamma)+\pi c\,r_\P.$$ 
Adding the contributions of the remaining pieces of $\varphi$ (which by construction are portions of $\gamma$) we obtain
\begin{equation}\label{pathineq}
L_{g}(\varphi)\le L_{g}(\gamma)+2\pi  c^2\,\tfrac{r_\P}{\sigma_\P} L_{g}(\gamma)+\pi c\,r_\P.
\end{equation}
Taking infimums of both sides of \eqref{pathineq} with respect to $\gamma$ (and consequently $\varphi$)  produces 
$$d_{(\U, g)} (x,y)\le d_{(\V,g)}(x,y)+2\pi c^2\,\tfrac{r_\P}{\sigma_\P} d_{(\V,g)}(x,y)+ \pi c\,r_\P.$$
Observing that $d_{(\V, g)}(x,y)\le d_{(\U, g)} (x,y)$ completes the proof of \eqref{babylambda}.
\end{proof}

\subsection{Theorem \ref{GHduhduhduh}: convergence of $(\U_{n,R},g_n)$ towards $(B_{\geucl}(0,R),g)$}\label{DUH}

We now apply Lemma \ref{distancecomparisonlemma} to $\U=\U_{n,R}$ and $\V=B_{\geucl}(0,R)$ where 
$$\U_{n,R}:=B_{\geucl}(0,R)\smallsetminus \left(\bigcup_i B_{\geucl}(p_{i,n}, \tfrac{D}{n^2})\right).$$
Note that in this setting \eqref{usedtobe39} holds, and that 
$$r_\P=D/n^2,\ \ \sigma_\P=D/n.$$
Since $\tfrac{r_\P}{\sigma_\P}, r_\P\to 0$ while $\mathrm{diam}(B_{\geucl}(0,R),g)$ is finite, fixed and of class $\C[R]$, the quantity on the right hand side of \eqref{babylambda} can be made as small as desired by taking $n$ to be sufficiently large relative to $\C[R]$. More specifically, it follows that the inclusion mapping 
$$\iota_n:\U_{n,R}\hookrightarrow B_{\geucl}(0,R)$$ 
between the metric spaces $(\U_{n,R},g)$ and $(B_{\geucl}(0,R),g)$ satisfies at least one of the two conditions of being a $\tfrac{\C[R]}{n}$-isometry. 

To see that this mapping indeed is a $\tfrac{\C[R]}{n}$-isometry when $n$ is large, note that for each point $p\in B(p_{i,n}, \tfrac{D}{n^2})$ there is a point $p'\in \U_{n,R}$ such that $d_{(B_{\geucl}(0,R),g)}(p,p')\le CD/n^2$ with $C$ is of class $\C$. For example, we can take the point $p'$ to be on the Euclidean sphere of radius $2D/n^2$ centered at $p_{i,n}$ so that $p_{i,n}$, $p$ and $p'$ are collinear. We now have 
$$B_{\geucl}(0,R)=\mathscr{T}_{CD/n^2}(\U_{n,R})$$ 
within the metric space $(B_{\geucl}(0,R),g)$. Overall, it follows that $\iota_n:\U_{n,R}\to B_{\geucl}(0,R)$ is a $\tfrac{\C[R]}{n}$-isometry and that the sequence of metric spaces $(\U_{n,R},g)$ converges to $(B_{\geucl}(0,R),g)$ in the Gromov-Hausdorff sense. 

In part \eqref{qualityofapproximation} of Proposition \ref{BLRcontrollemma} we saw that 
$$g_n=(\chi_n\psi_n)^2\geucl \approx (\chi\psi)^2\geucl=g$$
over $\U_{n,R}$. It is reasonable to expect that such proximity of the metrics $g_n$ and $g$ implies the proximity of metric spaces $(\U_{n,R},g_n)$ and  $(\U_{n,R},g)$ in the Gromov-Hausdorff sense. The following lemma quantifies this particular point. Once again, we keep the language of the lemma general because of its further applications within this paper.

\begin{lemma}\label{GHduh}
Suppose that $g$ and $h$ are two Riemannian metrics on $\U\subseteq \mathbb{R}^3$. 
If $\gamma$ is any piece-wise smooth path in $\U$ then 
$$L_h(\gamma)\le L_g(\gamma)\left(1+\tfrac{1}{2}\|\mathrm{Id}-g^{-1}h\|_{L^\infty(\U,\geucl)}\right).$$
In particular, we have 
$$d_{(\U,h)}(x,y)-d_{(\U,g)}(x,y)\le \tfrac{d_{(\U,g)}(x,y)}{2}\|\mathrm{Id}-g^{-1}h\|_{L^\infty(\U,\geucl)}$$
for all $x,y\in\U$.
\end{lemma}

\begin{proof}
Suppose $x,y\in \U$ and let $\gamma:[0,L]\to \U$ denote a (piece-wise smooth) path joining $x$ to $y$. Without loss of generality we may assume that $|\dot{\gamma}|_g=1$ so that $L_g(\gamma)=L$. Since
$$\left||\dot{\gamma}|_g^2-|\dot{\gamma}|_h^2\right|\le \|\mathrm{Id}-g^{-1}h\|_{L^\infty}\cdot |\dot{\gamma}|^2_{g}$$
we have that
$$|\dot{\gamma}|_h\le \sqrt{1+\|\mathrm{Id}-g^{-1}h\|_{L^\infty}}\le 1+\tfrac{1}{2}\|\mathrm{Id}-g^{-1}h\|_{L^\infty}.$$
Claims of our lemma follow after applying integration and taking infimums over $\gamma$.
\end{proof} 

We are about to apply Lemma \ref{GHduh} to $\U=\U_{n,R}$ equipped with metrics $g_n\ge \geucl$ and $g$. It follows from \eqref{babylambda} and the bound on the diameter of $(B_{\geucl}(0,R), g)$ that there is a bound of class $\C[R]$ on $\mathrm{diam}(\U_{n,R},g)$. Given the nature of convergence $g_n\to g$ (Proposition \ref{BLRcontrollemma}) over $\U_{n,R}$, the metrics $g_n$ can be bounded by a (uniform and of class $\C^+$) multiple of $g$. Thus there is a uniform bound of class $\C^+[R]$ on all $\mathrm{diam}(\U_{n,R},g_n)$. Lemma \ref{GHduh}, together with the understanding that $\|g_n-g\|_{L^\infty}=O(\tfrac{1}{n})$ based on Proposition \ref{BLRcontrollemma}, implies 
$$\left|d_{(\U_{n,R},g_n)}(x,y)-d_{(\U_{n,R},g)}(x,y)\right|\le \tfrac{C}{n} \text{\ \ for all\ \ }x,y\in \U_{n,R}$$
for some $C$ of class $\C^+[R]$. In particular, we see that the identity mapping on $\U_{n,R}$ serves as an $\frac{\C^+[R]}{n}$-isometry between $(\U_{n,R},g_n)$ and $(\U_{n,R},g)$. The fact that $(\U_{n,R},g)\to (B_{\geucl}(0,R),g)$ in the Gromov-Hausdorff sense now implies that metric spaces $(\U_{n,R},g_n)$ converge to $(B_{\geucl}(0,R),g)$. We have just proved Theorem \ref{GHduhduhduh}. \qed

\subsection{Gromov-Hausdorff convergence in the case of shallow wells}\label{GHthm}

This section is dedicated to the proof of Theorem \ref{GHthm1}. In light of the (proof of) Theorem \ref{GHduhduhduh} (see Section \ref{DUH} above) it remains to address the proximity of $(\U_{n,R},g_n)$ and $(\V_{n,R},g_n)$ as metric spaces. We show that the Gromov-Hausdorff distance between the two can be made apropriately small by proving that the inclusion 
$\iota_n:\U_{n,R}\hookrightarrow \V_{n,R}$ is a $C(\frac{1}{n}+\ell_n)$-isometry with $C$ of class $\C[R]$. The condition that 
$$|d_{(\U_{n,R},g_n)}(x,y)-d_{(\V_{n,R},g_n)}(x,y)|<C\tfrac{1+\ell_n}{n} \text{\ \ for all\ \ } x,y\in \U_{n,R}$$
is once again a consequence of Lemma \ref{distancecomparisonlemma}: we use 
$$r_\P=D/n^2,\ \ \sigma_\P=D/n,$$
and the fact that $\mathrm{diam}(\V_{n,R},g_n)$ is bounded by a $\C$-multiple of $R+\ell_n D< R(1+\ell_n)$. Thus it remains to show that 
$$\V_{n,R}=\mathscr{T}_{C(\frac{1}{n}+\ell_n)}(\U_{n,R})$$
within $(\V_{n,R},g_n)$. 

Consider $p\in \V_{n,R}\cap B_{\geucl}(p_{i,n},\tfrac{D}{n^2})$ for some $i$. Let the point $p'$ be the location where the ray from $p_{i,n}$ towards $p$ pierces the Euclidean sphere of radius $\tfrac{2D}{n^2}$ centered at $p_{i,n}$. By Lemma \ref{necklengthlemma} we have 
$$d_{(\V_{n,R},g_n)}(p,p')\le CD(\tfrac{1}{n^2}+\ell_n)$$
for some constant $C$ of class $\C$. It now follows that $\V_{n,R}=\mathscr{T}_{(\frac{1}{n}+\ell_n)\C[R]}(\U_{n,R})$, and that $\iota_n$ is a $(\frac{1}{n}+\ell_n)\C[R]$-isometry. \qed

\subsection{A non-example of Gromov-Hausdorff convergence}\label{GHdne}

In situations where we do not have shallow wells the Gromov-Hausdorff convergence is generally speaking not expected. To understand the reasons behind this consider Example \ref{Example3}. Our analysis of this example presented an explicit subsequence, indexed by $n_k$, of the sequence of Brill-Lindquist-Riemann sums of midpoint type with at least $k$ distinct locations $p_{i,n_k}$ where 
$$\ell_{i,n_k}\ge 8.$$
For each such $p_{i,n_k}$ consider a point $q_{i,n_k}\in \V_{n,R}$ such that $|q_{i,n_k}-p_{i,n_k}|=2\sqrt{\alpha_{i,n_k}\beta_{i,n_k}}$. Our next goal is to show that for each such $p_{i,n_k}$ the geodesic ball $B_{g_n}(q_{i,n_k}, \tfrac{2D}{3})$ in $\V_{n_k,R}$ is contained within $B_{\geucl}(p_{i,n_k},\tfrac{D}{2n_k^2})$. To this end it suffices to argue that for each $q'\in \V_{n,R}$ with $|q'-p_{i,n_k}|=D/(2n_k^2)$ we have $d_{(\V_{n_k,R},g_{n_k})}(q_{i,n_k},q')>2D/3$, at least if $k$ is really large. 

Let $q$ be collinear with and between $p_{i,n_k}$ and $q'$ with  $|q-p_{i,n_k}|=2\sqrt{\alpha_{i,n_k}\beta_{i,n_k}}$. It follows from Lemma \ref{necklengthlemma} that 
$$d_{(\V_{n_k,R},g_{n_k})}(q,q')\ge \tfrac{D}{12}\ell_{i,n_k}=\tfrac{3D}{4}.$$
On the other hand, Remark \ref{goingaround} gives us an estimate 
$$d_{(\V_{n_k,R},g_{n_k})}(q_{i,n_k},q)=O(D/n_k^2)$$ 
with the implied proportionality constant of class $\C$. The claim that 
$$d_{(\V_{n_k,R},g_{n_k})}(q_{i,n_k},q')>2D/3$$ 
for large $k$ is now a consequence of the triangle inequality applied to points $q_{i,n_k}$, $q$ and $q'$. 

Ultimately, we see that $\mathrm{Cap}_{(\V_{n_k,R},g_{n_k})} (2D/3)\ge k$ and as a result the sequence $(\V_{n_k,R},g_{n_k})$ cannot converge in the Gromov-Hausdorff sense. (In fact, it cannot even have any convergent subsequences!)

\subsection{Gromov-Hausdorff limit may depend on the choice of sample points in \eqref{commonsensechoice}}\label{dependencyonsamplepoints}
Admittedly, there are situations when we do not have shallow wells and yet we do have Gromov-Hausdorff convergence. The example we present here is based on Example \ref{Example2} for a fixed value of $\lambda$ though we could have equally made use of Example \ref{Example1}. Ultimately, the lesson we learn here is that in situations where neither deep nor shallow wells occur the Gromov-Hausdorff limit may highly depend on the procedure used to find the exact value of the parameters $\alpha_{i,n}$ and $\beta_{i,n}$.

Our analysis of Example \ref{Example2} and its Gromov-Hausdorff limit revolves around the set 
$$\W_{n,R}:=\U_{n,R}\cup L_{i,n}$$
where
$$L_{i,n}:=\{(1-t)p_{i,n}+tq_{i,n}\,\big{|}\,2\sqrt{\alpha_{i,n}\beta_{i,n}}\le t|q_{i,n}-p_{i,n}| \le D/n^2\}.$$
Informally speaking, the set $\W_{n,R}$ is formed by adding a line going down the neck at $p_{i,n}$ to $\U_{n,R}$. We begin by showing that the inclusion $\iota_n:(\W_{n,R},g_n)\to (\V_{n,R},g_n)$ is an $\e$-isometry when $n$ is sufficiently large. What we are taking advantage of here is the fact that ``neck" at $p_{i,n}$ is thin enough so that the sequence of metric spaces $(\mathcal{B}_{i,n},g_n)$ where 
\begin{equation}\label{defnmathcalB}
\mathcal{B}_{i,n}:=\bar{B}_{\geucl}(p_{i,n},\tfrac{D}{n^2})\cap \V_{n,R},
\end{equation}
can be shown to converge in Gromov-Hausdorff sense to a line segment. 

Next, we create an $\e$-isometry between $(\W_{n,R},g_n)$ and the metric space defined as follows: Let 
$$\W_{\infty,R}=B_{\geucl}(0,R)\times \{0\}\cup {(1,0,0)}\times [0,L]\subseteq B_{\geucl}(0,R)\times [0,L]$$ be the set formed by attaching a line segment of length 
$$L=\tfrac{1}{2\lambda}\chi(1,0,0)$$ 
to the Euclidean ball at $(1,0,0)$. Consider the 
taxi-cab-style metric on  $B_{\geucl}(0,R)\times [0,L]$ given by 
$$d_\infty((x,t), (y,s))=d_g(x,y)+|t-s|$$
and its restriction (which we also denote by $d_\infty$) to $\W_{\infty,R}$. 

Ultimately, the point is that the sequence of metric spaces $(\V_{n,R},g_n)$ converges in the Gromov-Hausdorff sense to $(\W_{\infty,R}, d_\infty)$. From the technical perspective the crux of our argument lies in the following lemma.

\begin{lemma}\label{lastlemma?}
Adopt the notation of Example \ref{Example2}, fix $\lambda>0$ and restrict your attention to values of $n$ for which $n^3\ge \lambda$. There exists a constant $C$ of class $\C$ such that for all $n$ which are large relative to $\C$ the following holds: given a piece-wise smooth path $\gamma$ in $\V_{n,R}$ connecting points $x,y\in \W_{n,R}$ there exists a piece-wise smooth path $\varphi$ in $\W_{n,R}$ connecting $x$ and $y$ whose length with respect to $g_n$ satisfies
$$L_{g_n}(\varphi)-L_{g_n}(\gamma) \le C\left(\tfrac{1}{n}L_{g_n}(\gamma)+\tfrac{D}{n^2}(1+\tfrac{1}{\lambda})\right).$$
\end{lemma}

\begin{proof}
If $x,y\in \U_{n,R}$ then our claim is a consequence of Remark \ref{alttakegarricklemma}. Thus it suffices to focus on paths $\gamma$ located entirely in $\mathcal{B}_{i,n}$ (see \eqref{defnmathcalB} above) and whose endpoints $x$ and $y$ satisfy one of the following:
\begin{itemize}
\item $x,y\in L_{i,n}$;
\medbreak
\item $x\in L_{i,n}$ while $|y-p_{i,n}|=\tfrac{D}{n^2}$.
\end{itemize}
We proceed by investigating these two cases individually.

{\sc The case of $x,y\in L_{i,n}$:} Recall from Lemma \ref{BLRcontrollemmaC1} that 
$|d\chi_n^{(i)}|$ and $|d\psi_n^{(i)}|$ are bounded by a constant of class $\C$ over the entire ball $B_{\geucl}(p_{i,n},\tfrac{1}{n^2})$. In particular, the Mean Value Theorem implies 
$$\left|\chi_n^{(i)}(x)-\left(\hat{\chi}_n^{(i)} +\frac{\alpha_{i,n}}{|x-p_{i,n}|}\right)\right| \le  C\frac{1}{n^2},\ \ x\in \mathcal{B}_{i,n},$$
as well as a similar estimate for $\psi_n^{(i)}$. This motivates the consideration of the metric 
\begin{equation}\label{defngRN}
g_{\RN,i}:=\left(\hat{\chi}_n^{(i)} +\frac{\alpha_{i,n}}{|x-p_{i,n}|}\right)^2\left(\hat{\psi}_n^{(i)} +\frac{\beta_{i,n}}{|x-p_{i,n}|}\right)^2\geucl.
\end{equation}
Note that we have 
$$(1-\tfrac{C}{n^2})^4 g_{\RN,i}\le g_n\le (1+\tfrac{C}{n^2})^4 g_{\RN,i},$$
or in other words: 
$$\|\mathrm{Id}-g_{\RN,i}^{-1}g_n\|\le \tfrac{C}{n^2}$$ 
for some (potentially larger) constant $C$.
Due to spherical symmetry of $g_{\RN,i}$, the length minimizer $\varphi$ between points $x,y\in L_{i,n}$ lies within $L_{i,n}$. Together with Lemma \ref{GHduh} we have 
$$\begin{aligned}
L_{g_n}(\varphi)\le &L_{g_{\RN,i}}(\varphi) \left(1+C/(2n^2))\right)\\
\le & L_{g_{\RN,i}}(\gamma) \left(1+C/(2n^2)\right)\le L_{g_n}(\gamma)\left(1+C/(2n^2)\right)^2
\end{aligned}$$
and consequently
$$(1-C/n^2)L_{g_n}(\varphi)\le L_{g_n}(\gamma).$$
Since $L_{g_n}(\varphi)\le CD(\tfrac{1}{n^2}+\ell_n)$ by Lemma \ref{necklengthlemma}, we have 
$L_{g_n}(\varphi)\le CD\left(1+\tfrac{1}{\lambda}\right)$.  It then further follows that 
$$L_{g_n}(\varphi)- L_{g_n}(\gamma)\le \tfrac{CD}{n^2}\left(1+\tfrac{1}{\lambda}\right).$$

{\sc The case of $x\in L_{i,n}$ and $|y-p_{i,n}|=1/n^2$:} Let $y'\in L_{i,n}$ be such that $|y'-p_{i,n}|=D/n^2$. Consider the path $\tilde{\gamma}$ obtained from $\gamma$ by appending the circular arc between $y$ and $y'$. It follows from Remark \ref{goingaround} that 
$$L_{g_n}(\tilde{\gamma})\le  L_{g_n}(\gamma) +\tfrac{CD}{n^2}.$$
Consider the path $\varphi$ formed by taking $\tilde{\varphi}$ connecting $x$ and $y'$ as in the previous case and appending the circular arc connecting $y'$ back to $y$ to it.
Applying Remark \ref{goingaround} once again we obtain 
$$L_{g_n}(\varphi)-L_{g_n}(\gamma)\le L_{g_n}(\tilde{\varphi})+\tfrac{CD}{n^2}-(L_{g_n}(\tilde{\gamma})-\tfrac{C}{n^2})\le \tfrac{CD}{n^2}\left(1+\tfrac{1}{\lambda}\right) +\tfrac{3CD}{n^2}.$$
This observation completes our proof.
\end{proof}

Let $$\iota_n:(\W_{n,R},g_n)\to (\V_{n,R},g_n)$$ denote the natural inclusion. It follows from Lemma \ref{diameterlemma} and Lemma \ref{lastlemma?} that 
$$d_{(\W_{n,R},g_n)}(x,y)-d_{(\V_{n,R},g_n)}(x,y)\le C\left(\tfrac{1}{n}(R+\tfrac{D}{\lambda})+\tfrac{D}{n^2}(1+\tfrac{1}{\lambda})\right)\le \tfrac{2CR}{n}\left(1+\tfrac{1}{\lambda}\right)$$
for all $x,y\in \W_{n,R}$. Temporarily set\footnote{Note that there is no harm in replacing $1+\tfrac{1}{\lambda}$ by $\lambda+\tfrac{1}{\lambda}$ as we are already assuming $n$ is large relative to $\lambda$, $n^3\ge \lambda$.} $\e=\tfrac{2CR}{n}\left(1+\tfrac{1}{\lambda}\right)$. The inclusion $\iota_n$ can be thought of as an $\e$-isometry provided we can show that 
$$\V_{n,R}=\mathscr{T}_\e(\W_{n,R}).$$ 
Recall that $\ell_{j,n}=O(\tfrac{1}{n})$ when $j\neq i$. For the reasons presented in the proof of Theorem \ref{GHthm1} in Section \ref{GHthm} we know that for all $p\in \mathcal{B}_{j,n}$ (see \eqref{defnmathcalB}) there exists a point $p'\in \U_{n,R}$ such that 
$$d_{(\V_{n,R}, g_n)}(p,p')<\tfrac{CD}{n}.$$ 
Thus it suffices to prove that for each point 
$p$ in $\mathcal{B}_{i,n}$  there exists a point $p''\in L_{i,n}$ for which $d_{(\V_{n,R},g_n)}(p,p'')\le \tfrac{CD}{n}$. So, let 
$p\in \mathcal{B}_{i,n}$. If $|p-p_{i,n}|< 2\sqrt{\alpha_{i,n}\beta_{i,n}}$ consider in addition the point $p'$ where the ray from $p_{i,n}$ to $p$ pierces the Euclidean sphere $S_{\geucl}(p_{i,n},2\sqrt{\alpha_{i,n}\beta_{i,n}})$; for convenience define $p'=p$ whenever $|p-p_{i,n}|\ge 2\sqrt{\alpha_{i,n}\beta_{i,n}}$. Note that 
$$d_{(\V_{n,R},g_n)}(p,p')\le 12(\alpha_{i,n}+\beta_{i,n})=O(D/n^3),$$
as in Remark \ref{goingaround}. Next, consider $p''\in L_{i,n}$ such that $|p''-p_{i,n}|=|p'-p_{i,n}|$. 
Remark \ref{goingaround} further implies $d_{(\V_{n,R},g_n)}(p',p'')=O(D/n^2)$ and, by the triangle inequality, 
$$d_{(\V_{n,R},g_n)}(p,p'')\le \tfrac{CD}{n^2}\le \tfrac{CD}{n}.$$ 
We are now in position to conclude that the inclusion $\iota_n$ is an $\e$-isometry. 
\bigbreak
Next, note that the length of $L_{i,n}$ with respect to the metric $g_{\RN,i}$ defined in \eqref{defngRN} is 
$$\int_{r=2\sqrt{\alpha_{i,n}\beta_{i,n}}}^{r=D/n^2} \left(\hat{\chi}_n^{(i)} +\frac{\alpha_{i,n}}{r}\right)\left(\hat{\psi}_n^{(i)} +\frac{\beta_{i,n}}{r}\right)\,dr=-\tfrac{1}{2}\beta_{i,n}\hat{\chi}_n^{(i)}\ln (\alpha_{i,n}/D) +O(\tfrac{D}{n^2}).$$
The functions $A$ and $B$ in Example \ref{Example2}, as well as the values of $\alpha_{i,n}$ and $\beta_{i,n}$, are chosen precisely so that 
$$-\beta_{i,n}\ln (\alpha_{i,n}/D)=\frac{D}{n^3(1-|q_{i,n}|^2)}=\frac{D}{\lambda}+O\left(\frac{D}{n^3}\right).$$
Given the estimates on $\chi_n^{(i)}-\chi$ of Proposition \ref{BLRcontrollemma}, the length of $L_{i,n}$ with respect to $g_{\RN,i}$ behaves as 
$$L+O\left(\tfrac{D}{n}(1+\tfrac{1}{\lambda})\right) \text{\ \ where\ \ } L=\tfrac{1}{2\lambda}\chi(1,0,0).$$
In fact, the same statement applies to the length of $L_{i,n}$ with respect to $g_n$ due to Lemma \ref{GHduh} and the approximation $g_n\approx g_{\RN,i}$ used in the proof of Lemma \ref{lastlemma?}.
Consequently, both the sequence $(L_{i,n},g_n)$ and the sequence $(L_{i,n},g_{\RN,i})$ converge in the Gromov-Hausdorff sense to the line segment $[0,L]$. For the remainder of the proof let $F_L$ denote the  $\tfrac{D}{n}(1+\tfrac{1}{\lambda})$-isometry between $(L_{i,n},g_n)$ and the line segment $[0,L]$ given by 
$$F_L(y)=\min\{L_{g_n}(\gamma_y),L\} \text{\ \ with\ \ } \gamma_y:[|y-p_{i,n}|, D/n^2]\to L_{i,n},\ \ \gamma_y(t)=p_{i,n}+t |y-p_{i,n}|.$$
Note that $F_L(z_{i,n})=0$ for $\{z_{i,n}\}=L_{i,n}\cap \U_{n,R}$. 

Finally, recall from the proof of Theorem \ref{GHduhduhduh} in Section \ref{DUH} that the inclusion mapping $\U_{n,R}\hookrightarrow B_{\geucl}(0,R)$ is an $\tfrac{\C^+[R]}{n}$-isometry between $(\U_{n,R}, g_n)$ and $(B_{\geucl}(0,R), g)$. We use this fact to show that the mapping $F:\W_{n,R}\to \W_{\infty,R}$ given by 
$$F(x)=\begin{cases}
x\times \{0\} &\text{\ \ if\ \ } x\in \U_{n,R};\\
\left((1,0,0), F_L(x)\right) &\text{\ \ if\ \ } x\in L_{i,n}\smallsetminus \U_{n,R}
\end{cases}$$
is an $\e$-isometry with $\e$ of class $\e=\tfrac{\C^+[R]}{n}(\lambda+\tfrac{1}{\lambda})$. 
Our proof is going to be complete after we prove the estimate 
$$|d_{(\W_{n,R},g_n)}(x,y)-d_\infty(F(x), F(y))|<\e$$
in the case when $x\in \U_{n,R}$ and $y\in L_{i,n}$.  Since 
$$\begin{aligned}
d_{(\W_{n,R},g_n)}(x,y)=&d_{(\U_{n,R},g_n)}(x,z_{i,n})+d_{(L_{i,n},g_n)}(z_{i,n},y)\\
d_\infty(F(x),F(y))=&d_{(B_{\geucl},g)}(x,(1,0,0))+\left|F_L(z_{i,n})-F_L(y)\right|,
\end{aligned}$$
it suffices to show that $d_g(z_{i,n}, (1,0,0))$ can be made appropriately small. This is immediate from the observation that 
$$d_g(z_{i,n}, (1,0,0))\le Cd_{\geucl}(z_{i,n}, (1,0,0))\le CD\tfrac{\lambda}{n^3}.$$ 

\subsection{Gromov-Hausdorff convergence in presence of deep wells}
At the end of Example \ref{Example2} we mentioned a possibility to alter the sample point $q_{i,n}$ so that the resulting sequence of Brill-Lindquist-Riemann sums has deep wells:
$$q_{i,n}=(1-\tfrac{1}{2n^4},0,0).$$
In that particular context we are able to repeat the argument of Section \ref{dependencyonsamplepoints} to prove that the sequence of metric spaces $(\V_{n,R,R'},g_n)$ converges in the Gromov-Hausdorff sense to the $B_{\geucl}(0,R)\times \{0\}\cup {(1,0,0)}\times [0,R']$ equipped with the $d_\infty((x,t), (y,s))=d_g(x,y)+|t-s|$. One could argue that this example was easy to produce because it only really featured one deep well. 

On the other extreme end there are examples where every ``neck" is a cylindrical end, i.e examples with $A\equiv 0$ or $B\equiv 0$. Those kinds of situations are more akin to Section \ref{GHdne} where we established non-existence of the Gromov-Hausdorff limit. Specifically, we can use ideas of Section \ref{GHdne} to show that $\mathrm{Cap}_{(\V_{n,R,R'},g_n)} (R'/2)\to \infty$ as $n\to \infty$, which in turn implies non-existence of the Gromov-Hausdorff limit.

\section{The intrinsic flat limit}\label{IFL:sec}

\subsection{A review of the intrinsic flat limit}
Hausdorff distance rests on the concept of point-wise distance between elements of two sets. On the other hand, one might hope for a weaker (pun intended!) approach where the distance between two sets is captured by some genre of ``volume" needed to transition from one set to the next. For example, we might benefit from having a concept of distance between two curves in $\mathbb{R}^3$ with shared endpoints based on surface areas of possible ``fillers". Such a concept of distance would be more tolerant of occasional ``spikes" and as such it would be far more suitable for applications to Brill-Lindquist-Riemann sums, especially in the presence of deep wells. 

One such concept of distance appears in the work of H. Whitney \cite{Whitney} in relation to what is called flat norm. Subsequently, H. Federer and W. H. Flemming in \cite{FF60} introduced the concept of integral currents as part of their framework for $k$-dimensional integration in $\mathbb{R}^n$, now the cornerstone of what we call geometric measure theory; they also broadened the concept of Whitney's flat distance to apply to integral currents. The work is technical and we shall not go into any of its details. For us here it is sufficient to know that compact oriented submanifolds of $\mathbb{R}^n$ are integral currents and that the concept of mass $\mathbf{M}$ of an integral current generalizes the concept of volume of the submanifold. Whitney's flat distance between two integral currents $M_1$ and $M_2$ is identified in \cite{FF60} as 
$$\inf_{Q,R}\{\mathbf{M}(Q)+\mathbf{M}(R)\,\big{|}\, M_1=M_2+Q+\partial R\}.$$
This is illustrated on Figure \ref{fig4}, where $M_1$ and $M_2$ are the base and the lid respectively, $R$ is the higher dimensional filler and $Q$ consists of the surface wrapping around $R$ and the two (appropriately oriented) components protruding to the right. The idea here is that if $M_1\approx M_2$ then a filler with small volume and small leftover surface area can be found. All of these concepts have since been extended from $\mathbb{R}^n$ to general metric spaces, e.g the work of \cite{AK}.

\begin{figure}
\centering
\begin{tikzpicture}[scale=0.9]
\draw[thick] (8, 1.5) to [out=105, in=-180] (9, 2.5) to [out=0, in=225] (12, 2.3);

\draw[thick] (12.5, 2.3) to [out=-10, in=225] (13.8, 2.75) to [out=-135, in=30] (13, 1.5);

\draw[thick] (11.75, 1.9) to [out=0, in=180] (12.25, 3) to [out=0, in=-180] (12.75, 1.8);

\path [fill=lightgray] (13,1.5) to [out =45, in=-90] (15.75, 2) to [out =90, in =45] (13.8, 2.75) to [out=-135, in=30] (13, 1.5);

\path [fill=lightgray] (13,0.5) to [out =30, in=-90] (14.45, -0.35) to [out =90, in =-60] (13.8,1.75) to [out=-135, in=30] (13, 0.5);

\path [fill=lightgray] (13, 1.5) to [out=-90, in =60] (13,1) to [out=-120, in =90] (13,0.5) to [out =30, in =-135] (13.8, 1.75) to [out =90, in=-120] (13.8, 2.25) to [out =60, in =-90] (13.8, 2.75) to [out =-135, in =30] (13, 1.5);

\draw [thick] (13,1.5) to [out =45, in=-90] (15.75, 2) to [out =90, in =45] (13.8,2.75) to [out=-135, in=30] (13, 1.5);

\draw [thick] (13, 1.5) to [out=-90, in =60] (13,1) to [out=-120, in =90] (13,0.5);

\draw [thick] (13,0.5) to [out =30, in=-90] (14.45, -0.35) to [out =90, in =-60] (13.8,1.75) to [out=-135, in=30] (13, 0.5);

\draw[dashed] (13.8, 1.75) to [out =90, in=-120] (13.8, 2.25) to [out =60, in =-90] (13.8, 2.75);

\draw[thick] (13.8, 2.75) to [out =-135, in =30] (13, 1.5);

\path [fill=lightgray] (8,1.5) to [out=0, in=-150] (8.5,1.75) to [out=30, in=-180] (13, 1.5) to [out=-90, in =60] (13,1) to [out=-120, in =90] (13,0.5) to [out =-180, in=30] (8.5, 0.75) to [out=-150, in=0] (8, 0.5)  to [out =105, in=-75] (8, 1.5);

\draw [thick] (8,1.5) to [out=0, in=-150] (8.5,1.75) to [out=30, in=-180] (13, 1.5) to [out=-90, in =60] (13,1) to [out=-120, in =90] (13,0.5) to [out =-180, in=30] (8.5, 0.75) to [out=-150, in=0] (8, 0.5)  to [out =105, in=-75] (8, 1.5);

\draw[dashed] (8, 0.5) to [out=95, in=-180] (9, 1.5) to [out=0, in=225] (13.8, 1.75);

\draw[thick] (13.8,1.75) to [out=-135, in=30] (13, 0.5);

\draw (7.5,0.5) to (7.5,1.75);

\draw (7.3, 0.7) to (7.5, 0.5) to (7.7, 0.7);

\draw (7.3, 1.55) to (7.5, 1.75) to (7.7, 1.55);

\node[left] at (7.4, 1) {$R$};

\node[right] at (13.75, 2.25) {$Q$};

\node[right] at (13.6, 0.75) {$Q$};

\node[below] at (9.2, 2.62) {$M_2$};
\node[below] at (9.4, 1.18) {$M_1$};
\end{tikzpicture}
\caption{Illustration of the concept of flat distance}\label{fig4}
\end{figure}
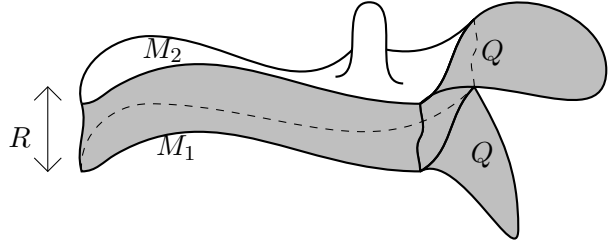

What Gromov-Hausdorff distance is to Hausdorff distance the concept of intrinsic flat distance is to flat distance. The intrinsic flat distance between two oriented Riemannian
manifolds with boundary was introduced in the joint work of C. Sormani and S. Wenger \cite{SW-JDG}. This distance is measured by first viewing each of the two manifolds as an integral current, pushing forward these integral currents into a common complete metric space via distance preserving maps, and then measuring the flat distance between the two push forwards; one takes the infimum over all distance preserving maps into all complete metric spaces. In practice it is often possible to estimate the intrinsic flat distance by only using notions from Riemannian geometry. A particularly easy-to-use estimate was proven by S. Lakzian and C. Sormani in \cite{LS13}. For the convenience of the reader we state the theorem of Lakzian and Sormani in full.

\begin{theorem} \label{thm-subdiffeo} 
Suppose $(M_1,g_1)$ and $(M_2,g_2)$ are oriented
precompact Riemannian manifolds with diffeomorphic subregions $U_i \subset M_i$.
Identifying $U_1=U_2=U$ assume that
on $W$ we have
\begin{equation}\label{thm-subdiffeo-1}
g_1 \le (1+\varepsilon)^2 g_2 \textrm{ and }
g_2 \le (1+\varepsilon)^2 g_1
\end{equation}
Then
\begin{eqnarray*}
d_{\mathcal{IF}}(M_1, M_2) &\le&
\left(2\bar{h} + a\right) \Big(
\Vol_{g_1}(U)+\Vol_{g_2}(U)+\Area_{g_1}(\partial U)+\Area_{g_2}(\partial U)\Big)\\
&&+\Vol_{g_1}(M_1\setminus U)+\Vol_{g_2}(M_2\setminus U).
\end {eqnarray*}
where
\begin{equation} \label{thm-subdiffeo-3}
a>\frac{\arccos(1+\varepsilon)^{-1} }{\pi}D,
\end{equation}
\begin{equation} \label{thm-subdiffeo-5}
\bar{h}= \max\{\sqrt{2\lambda D },  \sqrt{\varepsilon^2 + 2\varepsilon} \; D \}
\end{equation}
where\footnote{The use of $1$ in the definition of $\diam$ is meant to be accompanied with units of length.}
\begin{equation} \label{DU}
D= \max\{1,\diam(M_1), \diam(M_2)\},
\end{equation}
and 
\begin{equation} \label{lambda}
\lambda=\sup_{x,y \in U}
|d_{M_1}(x,y)-d_{M_2}(x,y)|.
\end{equation}
\end{theorem}

\subsection{Estimates on the intrinsic flat distance}

Estimation of $\lambda$ is probably the most delicate step in any application of Theorem \ref{thm-subdiffeo}. Thankfully, for our purposes here we already addressed this parameter!  This was done in Lemma \ref{distancecomparisonlemma} of Section \ref{GH:sec}. Keeping with the spirit of Lemma \ref{distancecomparisonlemma} we continue by developing a general result about intrinsic flat distance among ``perforated spaces"; we then apply it to sequences of Brill-Lindquist-Riemann sums.

\begin{proposition} \label{IFL-thm}
Let $g$ be a metric which is equivalent to $\geucl$ over $\mathbb{R}^3$, i.e a metric such that $c^{-2}\geucl\le g\le c^2\geucl$ over $\mathbb{R}^3$ for some constant $c$. Let $\V\subseteq \mathbb{R}^3$ be a connected open subset and let $g_\V$ be a metric on $\V$. Consider a finite union 
$$\P=\bigcup \bar{B}_{\geucl}(p_i, r_i) \text{\ \ with\ \ } r_\P\le \tfrac{1}{4}\sigma_\P,$$  
where 
$$r_{\P}:=\max_{i}r_i \text{\ \ and\ \ }\sigma_\P:=\min_{i,j}|p_i-p_j|.$$ 
Finally, assume that there is some fixed $R\gg 1$ such that 
$$\V\smallsetminus \P= B_{\geucl}(0,R)\smallsetminus \P$$
along with $\P\subseteq B_{\geucl}(0,R)$. Then  
$$d_{\mathcal{IF}}((\V,g_\V), (B_{\geucl}(0,R),g))\le \e+\e\,\mathrm{Area}_{\geucl}(\partial \P)\cdot (R+\mathrm{diam}(\V,g_\V))+\mathrm{Vol}_{g_\V}(\V\cap \P),$$
where $\e$ denotes a quantity which can be made arbitrarily small by making 
$$r_\P,\ \ \frac{r_\P}{\sigma_\P}, \text{\ \ and\ \ }  \|g_\V-g\|_{L^\infty(\V\smallsetminus\P,\geucl)}$$
appropriately small. 
\end{proposition}

\begin{proof}
Throughout this proof we let $\e$ denote quantities which remain controlled by $r_\P$, $\frac{r_\P}{\sigma_\P}$, and  $\|g_\V-g\|_{L^\infty(\V\smallsetminus\P,\geucl)}$ in the sense spelled out in the statement of the lemma; note that by doing so we permit the exact expression and units for $\e$ to change from term to term and line to line. 

The overall plan is to apply Theorem \ref{thm-subdiffeo} to 
$$\U:=\V\smallsetminus \P=B_{\geucl}(0,R)\smallsetminus \P\subseteq \V\cap B_{\geucl}(0,R).$$  
Without any loss of generality we assume throughout this proof that   
\begin{equation}\label{g-bds-2}
(1+\e)^{-2}g\le g_{\V}\le (1+\e)^2 g \text{\ \ and\ \ }
c^{-2}\,\geucl\le g,g_\V \le c^2\, \geucl \text{\ \ over\ \ }\U.
\end{equation}

We start by controlling the parameter $\lambda$ of Theorem \ref{thm-subdiffeo}. Let $x,y\in\U$. To compare $d_{(B_{\geucl},g)}(x,y)=d_{(\V\cup\P,g)}(x,y)$ to $d_{(\V,g_\V)}(x,y)$ we take the following steps:

\begin{itemize}
\item We first compare $d_{(\V\smallsetminus \P, g)} (x,y)$ to $d_{(\V\cup\P,g)}(x,y)$ through an application of Lemma \ref{distancecomparisonlemma}: equivalency of $g$ and $\geucl$ yields a uniform bound 
$$d_{(\V\cup\P,g)}(x,y)\le 2cR$$
which further implies
$$0\le d_{(\V\smallsetminus \P, g)} (x,y)- d_{(\V\cup\P,g)}(x,y)\le 4\pi c^3\,\frac{r_\P}{\sigma_\P}\,R+ \pi c\, r_\P\le \e.$$
\medbreak
\item Next, we examine the proximity of $d_{(\V\smallsetminus\P, g_\V)}(x,y)$ to $d_{(\V\smallsetminus\P, g)}(x,y)$ as a consequence of the proximity $g_\V\approx g$ stated in \eqref{g-bds-2}. Specifically, taking infimums of
$$(1-\e)L_{g}(\gamma)\le L_{g_\V}(\gamma)\le (1+\e)L_{g}(\gamma)$$
over curves $\gamma$ in $\V\smallsetminus\P$ and using the uniform bound 
$$d_{(\V\smallsetminus \P, g)} (x,y) \le d_{(\V\cup\P,g)}(x,y)+\e\le 2c\,R+\e$$
proves that
$$
\left|d_{(\V\smallsetminus\P, g_\V)}(x,y)-d_{(\V\smallsetminus\P, g)}(x,y)\right|\le \e
$$
holds uniformly for all $x,y\in \U$. 
\medbreak
\item Finally, we relate $d_{(\V\smallsetminus \P, g_\V)} (x,y)$ to $d_{(\V,g_\V)}(x,y)$ through Lemma \ref{distancecomparisonlemma}:
$$\begin{aligned}
0\le d_{(\V\smallsetminus\P, g_\V)} (x,y)- d_{(\V,g_\V)}(x,y)\le &2\pi c^2\,\frac{r_\P}{\sigma_\P}\,d_{(\V,g_\V)}(x,y)+ \pi c\, r_\P\\
\le &\e+\e\,\diam(\V,g_\V).
\end{aligned}$$
\end{itemize}
 
Combining all of the above proves that 
$$\lambda:=\sup_{x,y\in \U} |d_{(\V, g_\V)}(x,y)-d_{(B_{\geucl}(0,R), g)}(x,y)|\le \e+\e\,\diam(\V,g_\V).$$

We continue by addressing the parameters $D$, $\bar{h}$ and $a$ of Theorem \ref{thm-subdiffeo}. Due to equivalence \eqref{g-bds-2} we have 
$$\mathrm{diam}(B_{\geucl}(0,R), g)\le 2cR.$$
Assuming $R\gg 1$ we see that $D=O(R+\diam(\V,g_\V))$ so that 
$$2\bar{h}+a=\e+\e\,\diam(\V,g_\V).$$
 
It remains to address area and volume estimates needed for applications of Theorem \ref{thm-subdiffeo}.
\begin{itemize}
\item {\it Boundedness of $\Vol_{g_\V}(\U)$ and $\Vol_g(\U)$.} The bound $\Vol_g(\U)\le (4/3)c^3\pi R^3$ is immediate from $\U\subseteq B_{\geucl}(0, R)$. The bound on $\Vol_{g_\V}(\U)$ is then a consequence of
$$|\Vol_{g_\V}(\U)-\Vol_g(\U)|\le \e,$$
which in turn holds because $\U=\V\smallsetminus \P$.
\medbreak
\item {\it Boundedness of $\Area_{g_\V}(\partial \U)$ and $\Area_{g}(\partial \U)$.} We once again have 
$$|\Area_{g_\V}(\partial \U)- \Area_{g}(\partial \U)|<\e$$ 
and so it suffices to provide a bound on $\Area_{g}(\partial \U)$. To that end observe that 
$$\partial \U\subseteq S_{\geucl}(0,R) \cup \partial \P$$
and that 
$$\Area_{g}(\partial \P)\le c^2\Area_{\geucl}(\partial\P).$$
\medbreak
\item {\it Smallness of $\Vol_{g_\V}(\V\smallsetminus \U)$ and $\Vol_g(B_{\geucl}(0,R)\smallsetminus \U)$.} To address these volumes observe that 
$$\V\smallsetminus \U\subseteq \V\cap \P \text{\ \ and\ \ } B_{\geucl}(0,R)\smallsetminus \U\subseteq \P$$
and that 
$$\Vol_g(\P)\le c^3 \Vol_{\geucl}(\P)\le \tfrac{1}{3} c^3 r_\P\Area_{\geucl}(\partial \P)\le\e\cdot \Area_{\geucl}(\partial \P).$$
\end{itemize}
The claim of our lemma is now immediate from Theorem \ref{thm-subdiffeo}. 
\end{proof}

\subsection{The intrinsic flat limit of sequences of Brill-Lindquist-Riemann sums}

To prove Theorem \ref{BIGMAMMA} we simply apply Proposition \ref{IFL-thm}. In the setting where there are no deep wells we use $\V_n=\V_{n,R}$ and 
$$r_n=\tfrac{D}{n^2}\to 0,\ \ \tfrac{r_n}{\sigma_n}=\tfrac{1}{n}\to 0,\ \ \mathrm{Area}_{\geucl}(\partial \P_n)=32\pi n^3(\tfrac{D}{n^2})^2=32\pi \tfrac{D^2}{n}\to 0.$$
That the Theorem \ref{IFL-thm} applies to $\V_n=\V_{n,R}$ is a consequence of \eqref{towardsIFL-2}.
Recall that 
$$\|g_n-g\|_{L^\infty(\V_{n,R}\smallsetminus \P_n)}\to 0$$
due to Proposition \ref{BLRcontrollemma}. 
By the very assumption of there being no deep wells we have that 
$\mathrm{diam}(\V_{n,R}, g_n)=O(1)$
while 
$\Vol_{g_n}(\V_{n,R}\cap \P_n)\to 0$
is a direct consequence of \eqref{volest-IFL1}.
The proof in the case of deep wells is exactly the same except that we use $\V_n=\V_{n,R,R'}$. Theorem \ref{IFL-thm} applies in this situation because of Lemma \ref{deepwell}. The diameter estimate is replaced by $\mathrm{diam}(\V_{n,R,R'}, g_n)\le 2R'$ while the volume estimate is obtained as a consequence of \eqref{volest-IFL2}.

\bibliographystyle{plain}
\bibliography{2014}

\end{document}